\documentclass{amsart}
\usepackage{amssymb}
\usepackage{amsfonts}
\usepackage{amstext}
\usepackage{algorithmic}
\usepackage{algorithm}
\usepackage{graphicx}
\usepackage{epstopdf}
\usepackage[all]{xy}

\numberwithin{equation}{section}
\newtheorem{theorem}{Theorem}[section]
\newtheorem{lemma}[theorem]{Lemma}
\newtheorem{proposition}[theorem]{Proposition}
\newtheorem{corollary}[theorem]{Corollary}

\newtheorem{conjecture}[theorem]{Conjecture}
\newtheorem{question}[theorem]{Question}
\newtheorem{questions}[theorem]{Questions}
\newtheorem{fact}[theorem]{Fact}
\newtheorem{facts}[theorem]{Facts}
\theoremstyle{definition}
\newtheorem{definition}[theorem]{Definition}
\newtheorem{example}[theorem]{Example}

\theoremstyle{remark}
\newtheorem{remark}[theorem]{\bf{Remark}}

\newtheorem{convention}[theorem]{\bf{Convention}}

\newtheorem{construction}[theorem]{\bf{Construction}}
\newtheorem{openquestions}[theorem]{\bf{Open questions}}
\newtheorem{notation}[theorem]{\bf{Notation}}

\newtheorem*{proofproposition}{Proof of the proposition}

\newcommand{\Hom}{\mathop{\mathrm{Hom}}}

\newcommand{\Nor}{{\rm{Nor}}}

\newcommand{\Acal}{{\mathcal{A}}}

\newcommand{\X}{{\mathbb{X}}}

\newcommand{\Dcal}{{\mathcal{D}}}

\newcommand{\Gcal}{{\mathcal{G}}}

\newcommand{\Lcal}{{\mathcal{L}}}

\newcommand{\Ocal}{{\mathcal{O}}}

\newcommand{\Rcal}{{\mathcal{R}}}

\newcommand{\Xcal}{{\mathcal{X}}}

\newcommand{\la}{{\triangleright}}
\newcommand{\ra}{{\triangleleft}}

\renewcommand{\ker}{\mathop{\mathrm{ker}}}

\newcommand{\Ker}{{\rm{Ker}}}

\newcommand{\tens}{\otimes}

\renewcommand{\la}{{\triangleright}}
\renewcommand{\ra}{{\triangleleft}}

\newcommand{\id}{\mathord{\mathrm{id}}}
\newcommand{\mpl}{\mathop{\mathrm{mpl}}}
\newcommand{\sol}{\mathop{\mathrm{sl}}} 
\newcommand{\Sym}{\mathop{\mathrm{Sym}}}
\newcommand{\Aut}{\mathop{\mathrm{Aut}}}
\newcommand{\Cay}{\mathop{\mathrm{Cay}}}
\newcommand{\Ret}{\mathop{\mathrm{Ret}}}
\newcommand{\stu}{\mathbin{\natural}}
\renewcommand{\wr}{\mathbin{\mathrm{wr}}}

\begin{document}

\title[Multipermutation solutions of the YBE ]{{Multipermutation solutions of the Yang--Baxter
equation} \keywords{Yang--Baxter equation, semigroup, graph, permutation group}
\subjclass[2000]{Primary 81R50, 16W50, 16S36}}
\thanks{The first author is partially supported by Isaac Newton Institute, UK, The Abdus Salam International Centre for Theoretical Physics (ICTP), and by
Grant MI 1503/2005 of the Bulgarian National Science Fund of the
Ministry of Education and Science}

\author{Tatiana Gateva-Ivanova and Peter Cameron}
\address{TGI: Institute of Mathematics and Informatics\\
Bulgarian Academy of Sciences\\
Sofia 1113, Bulgaria\\
P.C:  Queen Mary, University of London\\
School of Mathematics, Mile End Rd, London E1 4NS, UK}
\email{tatianagateva@yahoo.com, tatyana@aubg.bg,
p.j.cameron@qmul.ac.uk}

\date{\today}

\begin{abstract}
Set-theoretic solutions of the Yang--Baxter equation form a meeting-ground
of mathematical physics, algebra and combinatorics. Such a solution consists of a set $X$
and a function $r:X\times X\to X\times X$ which satisfies the braid relation.
We examine solutions here
mainly from the point of view of finite permutation groups: a solution
gives rise to a map from $X$ to the symmetric group $\Sym(X)$ on $X$
satisfying certain conditions.

Our results include many new constructions based on strong twisted union
and wreath product, with an investigation of retracts and the multipermutation
level and the solvable length of the groups defined by the solutions; and
new results about decompositions and factorisations of the groups defined
by invariant subsets of the solution.
\end{abstract}


\maketitle

\setcounter{tocdepth}{1}
\tableofcontents

\section{Introduction }

Let $V$ be a vector space over a field $k$. It is well-known that the
``Yang--Baxter equations'' on a linear map $R:V\tens V\to V\tens V$,
the equation
\[ R_{12}R_{23}R_{12}=R_{23}R_{12}R_{23}\]
(where $R_{i,j}$ denotes $R$ acting in the $i,j$ place in $V\tens V\tens V$),
give rise to a linear representation of the braid group on tensor powers of
$V$. When $R^2=\id$ one says that the solution is involutive, and in this case
one has a representation of the symmetric group on tensor powers.

A particularly nice class of solutions is provided by set-theoretic solutions
where $X$ is a set and $r:X\times X\to X\times X$ obeys similar relations on
$X\times X\times X$. Of course, each such solution extends linearly to $V=kX$
with matrices in this natural basis having only entries from 0,1 and many
other nice properties.

Associated to each set-theoretic solution are several algebraic constructions:
the semigroup $S(X,r),$ the group $G(x,r)$,  the semigroup algebra
$k S(X,r)= k_R[V]$ generated by $X$ with relations $xy=\cdot r(x,y)$
(where $\cdot$ denotes product in the free semigroup, resp free group)
and the permutation group $\Gcal(X,r) \subset \Sym (X)$ defined by the
corresponding left translations $y\mapsto {}^xy$ for $x\in X$, where
$r(x,y)=({}^xy,y^x)$ (under assumptions which will be given later).

In this paper we study the special case  when $(X,r)$ is a
\emph{square-free solution}, a square-free symmetric set of arbitrary
cardinality. Our special interest is in the retractability of such solutions.
We study multipermutation solutions and find close relation between the
multipermutation level of such a solution and the properties of the
associated algebraic objects $\Gcal(X,r)$, $G(X,r)$, and $S(X,r)$.
A feature of our approach is to give prominence to the group $\Gcal(X,r)$.

We introduce and study various constructions of solutions such as strong
twisted unions of solutions, doubling of solutions, wreath product of
solutions, and various other constructions, and study the multipermutation
level of the new solutions and properties such as solvable length of their
associated algebraic structures.

We now describe the contents of the paper in greater detail.

In Section \ref{Preliminaries} we recall basic definitions and basic results,
and give a general description of the permutation groups $H \subseteq\Sym(X)$
which can serve as a YB permutation group of some square-free  solution
$(X,r)$, see Proposition \ref{permYBgroupprop}.

In Section \ref{strong twisted unions}
we recall the definitions and basic properties of
homomorphisms and automorphisms of solutions,
we give a general construction, the \emph{strong twisted union of a
finite number of solutions},
and describe strong twisted unions in terms of split maps, see
Proposition \ref{splitmapprop}.

In Section \ref{Decompositions_Section} we study various decompositions of
square-free solutions $(X,r)$ into disjoint unions of a finite number of
$r$-invariant subsets and the corresponding factorisation of
$S(X,r)$, $G(X,r)$, and $\Gcal(X,r)$, see Theorems \ref{MAINDECOMPOSITIONTHM},
\ref{decompositiontheorem} and \ref{decompositionretractclasses}.
We make essential use of the \emph{matched pairs approach to solutions}
(in the most general setting), developed in \cite{TSh08}.

Section \ref{Multipermutation solutions of low levels} continues the study
of multipermutation solutions of low levels which was initiated in
\cite{TSh07}, and deepened in \cite{TSh0806} with detailed description of
solutions of multipermutation level 2.
Using matched pair approach and results from \cite{TSh08}  we show in
Proposition \ref{mpl2braidmonoidprop}
that when $(X,r)$ has multipermutation level 2,  the \emph{associated
braided monoid} $(S,r_S)$  and the \emph{associated braided group} $(G,r_G)$
are  symmetric sets which inherit some nice combinatorial conditions such as
the cyclic condition and condition \textbf{lri} (seeDefinition \ref{lri&cl})
but they are not square-free. Furthermore, $S$ (respectively $G$)  acts on
itself as automorphisms. Proposition \ref{mpl2permgroups} characterises the
permutation groups $H\subset \Sym(X)$ which define (via the left action)
square-free solutions $(X,r)$ with multipermutation level $2$. As a corollary
we obtain that every  finite abelian group $H$ is isomorphic to the
permutation group $\Gcal(X,r)$ of some  square-free solution with
multipermutation level 2.

We characterise solutions with multipermutation level 3 in
Proposition \ref{mpl3prop2}, and Corollary \ref{mpl3_Cor1}. We show that, for
$(X,r)$ of arbitrary cardinality, but having only a finite number of
$G$-orbits, each solution with multipermutation level $3$  decomposes as a
strong twisted union of a finite  number of solutions with multipermutation
level $\leq 2$, and the permutation group $\Gcal$ decomposes as a product
of abelian subgroups, see Propositions \ref{theoremC}, and \ref{mpl3prop2}.

One can  say a surprising amount about solutions $(X,r)$ for which
$\Gcal(X,r)$ is abelian: this is the theme of Section
\ref{abelian permutation group}.
In Subsection \ref{Computations with actions } we develop a technique for
computation with actions in the case of abelian $\Gcal$.

This technique and Theorem \ref{mplmtheorem}  are used to prove the main
results of the section: Theorems \ref{theoremA} and \ref{theoremB}.
We assume $(X, r)$ is a square-free solution of arbitrary  cardinality,
$\Gcal=\Gcal(X,r)$ is abelian, and $X$ has finite number $t$ of
$\Gcal$-orbits. Theorems \ref{theoremA} shows that every such a solution
is a multipermutation solution of level $\mpl X \leq t,$
and $X$ decomposes as a strong twisted union of its orbits. Furthermore,
each orbit $X_i$ is itself a trivial solution (so $mpl X_i = 1$)). This
confirms two  conjectures of the first author (only for the case of abelian
$\Gcal$, of course), see \cite{T04} Conjecture I, Conjecture II, formulated
for finite square-free solutions (see also Conjecture \ref{conj1}).

Theorem \ref{theoremB} proves that, under the assumption that $\Gcal(X,r)$ is
abelian, a strong twisted union $X = X_1 \stu X_2$ of two multipermutation
solutions is itself a multipermutation solution with
$\mpl X \leq \mpl X_1 + \mpl X_2$. In the general case, the question whether
a strong twisted union $X = X_1 \stu X_2$
of multipermutation solutions $X_1, X_2$ is also a multipermutation solution
remains open.

A resent result of Ced\'{o}, Jespers, and
Okni\'{n}ski, \cite{Eric2009}, shows that
each \emph{finite} square-free solution $(X,r)$ with $\Gcal$ abelian  is
retractable. The proof is combinatorial (and different from ours) and relies
strongly on the finiteness of $X$; there is no estimate of $\mpl X$.
Our proof uses a general technique which is applicable to  solutions with
arbitrary  cardinality of $X$
(but a finite number $t$ of  $G$-orbits). In fact we show that $t$ is an
upper bound for the multipermutation level $\mpl X$.

Section \ref{generalmplsection} studies the general case of multipermutation
square-free solutions.
In Subsection \ref{General results} we recall  basic notions and facts
from \cite{T04}, and use them to develop
a basic technique for dealing with retracts and retract classes.
Theorem \ref{mplmtheorem} gives an explicit identity
in terms of action necessary and sufficient for $\mpl X = m$.
This identity plays an essential role in the paper.

Subsection \ref{The groups subsection} contains some of the important
results of this paper.
We study the groups $G= G(X.r)$ and $\Gcal =\Gcal(X,r)$ of multipermutation
solutions. We show that if $(X,r)$ is a square-free multipermutation solution
(of arbitrary cardinality) the groups $G$ and $\Gcal$ are solvable.
This was known for finite symmetric sets, see \cite{ESS},
(see also \cite{T04} for finite square-free solutions),
but no information about the solvable length of $G$ was known.

Lemma \ref{rethomlemma} and induction on $m= \mpl X$ allow us not only to
prove solvability (without assuming $X$ or $\Gcal$ finite), but also to find
an upper bound for the solvable lengths: $\sol G \leq m$ and
$\sol \Gcal \leq m-1$.
The results of Section \ref{infinitesolutions} show that these upper bounds
are attained.
Theorem \ref{slG=slGcal+1thm} verifies that, whenever $(X,r)$ is a
square-free solution of finite order, then  $\sol (G) = \sol (\Gcal) + 1$.

In Section \ref{Wreath products} we define the notion of
\emph{wreath product of solutions}, by analogy with the wreath products of
permutation groups. Theorem \ref{wreathconstructionth}
shows that wreath product $(Z,r) = (X_0,r_{X_0})\wr(Y,r_Y)$ is a
square-free solution, with
$\Gcal(Z, r) = \Gcal(X_0,r_{X_0})\wr\Gcal(Y,r_Y)$. Furthermore, whenever
$(X_0,r_{X_0})$ and $(Y,r_Y)$ are multipermutation solutions, one has
$\mpl (Z,r) = \mpl (X_0,r_{X_0})+\mpl (Y,r_Y) - 1$.

In  Section \ref{infinitesolutions} we construct   an interesting
sequence of explicitly defined solutions $(X_m, r_m),$  $m = 0, 1,
2, \ldots$, such that $\mpl(X_m)= m$,
$\Ret(X_{m+1}, r_{m+1}) \simeq (X_m, r_m)$,  $m = \sol G(X, r_m) = \sol
\Gcal(X, r_m) + 1$: see Definition \ref{beautifulconstructiondef}
and Theorem \ref{beautifulconstruction}.

\section{Preliminaries on set-theoretic solutions}
\label{Preliminaries}

There are many  works on set-theoretic solutions and related
structures, of which a relevant selection for the interested
reader is \cite{W, TM,  ESS, T04, T04s,  TSh07, TSh08, TSh0806,
Lu, Rump, Takeuchi, V}. In this section we recall basic notions
and results which will be used in the paper.
 We shall use the terminology,  notation and some results from
\cite{T04,  TSh07, TSh08, TSh0806}.

 \begin{definition}
Let $X$ be a nonempty set (not necessarily finite) and let
$r: X\times X \longrightarrow X\times X$ be a bijective map. We refer
to it as a \emph{quadratic set}, and denote it by  $(X,r)$. The image
of $(x,y)$ under $r$ is presented as
\begin{equation}
\label{r} r(x,y)=({}^xy,x^{y}).
\end{equation}
The formula (\ref{r}) defines a ``left action'' $\Lcal: X\times X
\longrightarrow X,$ and a ``right action'' $\Rcal: X\times X
\longrightarrow X,$ on $X$ as:
\begin{equation}
\label{LcalRcal} \Lcal_x(y)={}^xy, \quad \Rcal_y(x)= x^{y},
\end{equation}
for all $ x, y \in X.$ The map  $r$ is \emph{nondegenerate}, if
the maps $\Lcal_x$ and $\Rcal_x$ are bijective for each $x\in X$.
In this paper we shall consider only  the case where $r$
is nondegenerate. As a notational tool,
we  shall often identify
the sets $X\times X$ and $X^2,$ the set of all monomials of length
two in the free semigroup $\langle X\rangle.$
\end{definition}

\begin{definition}
\label{defvariousr}
\begin{enumerate}
 \item
$r$ is \emph{square-free} if $r(x,x)=(x,x)$ for all $x\in X.$
\item \label{YBE} $r$ is \emph{a set-theoretic solution of the
Yang--Baxter equation} or, shortly \emph{a solution}
 (YBE) if  the braid relation
\[r^{12}r^{23}r^{12} = r^{23}r^{12}r^{23}\]
holds in $X\times X\times X,$ where the two bijective maps
$r^{ii+1}: X^3 \longrightarrow X^3$, $ 1 \leq i \leq 2$ are
defined as  $r^{12} = r\times\id_X$, and $r^{23}=\id_X\times r$. In
this case we  refer to  $(X,r)$ also as   \emph{a braided set}.
\item A braided set $(X,r)$ with $r$ involutive is called \emph{a
symmetric set}.
\end{enumerate}
\end{definition}

\begin{convention}
By \emph{square-free solution} we mean a nondegenerate square-free
symmetric set $(X,r)$, where $X$  is a set of arbitrary
cardinality. Alternative definitions are given in terms of the
left action, see Lemma \ref{alternativedef} and Corollary
\ref{alternativedefc}.
\end{convention}

To each quadratic set $(X,r)$  we associate canonical algebraic
objects  generated by $X$ and with quadratic defining relations $\Re
=\Re(r)$ defined by
\begin{equation}
\label{defrelations} xy=zt \in \Re(r),\quad
  \text{whenever}\quad r(x,y) = (z,t).
\end{equation}

\begin{definition}
\label{associatedobjects}  Let $(X,r)$ be a quadratic set.

(i) The unital semigroup $ S =S(X, r) = \langle X; \Re(r)
\rangle$, with a set of generators $X$ and a set of defining
relations $ \Re(r),$ is called \emph{the semigroup associated with
$(X, r)$}.

(ii) The \emph{group $G=G(X, r)$ associated with} $(X, r)$ is
defined as
$G=G(X, r)={}_{gr} \langle X; \Re (r) \rangle$

(iii) For arbitrary fixed field $k$, \emph{the $k$-algebra
associated with} $(X ,r)$ is defined as $\Acal(k,X,r) = k\langle X ;
\Re(r) \rangle.$ ($\Acal(k,X,r)$ is isomorphic to the monoidal
algebra $kS(X,r)$).

(iv) To each nondegenerate braided set $(X,r)$ we also associate a
permutation group, called \emph{the group of left action} and
denoted $\Gcal= \Gcal(X,r)$, see Definition \ref{Gcaldef}.

 If $(X,r)$ is a solution, then $S(X,r)$, resp. $G(X,r)$,
resp. $\Gcal(X,r)$, resp. $\Acal(k,X,r)$ is called the
{\em{Yang--Baxter semigroup}}, resp. the {\em{Yang--Baxter group}},
, resp. the {\em{Yang--Baxter algebra}}, resp. the
{\em{(Yang--Baxter) permutation group}}  or shortly \emph{the YB
permutation group}, associated to $(X,r)$
\end{definition}

The YB permutation group   $\Gcal(X,r)$ will be of particular
importance in this paper.

\begin{example}
\label{trivialsolex} For arbitrary nonempty set $X$, denote by
$\tau_X = \tau$ the flip map $\tau(x,y)= (y,x)$ for all $x,y \in
X.$ Then  $(X, \tau)$ is a solution called \emph{the trivial
solution}.  It is clear that $(X,r)$ is the trivial solution
if and only if ${}^xy =y$, and $x^{y} = x,$ for all $x,y \in X,$ or
equivalently $\Lcal_x= \id_X =\Rcal_x$ for all $x\in X$. In this
case $S(X,r)$ is the free abelian monoid, $G(X,r)$ is the free
abelian group, $\Acal(k,X,r)$ the algebra of commutative
polynomials in $X$, and $\Gcal(X,r) = \{\id_X\}$ is the trivial group.
\end{example}

\begin{remark} \label{ybe}
Suppose  $(X,r)$ is a nondegenerate quadratic set. It is well
known, see for example \cite{TSh08}, that $(X,r)$   is a braided
set (i.e. $r$ obeys the YBE)
if and only if the following conditions hold
\[
\begin{array}{lclc}
 {\bf l1:}\quad& {}^x{({}^yz)}={}^{{}^xy}{({}^{x^y}{z})},
 \quad\quad\quad
 & {\bf r1:}\quad&
{(x^y)}^z=(x^{{}^yz})^{y^z},
\end{array}\]
 \[ {\rm\bf lr3:} \quad
{({}^xy)}^{({}^{x^y}{(z)})} \ = \ {}^{(x^{{}^yz})}{(y^z)},\]
 for all $x,y,z \in X$.

Clearly, conditions {\bf l1} imply that, for each nondegenerate
braided set $(X,r)$, the assignment $x \longrightarrow \Lcal_x$ for
$x\in X$ extends canonically to a group homomorphism
\begin{equation}
\label{Lcal} \Lcal:G(X,r) \longrightarrow \Sym(X),
\end{equation}
which defines  the \emph{canonical left action} of $G(X,r)$  on the set  $X$.
Analogously {\bf r1} gives the \emph{canonical right action} of
$G(X,r)$  on~$X$.
\end{remark}

\begin{definition} \cite{TSh08}
\label{Gcaldef}
Let $(X,r)$ be a nondegenerate braided set,
$\Lcal:G(X,r) \longrightarrow \Sym(X)$ be  the canonical group
homomorphism defined via the left action. The image
$\Lcal(G(X,r))$ is denoted by $\Gcal(X,r).$ We  call it \emph{the
(permutation) group of left action}.
\end{definition}

\begin{remark}
\label{Gcal=LcalS}
If $X$ is a finite set, then $\Gcal=\Lcal(S(X,r))$. Indeed,
for a finite group, generating sets as group and as semigroup
coincide; for any element $g$ of the group has finite order, say $m$,
and so its inverse $g^{-1}$ can be expressed as a positive power $g^{m-1}$.
\end{remark}

The following conditions were introduced and studied in \cite{T04,
TSh07, TSh08}:

\begin{definition} \cite{TSh08}
\label{lri&cl}
Let $(X,r)$ be a quadratic set.
\begin{enumerate}
\item  $(X,r)$ is called {\em cyclic} if the
following conditions are satisfied
\[\begin{array}{lclc}
 {\rm\bf cl1:}\quad&  {}^{y^x}x= {}^yx \quad\text{for all}\; x,y \in
 X;
 \quad&{\rm\bf cr1:}\quad &x^{{}^xy}= x^y, \quad\text{for all}\; x,y \in
X;\\
 {\rm\bf cl2:}\quad
  &{}^{{}^xy}x= {}^yx,
\quad\text{for all}\; x,y \in X; \quad & {\rm\bf cr2:}\quad
&x^{y^x}= x^y \quad\text{for all}\; x,y \in X.
\end{array}\]
 We refer to these
conditions as {\em cyclic conditions}.
\item Condition \textbf{lri} is defined as
 \[ \textbf{lri:}
\quad ({}^xy)^x= y={}^x{(y^x)} \;\text{for all} \quad
x,y \in X.\]
In other words \textbf{lri} holds if and only if
$(X,r)$ is nondegenerate and $\Rcal_x=\Lcal_x^{-1}$ and $\Lcal_x =
\Rcal_x^{-1}$
\end{enumerate}
\end{definition}

\subsection{Square-free solutions}

In this paper the class of \emph{square-free solutions}  will be
of special interest. We now introduce these.

In the case when  $(X,r)$ is a square-free solution \emph{of
finite order} $|X|= n >2$, the algebras $\Acal(X,r)$
are \emph{binomial skew polynomial rings}, see \cite{T04,T04s}, which
provided new classes of Noetherian rings \cite{T94,T96},
Gorentstein (Artin--Schelter regular) rings
\cite{T96Preprint,T00,T04s} and so forth.
 Artin--Schelter regular
rings were introduced in \cite{AS} and are of particular interest.
The algebras $\Acal(X,r)$ are similar in spirit to the quadratic
algebras associated to linear solutions, particularly studied in
\cite{Man}, but have their own remarkable properties. The
semigroups $S(X,r)$ were studied particularly in \cite{TSh08} with
a systematic theory of `exponentiation' from the set to the
semigroup by means of the `actions' $\Lcal_x,\Rcal_x$ (which in
the process become a matched pair of semigroup actions) somewhat
in analogy with the Lie theoretic exponentiation in \cite{Ma:mat}.

We shall recall some basic facts and recent results needed in this
paper.

The following result is extracted from \cite{TSh08}, Theorem 2.34,
where more equivalent conditions are pointed out. Note that in our
considerations below (unless we indicate the contrary) the set $X$
is not necessarily  of finite order.

\begin{facts}
\label{basictheorem} \cite{TSh07}. Suppose $(X,r)$ is nondegenerate,
involutive and square-free quadratic set (not necessarily finite). Then the
following conditions are equivalent:
\begin{itemize}
\item[(i)]  $(X,r)$ is a set-theoretic solution of the Yang--Baxter equation;
\item[(ii)]  $(X,r)$ satisfies {\bf l1};
\item[(iii)]  $(X,r)$ satisfies {\bf r1};
\item[(iv)] $(X,r)$ satisfies {\bf lr3}.
\end{itemize}
In this case $(X,r)$ is cyclic and satisfies  {\bf lri}.
\end{facts}

\begin{corollary}
\label{constructivecor} Every  square-free solution  $(X,r)$
satisfies \textbf{lri}, so it is  uniquely determined by the left
action $\Lcal: X\times X \longrightarrow X,$ more precisely,
\[
r(x,y) = (\Lcal _x(y), \Lcal^{-1}_y(x)).
\]
Furthermore it  is cyclic.
\end{corollary}

The following is straightforward and gives an alternative definition of
square-free solutions.

\begin{lemma}
\label{alternativedef} Let $X$ be a nonempty set and $\Lcal$ be a
map
\[\Lcal: X \longrightarrow \Sym X;\qquad x\mapsto \Lcal_x \in \Sym
X.\] Denote
 $\Lcal_x(y)={}^xy, \quad \Lcal_x^{-1}(y)=y^x$ and define
$r: X\times X \longrightarrow X\times X$ as $r(x,y) =({}^xy, x^y).$
Then $(X,r)$ is a square-free solution if and only if the following
three conditions are satisfied for all $x,y,z \in X$:
\begin{itemize}
\item[(i)] $\quad{}^xx=x$ \item[(ii)] $\quad({}^{y^x}x)= {}^yx$
\item[(iii)] $\quad{}^x{({}^yz)}={}^{{}^xy}{({}^{x^y}{z})}$
\end{itemize}
\end{lemma}

\begin{remark}
\label{cclriinvol} Note that in the hypothesis of Lemma
\ref{alternativedef} condition (ii) implies $(X,r)$ involutive.
\end{remark}

\begin{corollary} \cite{TSh08}
\label{alternativedefc} In the hypothesis and notation of Lemma
\ref{alternativedef},   $(X,r)$ is a square-free solution
if and only if the following  conditions are satisfied for all $x,y,z
\in X$:
\begin{itemize}
\item[(i)] $\quad{}^xx=x$  \item[(ii)] $\quad
{}^{{}^yx}{({}^yz)}={}^{{}^xy}{({}^xz)}$
\end{itemize}
\end{corollary}
Recall that a quadratic set $(X,r)$ which satisfies condition (ii)
of Corollary \ref{alternativedefc} is called \emph{a cyclic set},
see \cite{Rump, TSh08}.

\begin{definition}
\label{permYBgroupdef} The permutation group $H \subseteq \Sym X$
is called \emph{a YB permutation group}, if it is isomorphic to
$\Gcal(X,r)$ for some square-free solution $(X,r)$.
\end{definition}

\begin{openquestions} Let $X$ be a nonempty set.
\begin{enumerate}
\item For which permutation groups $H \subseteq \Sym X$ is there a
square-free solution $(X,r)$ with $\Gcal(X,r)=H$?
\item
Let $m$ be a positive integer.
For which permutation groups $H \subseteq \Sym X$ is there a
square-free solution $(X,r)$ with $\Gcal(X,r)=H$ and $\mpl(X,r) =m$
(see Definition \ref{mpldef})?
\end{enumerate}
\end{openquestions}

The next result gives a translation of the first question which is not
very easy to check.

\begin{proposition}
\label{permYBgroupprop} Let $H \subseteq \Sym X$ be a permutation
group. Then $H$ is a  YB permutation group for some square-free
solution $(X,r)$ if and only if there exists a map
\[
f: X \longrightarrow H;\qquad
  x\mapsto f_x
\]
such that the following conditions hold:
\begin{enumerate}
\item $f(X)$ is a generating set for $H$;
\item \label{i} $f_x(x) = x$;
\item \label{ii} $f_{f_y(x)}\circ f_y = f_{f_x(y)}\circ f_x$.
\end{enumerate}
In this case the quadratic set $(X,r)$ with $r(x,y) = (f_x(y),
(f_y)^{-1}(x))$ is a square-free solution, and $H \simeq
\Gcal(X,r)$.
\end{proposition}

\begin{proof}
Corollary \ref{alternativedefc} implies that the quadratic set
$(X,r)$ is a square-free solution. Clearly, in this case $\Lcal_x
= f_x \in \Sym X.$ Condition \ref{ii} (together with \textbf{lri})
implies
\[
f_x\circ f_y = f_{{}^xy}\circ f_{x^y},
\]
so $(f_x\circ f_y)(a)= f_x(f_y(a))$ and the map $\Lcal_x \mapsto
f_x$ extends to a group homomorphism  $\varphi: \Gcal(X,r)
\longrightarrow H$. We have $f_{x_{i_1}}\circ \cdots f_{x_{i_k}}(y)=
{}^{x_{i_1}}{( \cdots ({}^{x_{i_k}}y)\cdots)},$ so the kernel of
$\varphi$ is trivial.
\end{proof}

In Proposition \ref{mpl2permgroups} we describe the permutation
groups $H \subseteq \Sym X$ which define square-free solutions
$(X,r)$ with $\mpl(X,r)= 2$, where $X$ is an arbitrary finite
nonempty set.

\begin{definition}
Let $(X,r)$ be a braided set, $G = G(X,r), \Gcal= \Gcal(X,r)$. A
subset $Y\subset X$ is said to be \emph{$r$-invariant } if $r(Y
\times Y)\subseteq Y \times Y.$ Suppose $Y$ is an $r$-invariant
subset of $(X,r)$. Then $r$ induces a solution $(Y, r_{\mid Y\times
Y})$. Denote $r_{\mid Y\times Y}= r_{\mid Y}$. We call $(Y, r_{\mid
Y})$ \emph{the restricted solution (on $Y$)}. We say that $Y\subset
X$ is \emph{a (left) $G$ -invariant subset of } $X$, or equivalently
\emph{a $\Gcal$-invariant subset},   if $Y$ is invariant under the
left action of $G$. Clearly, $Y$ is (left) $G$ -invariant if and only if
\[\Lcal_a(Y) \subseteq Y, \quad \forall a\in X.
\]
Right $G$-invariant subsets are defined analogously. In the case
when $(X,r)$ is symmetric, and condition \textbf{lri }  holds the
subset $Y$ is left $G$-invariant if and only if it is right
$G$-invariant. In this case we shall refer to it simply as \emph{a
$\Gcal$-invariant subset}.
\end{definition}
Clearly each $\Gcal$-invariant subset $Y$ of  $X$ is also an
$r$-invariant subset, but, in general, an $r$-invariant subset may,
or may not be $\Gcal$-invariant. The following holds:

\begin{lemma}
\label{gcalinvariantlemma} Let  $(X,r)$ be a symmetric set with
\textbf{lri}, $\Gcal= \Gcal(X,r)$. $Y\subset X$. Denote by $Z$ the
complement of $Y$ in $X$. The following conditions are equivalent.
\begin{enumerate}
\item
$Y$ is $\Gcal$-invariant;
\item
$Z$ is $\Gcal$-invariant;
\item
$Y$ and $Z$ are $r$-invariant complementary subsets of $X$.
\end{enumerate}
Moreover, in this case $(X,r)$ decomposes as a disjoint union of
$r$-invariant subsets $X=Y \bigcup Z.$
\end{lemma}

\begin{proof}
Suppose $Y$ is $\Gcal$-invariant. Clearly,  $x \in Y$ if and only if
the $\Gcal$-orbit of $x$ is contained in $Y$.

It is easy to show now that $Z$ is also $\Gcal$-invariant. Indeed
assume $z \in Z, a \in X, t = {}^az.$ If we assume that $t$ is not
in $Z$, this would imply $t \in Y.$ By \textbf{lri} one has
\[
t = {}^az \Longrightarrow t^a = ({}^az)^a = z
\]
But $t \in Y,$ and the $\Gcal$-orbit of $t$ is contained in $Y,$ so
$z= t^a \in Y,$ a contradiction. Thus $Z$ is $\Gcal$-invariant. This
proves $(1) \Longrightarrow (2)$.

The implication $(2) \Longrightarrow (1)$ is analogous. Clearly,
then each of the conditions (1) and (2) implies that $Y$ and $Z$ are
$r$-invariant.

Suppose now $(X, r)$ is decomposed into two $r$-invariant disjoint
subsets $X = Y\bigcup Z$. We claim that (1) and (2) hold. It will be
enough to verify (1). $Y$ is $r$-invariant, therefore
\[
{}^ay \in Y \quad\forall a,y \in Y.
\]
It remains then to show that
\[
{}^zy \in Y \quad\forall y \in Y, z \in Z.\]
Assume the contrary,
for some $ y\in Y, z \in Z,$  one has ${}^zy \in Z$. This yield
\[t = {}^zy \in Z \Longrightarrow
y = ({}^zy)^z = t^z \in Z, \]
since $Z$ s $r$-invariant. This contradicts $y \in Y$.

It is easy to show that $r_Y: Y \times Y \longrightarrow Y \times Y$
is surjective, and therefore bijective. Clearly, $(Y, r_Y)$ is involutive, hence $(Y, r_Y)$ is a
symmetric set.
\end{proof}

Note that $(X,r)$ is a square-free solution since in this
case condition \textbf{lri} holds, see Corollary
\ref{constructivecor}.

\begin{remark}
\label{Gorbitremark1} Suppose $(X,r)$ is a square-free solution.
Then it satisfies \textbf{lri} (see Corollary \ref{constructivecor})
and therefore Lemma \ref{gcalinvariantlemma} is in force. Clearly,
each $G$-orbit $X_0$ under the left action of $G$ on $X$ is
$\Gcal$-invariant and therefore it is an $r$-invariant subset. In
the case when $G$ acts non-transitively on $X$ (in particular, this
this holds when  $X$ is finite), $(X, r)$ decomposes into a disjoint
union $X = X_0\bigcup Z,$ of its $r$-invariant subsets $X_0$, and
$Z$, where $Z$ is the complement of $X_0$ in $X$.
\end{remark}

Each finite involutive solution $(X,r)$ with {\bf lri} can be
represented geometrically by its \emph{graph of the left action}
$\Gamma(X,r).$ It is an oriented labeled multi-graph (although we
refer to it as a {\em{graph}}). It was introduced  in \cite{T00} for
square-free solutions, see also \cite{TSh07}, Section 5, and
\cite{TSh08}. Here we recall the definition.

\begin{definition}\cite{TSh07} \label{defgraph}
Let $(X,r)$ be a finite symmetric set  with
 \textbf{lri}. We define the
graph $\Gamma=\Gamma (X,r)$ as follows. It is an oriented graph,
which reflects the left action of
 $G(X,r)$ on $X$. The set of
vertices of $\Gamma$ is $X$. There is a labeled arrow  $x
{\buildrel a \over \longrightarrow}y$ if $x,y,a \in X$ and ${}^ax =
y$. An edge $x {\buildrel a \over \longrightarrow}y$ with  $x\neq
y$ is called \emph{a nontrivial edge}. We will often consider the
simplified graph in which to avoid clutter we typically omit
self-loops unless needed for clarity or contrast. Also for the same
reason, we use the line type to indicate when the same type of
element acts, rather than labeling every arrow. Clearly,
$x{\buildrel a\over \longleftrightarrow} y$ indicates that ${}^ax =
y$ and ${}^ay = x.$ (One can make such graphs for arbitrary
solutions but then it should be indicated which action is
considered).
\end{definition}

Note that two solutions are isomorphic if and only if their
oriented graphs are isomorphic. Various properties of a
solution $(X,r)$ are reflected in the properties of its graph
$\Gamma(X,r)$, see for example the remark below.

\begin{remark} Let $(Z,r)$ be a symmetric set with \textbf{lri},
 $\Gamma=\Gamma(Z,r)$.
\begin{enumerate}
\item $(Z,r)$ is a square-free solution if and only if
$\Gamma$ does not contain a nontrivial edge $x{\buildrel x\over
\longrightarrow} y, x \neq y$; that is, the edge of type $x$ leaving $x$
is a loop.
\item In this case, $(Z,r)$ is a trivial solution (or equivalently,
$\mpl(Z,r)=1$) if and only if all the edges are loops.
\item
The $G$-orbits of $X$ are in 1-1 correspondence with the connected
components of $\Gamma$.
\item
\cite{TSh07},  Theorem 5.24 gives necessary and sufficient
conditions for $\mpl(X,r) = 2$ in terms of the properties of
$\Gamma(X,r)$. It can be read off from \cite{TSh07}, Theorem 5.22,
Theorem 5.24 that, in this case, each nontrivial connected component
$\Gamma_i$ of $\Gamma$ is a \emph{Cayley graph} (see below).
\end{enumerate}
\end{remark}

\begin{definition}
If $G$ is a group and $S$ a subset of $G$, the \emph{Cayley graph}
$\Cay(G,S)$ is the graph with vertex set $G$ and directed edges $(g,gs)$
for all $g\in G$ and $s\in S$.
\end{definition}

Note that the group $G$ acts by left multiplication as automorphisms of
the Cayley graph $\Cay(G,S)$. The graph is loopless if and only if
$\id\notin S$ and connected if and only if $\langle S\rangle=G$.

Now any graph admitting a group $G$ of automorphisms acting regularly
on the vertices is a Cayley graph for $G$. For the vertex set
can be identified with $G$ (with action by left multiplication); and, if
$S$ is the set of vertices $s$ for which $(\id,s)$ is an edge, then
$(g,gs)$ is an edge for all $g\in G$ since $G$ acts by automorphisms.
In particular, every transitive abelian group action is regular, so a graph
with a transitive abelian group of automorphisms is a Cayley graph.

The notions of retraction of symmetric sets and multipermutation solutions were
introduced in the general case
in \cite{ESS},  where $(X, r)$ is not necessarily finite, or square-free.

In \cite{T04, TSh07, TSh08, TSh0806} the multipermutation square-free
solutions are studied; we recall some notions and
results. Let $(X,r)$ be a nondegenerate symmetric set. An
equivalence relation $\sim$ is defined on $X$ as
\[ x \sim y \quad \text{if and only if} \quad
\Lcal_x = \Lcal_y.\] In this case we also have $\Rcal_x = \Rcal_y,$

We denote by  $[x]$ the equivalence class of $x\in X$, $[X]=
X/_{\sim}$ is the set of equivalence classes.

\begin{lemma}
\label{retractlemma} \cite{TSh08} Let $(X,r)$ be a nondegenerate
symmetric set.
\begin{enumerate}
\item
The left and the right actions of $X$ onto itself naturally induce
left and right actions on the retraction $[X],$ via
\[
{}^{[\alpha]}{[x]}:= [{}^{\alpha}{x}]\quad [\alpha]^{[x]}:=
[\alpha^x], \;\text{for all}\; \alpha, x \in X.
\]
\item The new actions
define a canonical map $r_{[X]}: [X]\times[X]
\longrightarrow [X]\times[X] $
where $r_{[X]}([x], [y])= ({}^{[x]}{[y]}, [x]^{[y]}).$
\item
$([X], r_{[X]})$ is a nondegenerate symmetric set.
Furthermore,
\item  $(X,r)\; \text{cyclic} \Longrightarrow([X],
r_{[X]})\;\text{cyclic}$.
\item
 $(X,r)\; \text{is}\; {\bf lri} \Longrightarrow([X],
r_{[X]})\;\text{is}\; {\bf lri}.$
\item
$ (X,r)\; \text{square-free} \Longrightarrow ([X], r_{[X]}) \;
\text{square-free}.$
\end{enumerate}
\end{lemma}

\begin{definition}
\label{mpldef} \cite{ESS} The solution $\Ret(X,r)=([X], [r])$ is
called the \emph{retraction of $(X,r)$}. For all integers $m \geq
1$, $\Ret^m(X,r)$ is defined recursively as $\Ret^m(X,r)=
\Ret(\Ret^{m-1}(X,r)).$

$(X,r)$ is  \emph{a multipermutation solution of level} $m$,
 if $m$ is
the minimal number (if any), such that $\Ret^m(X,r)$ is the trivial
solution on a set of one element. In this case we write $\mpl(X,r)=m$.

By definition $(X,r)$ is \emph{a multipermutation solution of level}
$0$ if and only if $X$ is a one element set.
\end{definition}
The following conjecture was made by the first author in 2004.

\begin{conjecture}
  \cite{T04}
\label{conj1}
\begin{enumerate}
\item
Every finite  square-free solution  $(X,r)$ is retractable.
\item
Every finite square-free solution $(X,r)$ of finite order $n$ is
multipermutation solution, with $\mpl(X,r)< n$.
\end{enumerate}
\end{conjecture}

A more recent conjecture states

\begin{conjecture}  \cite{T08ini}
Suppose $(X,r)$ is a nondegenerate square-free multipermutation solution
of finite order $n$. Then $\mpl(X) < \log_2 n$.
\end{conjecture}

Evidence for this conjecture will be given later in the paper.

\section{Homomorphisms, automorphisms, strong twisted unions}
\label{strong twisted unions}
In this section, we recall the definitions and basic properties of
 homomorphisms and automorphisms of solutions, see
and give a general construction, the \emph{strong twisted union} of solutions,
see \cite{TSh07,TSh08}.
\begin{definition} \cite{TSh07}
\label{defhom} Let $(X,r_X)$ and $(Y, r_Y)$ be arbitrary solutions
(braided sets). A map $\varphi: X \longrightarrow Y$ is a
{\em{homomorphism of solutions}}, if it satisfies the equality
\[(\varphi \times \varphi) \circ r_X =
r_Y\circ (\varphi \times \varphi).
\]
A bijective homomorphism of solutions  is called (as usual) an
\emph{isomorphism}. An isomorphism of the solution $(X,r)$
onto itself is an \emph{$r$-automorphism}.

We denote by $\Hom((X,r_X),(Y,r_Y))$ the set of all homomorphisms
of solutions $\varphi: X \longrightarrow Y$. The group of
$r$-automorphisms of $(X, r)$ will be denoted by $\Aut(X, r)$.
\end{definition}
Clearly, $\Aut(X,r)$ is a subgroup of $\Sym(X).$

\begin{remark} \cite{TSh07} Let $(X, r_X), (Y, r_Y)$ be finite square-free solutions
Every homomorphism of solutions  $\varphi: (X, r_X) \longrightarrow
(Y, r_Y)$ induces canonically a homomorphism of their graphs:
\[\varphi_{\Gamma}:
\Gamma(X, r_X) \longrightarrow \Gamma(Y, r_Y).\] Furthermore there
is a one-to one correspondence between $\Aut(X,r)$ and
$\Aut(\Gamma(X,r)),$ the group of \emph{automorphisms} of the
multigraph $\Gamma(X,r).$
\end{remark}


\begin{lemma} \cite{TSh07}
\label{automorphismlemma} Let $(X,r_X)$ and $(Y,r_Y)$ be braided
sets.
\begin{enumerate}
\item
 A map $\varphi: X \longrightarrow Y$ is a homomorphism of
solutions if and only if
\[
\varphi \circ \Lcal_x =\Lcal_{\varphi(x)} \circ\varphi {\rm{\ \ and\
\ }} \varphi \circ \Rcal_x =\Rcal_{\varphi(x)} \circ \varphi
\;\text{for all}\; x \in X.
\]
\item
If both $(X,r_X)$ and $(Y,r_Y)$ satisfy {\bf lri}, then $\varphi$ is
a homomorphism of solutions if and only if
\[
\varphi \circ \Lcal_x =\Lcal_{\varphi(x)}\circ \varphi,  \;\text{for
all}\; x \in X.
\]
\item If $(X,r)$ obeys {\bf lri}, (in particular, if $(X,r)$ is a
square-free solution) then
 $\sigma \in \Sym(X)$ is an automorphism of $(X,r_X)$ if and only if
\begin{equation}
\label{autre}
 \sigma\circ \Lcal_x  \circ \sigma^{-1}=\Lcal_{ \sigma(x)},
\;\text{for all}\; x \in X.
\end{equation}
\end{enumerate}
\end{lemma}

For example, if $(X,r)$ is the trivial solution then, clearly,
$\Aut(X, r)=\Sym(X).$ This is because each $\Lcal_x=\id_X$. More
generally, (\ref{autre}) implies:

\begin{corollary}  \cite{TSh07}
\label{Aut(Xr)Cor1} The group $\Aut(X, r)$ is a subgroup of
$\Nor_{\Sym( X)} \Gcal(X,r)$, the normalizer of  $\Gcal(X,r)$ in
$\Sym(X)$.
\end{corollary}

\begin{corollary}
\label{Aut(Xr)Cor2} Suppose $(X,r)$ is a braided set with {\bf lri}.
Let $Y$, be an $r$- invariant subset,  $(Y, r_{\mid Y})$ be the
restricted solution. Let $x \in X.$ Then $\Lcal_x \in \Aut(Y,
r_{\mid Y})$ if and only if
\begin{equation}
\label{autrelation}   (\Lcal_{{}^{\alpha}x})_{\mid
Y}=(\Lcal_{x})_{\mid Y} \quad \forall \alpha \in Y
\end{equation}
\end{corollary}

\begin{corollary}
\label{Aut(Xr)Cor3} Suppose $(X,r)$ is a square-free solution.
Then the following conditions hold.
\begin{enumerate}
 \item For each  $r$- invariant
subset $Y$, and each $x \in X$, one has $\Lcal_x \in \Aut(Y, r_{\mid
Y})$ if and only if (\ref{autrelation}) holds.
\item   The intersection
$\Gcal_0=\Aut(X,r) \bigcap \Gcal$ is an abelian subgroup of $\Gcal$.
\end{enumerate}
\end{corollary}

\begin{proof}
We know that each square-free solution satisfies {\bf lri}, hence,
Corollary  \ref{Aut(Xr)Cor2} gives (1). By Corollary
\ref{alternativedefc} (ii) the following equality holds:
\begin{equation}
\label{csGcal} \Lcal_{{}^xy}\circ \Lcal_{x}= \Lcal_{{}^yx}\circ
\Lcal_{y}\quad \forall x,y \in X
\end{equation}
Assume now that $\Lcal_{x}, \Lcal_{y} \in \Aut(X,r).$ Then (1)
implies
\begin{equation}
\label{GcalAutab1}
\begin{array}{cl}
\Lcal_{{}^yx} = \Lcal_{x}\quad : &
\quad\text{by}\quad \Lcal_{x} \in \Aut(X,r)\\
\Lcal_{{}^xy} = \Lcal_{y}\quad : & \quad\text{by}\quad \Lcal_{y}
\in \Aut(X,r)
\end{array}
\end{equation}
This together with (\ref{csGcal}) yield
\[
\Lcal_{x} \circ\Lcal_{y}= \Lcal_{y} \circ\Lcal_{x}.
\]
\end{proof}

We shall now discuss a special class of extensions of solutions
called \emph{strong twisted unions of solutions}

\begin{definition} \cite{TSh07}, \cite{TSh08}
\label{STUdef} Let $(X,r)$ be an involutive  quadratic set, suppose
$X$ is a disjoint union $X = X_1\bigcup X_2$ of $r$-invariant
subsets. Suppose the restricted sets $(X_1,r_1)$, $(X_2,r_2)$ are
symmetric sets ($r_i=r_{\mid X_i}, i = 1, 2$). The quadratic set
$(X,r)$ is a \emph{strong twisted union} of $(X_1,r_1)$ and $(X_2,
r_2)$ if
\begin{enumerate}
\item \label{a)} The assignment $\alpha \longrightarrow
\Lcal_{\alpha\mid X_1}, \alpha \in X_2,$ extends to a left action
of the associated group $G(X_2, r_2)$ (and the associated monoid
$S(X_2,r_2)$) on $X_1$, and the assignment $x \longrightarrow
\Lcal_{x\mid X_2} , x \in X_1$ extends to a right action of the
associated group of $G(X_1,r_1)$ (and the associated monoid
$S(X_1,r_1)$)
  on $X_2$;
\item \label{b} The actions satisfy
 \[\begin{array}{lclc}
 {\rm\bf stu :}\quad
&{}^{{\alpha}^y}x = {}^{\alpha}x; \quad
&{\alpha}^{{}^{\beta}x}={\alpha}^x,\quad  \text{for all}\quad x, y
\in X,\; \alpha,\beta \in Y.
\end{array}\]
\end{enumerate}
 \end{definition}

Note first that condition {\rm\bf stu } is equivalent to the
following condition:
\begin{equation}
\label{stu1}
\begin{array}{lclc}
 {\rm\bf stu1 :}\quad
&{}^{{}^y{\alpha}}x = {}^{\alpha}x; \quad
&{\alpha}^{x^{\beta}}={\alpha}^x,\quad  \text{for all}\quad x, y \in
X_1,\; \alpha,\beta \in X_2.
\end{array}
\end{equation}
We shall refer to it also as  {\rm\bf stu} condition.

Secondly, note that by Definition \ref{STUdef} a strong twisted union
 of symmetric sets is
 not necessarily a solution (a symmetric set).
 We shall use notation
 \[
 (X, r)= X_1\stu X_2
 \]
to denote that $(X,r)$ is a symmetric set which is a  strong twisted
union of its $r$-invariant subsets $X_1$ and $X_2.$

The strong twisted union  $(X,r)$ of $(X_1,r_1)$ and $(X_2, r_2)$ is
 \emph{nontrivial} if at least one of the actions in
(\ref{a)}) is nontrivial.  In the case when
      both actions (\ref{a)}) are trivial we write $(X, r) =
      X_1\stu _0 X_2$). In this case one has
      $r(x,\alpha) = (\alpha,x), \; r(\alpha,x)= (x,\alpha)$ for all $x \in X_1,\alpha\in X_2.$

In \cite{TSh07}, \cite{TSh08} and \cite{TSh0806} appear
strong twisted unions of $m$ disjoint symmetric sets, where $m$ is
arbitrary integer, $m \geq 2.$ Although a formal definition was not
given, the notion of strong twisted union there  is clear from the
context. Here we give a formal definition.

\begin{definition}
\label{STUmsetsdef} Let  $(X, r)$ be a symmetric set, which is a
disjoint union of $\Gcal$-invariant subsets  $X_1, \cdots, X_m$,
where $m\geq 2$. For each pair $i\neq j$, $ 1 \leq i, j \leq m$,
denote by $X_{ij}$ the $r$-invariant subset  $X_{ij} = X_i\bigcup
X_j$. We say that $(X, r)$ is a  \emph{a strong twisted union of}
$(X_i, r_i), 1 \leq i \leq m$ and write
\[
X = X_1\stu X_2\stu \cdots \stu X_m
\]
if
\[
X_{ij} = X_i\stu X_j\quad, \forall i\neq j, \; 1 \leq i, j \leq
m.
\]
Here as usual, $(X_i, r_i), 1 \leq i, j \leq m$ denotes the
symmetric set with $r_i$- the restriction of $r$ on $X_i\times X_i.$
 \end{definition}

 \begin{example}
 \label{trivialstu}
Let $(X_1, r_1), \cdots , (X_s, r_s)$ be pairwise disjoint square-free
solutions. Let
 $X = \bigcup_{1 \leq i \leq s} X_i$. Let $r: X\times X \longrightarrow X\times X$
 be the extension of $r_i, 1 \leq i \leq s$ satisfying $r(x, \alpha) = (\alpha, x)$
 whenever $x \in X_i, \alpha \in X_j,$ where $1 \leq i, j \leq s, i \neq j.$
 Then clearly $(X, r)$ is a square-free solution and $(X, r)= X_i \stu _0 X_2 \stu _0 \cdots \stu _0 X_s.$
 \end{example}

Note that when $(X,r)$ is a strong twisted union $X = X_1\stu
X_2\stu \cdots\stu X_m$ of $m$ $\Gcal$-invariant subsets,   $m > 2$,
each set $X_i$, and its complement $Z_i$  are $\Gcal$-invariant.
(Clearly, $Z_i= \bigcup _{1 \leq j \leq m, j \neq i}X_j$). However,
the union $X_i \bigcup Z_i$ may not be a strong twisted union of
solutions.

\begin{corollary}
\label{stucordef}  Let $(X,r)$ be a square-free solution. Suppose
$X_1, X_2, \cdots, X_m $ are disjoint $\Gcal$-invariant subsets of
$X.$

Then $(X,r)$ is a strong twisted union $X = X_1\stu X_2\stu
\cdots\stu X_m$  if and only if for each pair $i\neq j, 1 \leq i, j \leq
m$ and each $x \in X_i$ one has
\[
(\Lcal_{x})_{\mid X_j}  \in \Aut(X_j, r_j).
\]
\end{corollary}

\begin{definition}
\label{splitdef} Suppose $(Z,r)$ is an extension (nondegenerate,
involutive) of the square-free disjoint solutions $(X,r_X)$ and
$(Y, r_Y).$ Consider the  maps
\[
f= f(X,Y): Z \times Z \longrightarrow Z \times Z \quad g= g(Y,X):
Z \times Z \longrightarrow Z \times Z
\]
defined for all $x, y \in X, \alpha, \beta \in Y$ as
\[\begin{array}{lclc}
f(\alpha, x)= ({}^{\alpha}x, \alpha)& \quad f(x, \alpha)= (
\alpha, x^{\alpha})\quad& f_{\mid X\times X}= r_X, &\quad f_{\mid
Y\times Y}= \tau_Y
\\
& & &
\\
g(x, \alpha)= ({}^x{\alpha}, x)& \quad g( \alpha, x)= (x,
\alpha^x)\quad&  g_{\mid X\times X}=\tau _X, \quad & g_{\mid
Y\times Y}= \tau_Y
\end{array}
\]
(Here $\tau_X,\tau_Y$ are the corresponding flips, and the left
and the right actions ${}^{\alpha}{\bullet}, \cdots ,{\bullet}^x $
are the canonical actions defined via $r$ ).

We call $f$ and $g$, respectively, the \emph{associated $X$- and
$Y$-split maps} of $r$.
\end{definition}

\begin{proposition}
\label{splitmapprop} Suppose the quadratic set $(Z,r)$ is an
extension (nondegenerate, involutove) of the square-free disjoint
solutions $(X,r_X)$ and $(Y, r_Y),$ let $f,g$ be respectively the
associated $X$- , respectively, $Y$- split maps of $r$. Then the
following condition holds.
\begin{itemize}
\item[(i)] \label{ci} $f$ and $g$ are involutive maps.
\item[(ii)]
\label{cii}
  There is an equality of maps $r=f \circ\tau\circ g$.
\end{itemize}
Suppose furthermore that $(Z,r)$ is a square-free solution.
\begin{itemize}
\item[(iii)]
\label{ciii}
\[
(Z,f) \quad\text{is a square-free solution}\quad
\Longleftrightarrow \quad {}^{\alpha^y}{x}= {}^{\alpha}{x}\quad
\forall x,y \in X, \alpha \in Y.
\]
In this case $G(Y,r_Y)$ acts as automorphisms on $(X,r_X)$.
\item[(iv)]
\label{civ}
\[(Z,g) \quad\text{is a square-free solution} \Longleftrightarrow
{}^{x^{\beta}}{\alpha}= {}^{x}{\alpha}\quad \forall x\in X,
\alpha, \beta \in Y\] In this case $G(X,r_X)$ acts as
automorphisms on $(Y,r_Y).$
\item[(v)]
\[(Z,f)\; \text{and} \; (Z,g)\quad \text{are square-free solutions}\quad
\Longleftrightarrow \quad (Z,r) = X \stu Y.
\]
\end{itemize}
\end{proposition}

\begin{proof}
Let $x,y \in X$ and  $\alpha, \beta \in Y$. We look at the diagrams
(\ref{ybediagram}). The left hand-side diagram contains all
elements of the orbit of monomial $\alpha xy \in X^3$, under the
action of the group $\Dcal(f)={}_{gr}\langle f^{12}, f^{23}\rangle$.
Analogously the right hand-side diagram contains the elements
of the orbit of $x\alpha \beta \in X^3$ under the action of
$\Dcal(g)={}_{gr}\langle g^{12}, fg^{23}\rangle .$
\begin{equation} \label{ybediagram} \xymatrix{
  \alpha \;x\;y  \ar[r]^{f^{23}} \ar[d]_{f^{12}} & \alpha ({{}^xy}\;x^y)\ar[d]^{f^{12}} & &
x\; \alpha\;\beta  \ar[r]^{g^{23}} \ar[d]_{g^{12}} & x ({{}^{\alpha}{\beta}}\;{\alpha}^{\beta})\ar[d]^{g^{12}}\\
 {}^{\alpha} x\; \alpha\; y \ar[d]_{f^{23}} &  {}^{\alpha}{({}^xy)} \; \alpha\;x^y \ar[d]^{f^{23}} & &
{}^x{\alpha} \; x\; \beta \ar[d]_{g^{23}} & {}^{x}{({}^{\alpha} {\beta})} \; x\;{\alpha}^{\beta} \ar[d]^{g^{23}} \\
 {}^{\alpha} x\; {}^{\alpha}y\; \alpha \ar[d]_{f^{12}} & {}^{\alpha}{({}^xy)} \; {}^{\alpha}{(x^y )}\;\alpha  & &
 {}^x{\alpha} \; {}^{x}{\beta}\; x \ar[d]_{g^{12}} & {}^{x}{({}^{\alpha}{\beta})} \; {}^{x}{({\alpha}^{\beta} )}\;x \\
{}^{{}^{{}^{\alpha} x}}{({}^{\alpha}y)}\;({}^{\alpha}
x)^{{}^{\alpha}y}\; \alpha  & & & {}^{{}^{{}^x{\alpha}
}}{({}^{x}{\beta})}\;({}^x{\alpha})^{{}^{x}{\beta}}\; x
 &  }
\end{equation}

Suppose $f$ is a solution, then from the left hand-side diagram we
obtain an equality of words in the free monoid $\langle X
\rangle:$
\[
{}^{{}^{{}^{\alpha} x}}{({}^{\alpha}y)}\;({}^{\alpha}
x)^{{}^{\alpha}y}\; \alpha = {}^{\alpha}{({}^xy)} \;
{}^{\alpha}{(x^y )}\;\alpha,
\]
and therefore
\begin{equation}
\label{eqn1} {}^{{}^{{}^{\alpha}
x}}{({}^{\alpha}y)}={}^{\alpha}{({}^xy)}
\end{equation}
By assumption, $r$ is a solution, so \textbf{l1} gives
\begin{equation}
\label{eqn2} {}^{\alpha}{({}^xy)}={}^{{}^{{}^{\alpha}
x}}{({}^{{\alpha}^x}y)},
\end{equation}
Now (\ref{eqn1}) and (\ref{eqn2}) imply
\[
 {}^{{}^{{}^{\alpha} x}}{({}^{\alpha}y)}=
{}^{{}^{{}^{\alpha} x}}{({}^{{\alpha}^x}y)},
\]
which by the non degeneracy of $(Z,r)$ implies
\begin{equation}
\label{eqn3} {}^{\alpha}y= {}^{{\alpha}^x}y,
\end{equation}
and this is valid for all $x, y \in X, \alpha \in Y.$

Assume now that $(Z,r)$ is a solution. Then  direct computation to
shows that (\ref{eqn3}) implies that $(Z,f)$ satisfies condition
\textbf{l1}, and therefore is a square-free solution.

Analogous argument shows that under the assumption that $(Z,r)$ is
a solution, $(Z, g)$ is a square-free solution if and only if
\begin{equation}
\label{eqn4} {}^x{\alpha}= {}^{x^{\beta}}{\alpha}\quad \forall x
\in X, \alpha,\beta \in Y
\end{equation}
Clearly, both $(Z,f)$ and $(Z,g)$ are square free solutions
if and only if both (\ref{eqn3}) and (\ref{eqn4}) hold, which  by
Corollary is equivalent to
\[
(Z,r) = X \stu Y.
\]
\end {proof}

\section{Decomposition of solutions.}
\label{Decompositions_Section}
In this section we study various decompositions of square-free
solutions $(X,r)$ into disjoint unions of a finite number of
$r$-invariant subsets and the corresponding factorisation of
$S(X,r)$, $G(X,r)$, and $\Gcal(X,r)$. We use essentially \emph{the
matched pairs approach to solutions} (in the most general setting)
developed in \cite{TSh08} . We first recall the necessary notions
and results from \cite{TSh08}.

\subsection{The matched pairs approach to set-theoretic YBE}

The notion of a matched pair of groups in relation to group
factorisation has a classical origin.
By now there have been various works on matched pairs in different
contexts and we refer to the text in \cite{TSh08} and references
therein. In particular, this notion was used by Lu, Yan and Zhu  to
study the set-theoretic solution of YBE and the associated `braided
group', see \cite{Lu} and the excellent review \cite{Takeuchi}. The
notion of  \emph{a matched pair of monoids }, is developed in
\cite{TSh08}  with additional refinements that disappear in the
group case.

We now recall
some notions and results  from \cite{TSh08}.


 \begin{definition}\label{MLaxioms} $(S,T)$ is a matched pair of monoids if $T$
 acts from the left on $S$ by ${}^{(\ )}\bullet$ and $S$ acts on $T$
 from the right by $\bullet^{(\ )}$ and these two actions obey
\[\begin{array}{lclc}
{\bf ML0}:\quad & {}^a1=1,\quad  {}^1u=u;\quad &{\bf MR0:} \quad &1^u=1,\quad a^1=a \\
 {\rm\bf ML1:}\quad& {}^{(ab)}u={}^a{({}^bu)},\quad& {\rm\bf MR1:}\quad  & a^{(uv)}=(a^u)^v \\
{\rm\bf ML2:}\quad & {}^a{(u.v)}=({}^au)({}^{a^u}v),\quad &{\rm\bf
MR2:}\quad & (a.b)^u=(a^{{}^bu})(b^u),
\end{array}\]
for all $a, b\in T, u, v \in S$.
\end{definition}

\begin{proposition} \cite{TSh08} A matched pair $(S,T)$ of monoids implies the
existence of a monoid $S\bowtie T$ (called the {\em double cross
product}) built on $S\times T$ with product and unit
\[ (u,a)(v,b)=(u.{}^av,a^v.b),\quad 1=(1,1),\quad \forall u,v\in S,\ a,b\in T\]
and containing $S,T$ as submonoids. Conversely, suppose that there
exists a monoid $R$ factorising into monoids $S,T$ in the sense that
(i) $S,T\subseteq R$ are submonoids and (ii) the restriction of the
product of $R$ to a map $\mu:S\times T\to R$ is bijective. Then
$(S,T)$ are a matched pair and $R\cong S\bowtie T$ by this
identification $\mu$.
\end{proposition}

\begin{definition} A {\em strong monoid factorisation} is a factorisation  in submonoids
 $S,T$ as above such that $R$ also factorises into $T,S$.
 We say that a matched pair is {\em strong} if it corresponds to a strong factorisation.
 \end{definition}

\begin{definition}
 \label{braidedmonoiddef}
 A \emph{braided monoid}   is a monoid $S$ forming part of a matched pair $(S,S)$
 such that
\begin{enumerate}
\item The equality
\begin{equation}
\label{M3eq}
 \quad uv = ({}^uv)(u^v) \quad \text{holds in}\quad S, \;\forall  u,v\in S.
\end{equation}
\item
The \emph{associated map} $r_S$
\[r_S: S\times S\to S\times S \quad\text{defined by} \quad
r_S(u,v)= ({}^uv,u^v)\] is bijective and obeys the YBE. A braided
monoid is denoted by $(S, r_S).$
\item
The braided monoid $(S, r_S)$ is called \emph{a strong braided
monoid} if $(S,S)$ is a strong matched pair.
\end{enumerate}
\end{definition}

\begin{remark}
\emph{Matched pairs of groups} and \emph{braided groups} are defined
analogously. Note that if the group $G$ forms a matched pair $(G,G)$
such that $uv = ({}^uv)(u^v)$ holds for all $u,v \in G$ then the
associated map $r_G: G\times G\to G\times G$ with $r_G(u,v): =
({}^uv,u^v)$ is a solution of YBE so $(G, r_G)$ is a braided group.
\end{remark}

The following facts can be extracted from \cite{TSh08} Theorems 3.6
and 3.14.

\begin{facts}
\label{theoremAMP} Let  $(X,r)$ be a braided set and $S=S(X,r)$ the
associated monoid. Then
\begin{enumerate}
\item
The left and the right actions
\[{}^{(\;\;)}{\bullet}: X\times X  \longrightarrow
 X  , \; \text{and}\;\; \bullet^{(\;\;)}: X
\times X \longrightarrow  X
\]
defined via $r$ can be extended in a unique way to a left and a
right action
\[{}^{(\;\;)}{\bullet}: S\times S  \longrightarrow
 S  , \; \text{and}\;\; \bullet^{(\;\;)}: S
\times S \longrightarrow  S.
\]
which make $(S, r_S)$ a strong (graded) braided  monoid. (In
particular $(S, r_S)$ is a set-theoretic solution of YBE).
\item \label{theoremAextranondeg}
$(S, r_S)$ is nondegenerate if and only if $(X,r)$ is nondegenerate.
\item\label{theoremAextrainvol} $(S, r_S)$ is involutive if and only if
$(X,r)$ is involutive.
\end{enumerate}
\end{facts}

The following  is an interpretation of \cite{TSh08}, Remark 4.3 in
our settings.

\begin{remark}
\label{our_extensionsareregular_rem} Suppose the square-free
solution $(Z,r)$ decomposes into a disjoint union of its
$r$-invariant (nonempty) subsets $X$ and $Y$. By definition $(Z,r)$
is nondegenerate, thus the equalities $r(x,y) = ({}^xy,x^y),$
$r(y,x) = ({}^yx,y^x),$ and the nondegeneracy of $r$, $r_X,$ $r_Y$
imply that
\[
{}^yx, x^y \in X, \;\;\text{and }\;\; {}^xy, y^x \in Y,\;\;
\text{for all}\;\; x \in X, y\in Y.
\]
Therefore, $r$ induces bijective maps
\begin{equation}
\label{rhosigma} \rho: Y\times X \longrightarrow X\times Y ,  \;
\text{and} \;\sigma: X\times Y \longrightarrow Y\times X,
\end{equation}
and left and right ``actions''
\begin{equation}
\label{ractions1} {}^{(\;)}{\bullet}: Y\times X \longrightarrow
X,\;\;\; {\bullet}^{(\;)}: Y\times X \longrightarrow Y,\;
\text{projected from}\; \rho
\end{equation}
\begin{equation}
\label{ractions2}
 \la:
X\times Y \longrightarrow Y,\quad  \ra: X\times Y \longrightarrow X,
\ \text{projected\ from}\; \sigma.
\end{equation}
In the general setting of \cite{TSh08} extensions  satisfying
(\ref{rhosigma}) are called \emph{regular extension}.
\end{remark}

\subsection{Decompositions of square-free solutions, and factorisation
of $S$, $G$ and $\Gcal$}

From now on we keep the conventions and the usual notation of this
paper- $(X,r)$ will denotes a square-free solution, not necessarily
finite (unless we indicate the contrary). $S= S(X,r)$, $G = G(X,r)$,
$\Gcal= \Gcal(X,r)$ denote respectively the YB-monoid, YB-group and
the YB permutation group associated with $(X,r)$. $\Lcal:G(X,r)
\longrightarrow \Sym(X)$ is the canonical group homomorphism defined
via the left action, see Remark \ref{ybe}, and by definition $\Gcal
= \Lcal (G(X,r))$. One has $\Gcal = \Lcal (S(X,r))$ whenever $X$ is
a finite set.

It is known that when $(X,r)$ is a finite square-free solution the
group $G$, or equivalently  $\Gcal$, acts intransitively on $X$.
This follows from the decomposition theorem of Rump, \cite{Rump}.
Most of the statements in this and the later sections do not need
necessarily the assumption that $X$ is of finite order.

\begin{remark}
\label{intransitiveactionforinfiniteXRemark}  The group $\Gcal$ acts
intransitively on $X$ whenever $(X,r)$ is a square-free solution
with finite multipermutation level $\quad m\geq 1$ (where $X$ is a
set of arbitrary cardinality) . In this case the number of orbits
$t$ equals at least the cardinality of the $(m-1)$th retract,
$|\Ret^{m-1}(X, r)|$, see Corollary \ref{intransitiveactionforinfiniteXCor}.
\end{remark}
When the set $X$ is  infinite we shall often
 impose the restriction
that the number $t$ of $\Gcal$-orbits is finite, this will be
clearly indicated.

\begin{notation}
\label{notationstandard} $\Ocal_{\Gcal}(x)$ will denote the
$\Gcal$-orbit of $x, x \in X$.

In all cases when $X$ has finite number of $\Gcal$-orbits we shall
denote these by  $X_1, \cdots, X_t.$
\end{notation}

Clearly, the orbits are $r$-invariant subsets of $X$, and each
$(X_i, r_i), 1 \leq i \leq t$, where $r_i$ is the restriction $r_i=
r_{\mid X_i}$, is also a square-free solution. For each $\alpha \in
X_j$ there is an equality of sets
\[
X_j = \{{}^u{\alpha}\mid u \in \Gcal \}.
\]

\begin{notation}
\label{S(Y)G(Y)def} Let $(X,r)$ be a square-free solution, $Y\subset
X$ an arbitrary subset.  We shall use the following notation.

 $\quad \quad S(Y) := $  the submonoid of $S(X,r)$ generated by $Y$;

 $\quad \quad G(Y):= $  the subgroup  of $G(X,r)$ generated by $Y$;

 $\quad \quad \Gcal(Y):= \Lcal(G(Y)).$  This is a subgroup of $\Gcal(X,r)$
 generated by the permutations $\Lcal_y\in \Gcal(X,r),$ where $y \in Y:$
\[
\Gcal(Y) = gr\langle \Lcal_y \in \Gcal(X,r) \mid y \in Y \rangle
\]
\end{notation}

\begin{remark}
\label{trivialG(Y)remark}
Let  $Y\subset X$, an arbitrary subset.
Then $\Gcal(Y) = 1$ if and only if   $Y \subset \ker \Lcal$.
\end{remark}

\begin{remark}
Suppose $Y$ is $\Gcal$-invariant. Then $(Y,r_Y)$ is a square-free
solution, (as usual $r_Y$ denotes the restriction $r_{Y\times Y}$ of
$r$).Then $S(Y)\simeq S(Y,r_Y)$, $G(Y) \simeq G(Y,r_Y),$ (see
Theorem \ref{MAINDECOMPOSITIONTHM}). Note that in general, $\Gcal(Y)$ is different
from the permutation group $\Gcal(Y, r_{\mid Y}) \leq \Sym (Y)$.
Furthermore, if $(X, r)$ is a finite solution and $Y$ is an
$r$-invariant subset of $X$, then the group $\Gcal(Y)$ is the image
of $S(Y)$ under the map $\Lcal : G(X,r) \longrightarrow \Sym( X)$.
\end{remark}

\begin{proposition}
\label{Ginvariantprop} Let $(X,r)$ be a square-free solution, $S =
S(X,r)$, $G = G(X,r)$. Suppose $Y$ is a $G$-invariant subset of $X$.
Then
\[
\begin{array}{c}
(i) \quad {}^au \in S(Y), \quad u^a  \in S(Y) \quad \forall u \in
S(Y), a \in S;
\\
(ii) \quad {}^au \in G(Y), \quad u^a  \in G(Y) \quad \forall u \in
G(Y), a \in G.
\end{array}
\]
\end{proposition}
Under the hypothesis of the proposition we prove first the following
key lemma.

\begin{lemma}
\label{Sinvariantlemma}
With the assumptions and notation of Proposition
\ref{Ginvariantprop}, the following are equalities in $G$
\begin{equation}
\label{keyeq1} {}^a{(y^{-1})}= ({}^ay)^{-1};\quad \quad (y^{-1})^a =
(y^a)^{-1} \quad \forall a \in G, \;\; y \in X.
\end{equation}
\end{lemma}

\begin{proof}
Note that each element $a \in G$ can be presented as a monomial
\begin{equation}
\label{keyeq2} a = \zeta_1\zeta_2 \cdots \zeta_n,\quad \zeta_i \in X
\bigcup X^{-1}.
\end{equation}
We shall consider  \emph{a reduced form of} $a$, that is a
presentation \ref{keyeq2} with minimal length $n$. We shall use
induction on the length $n$ of the reduced form of $a$.

\textbf{Step 1.} $a \in X \bigcup X^{-1}$. Two cases are possible:
\textbf{(i)} $a \in X.$ By the cyclic condition we have ${}^{{}^ya}y
= {}^ay$. This implies
\[
({}^{{}^ya}y).({}^ay)^{-1} = 1
\]
Recall that $(G,G)$ is a matched pair of groups, thus ${}^a1= 1$ for all  $a \in G.$ Consider the equalities
\[
\begin{array}{clll}
1 = {}^{{}^ya}1&= {}^{{}^ya}{(y.y^{-1})} \quad &:& \text { by \textbf{ML0}} \\
\quad&=[{}^{{}^ya}{y}].[{}^{({}^ya)^y}{(y^{-1})}]\quad &:& \; \text {by \textbf{ML2}} \\
\quad &={}^{{}^ya}{y}.({}^a{(y^{-1})}\quad &:& \text{ since $({}^ya)^y=a$
by \textbf{lri}}.
\end{array}
\]
Hence
\[
1 = ({}^{{}^ya}y).({}^ay)^{-1}={}^{{}^ya}{y}.({}^a{(y^{-1})},
\]
which is an equality in the group $G$, therefore the left hand side
of (\ref{keyeq1}) holds. For the right hand side one uses analogous
argument. By the (right) cyclic condition one has $y^{a^y}=y^{a}$,
hence
\[
 (y^{a})^{-1}(y^{a^y}) = 1
\]
This time we act on the right-hand side:
\[
\begin{array}{clll}
1 = 1^{a^y} &=(y^{-1} y)^{a^y}\quad &:&  \text {  by \textbf{MR0}}  \\
&=[(y^{-1})^{{}^y{(a^y)}}] [y^{(a^y)}]\quad&:&\text { by \textbf{MR2}}\\
&=[(y^{-1})^a ][y^{(a^y)}] \quad&:& \text {  by \textbf{lri}}.
\end{array}
\]
So $(y^{-1})^a (y^{(a^y)})= (y^{a})^{-1}(y^{a^y})$ holds in
 $G$ and therefore $ (y^{-1})^a = (y^{a})^{-1}$. This verifies
RHS of (\ref{keyeq1}).
 \textbf{(ii)} $a\in X^{-1}$, or equivalently $a= \zeta^{-1}$,
where $\zeta \in X.$ Recall that there is an equality
\[{}^{\zeta^{-1}}{y} = y^{\zeta} \quad \forall \zeta, y \in X.\]
Consider now the equalities:
\[
\begin{array}{clll}
1 & = {}^{{}^y{(\zeta^{-1})}}{[y y^{-1}]}\quad &:&  \text {by \textbf{ML0}}  \\
  &= {}^{(({}^y{\zeta})^{-1})}{[y y^{-1}]}\quad &:&  \text {by case \textbf{i}}\\
  &= [{}^{(({}^y{\zeta})^{-1})}y].[ {}^{(({}^y{\zeta})^{-1})^y}
  {y^{-1}}]\quad &:&  \text {by \textbf{ML2}} \\
  &= [y^{({}^y{\zeta})}].[{}^{{(({}^y{\zeta})^y)}^{-1}}{(y^{-1})}\quad &:&
   \text {by case \textbf{i}}\\
  &=[y^{\zeta}].[{}^{({\zeta}^{-1})}{(y^{-1})}\quad &:&  \text {by the cyclic condition and \textbf{lri}}  \\
  &=[{}^{\zeta^{-1}}y][{}^{({\zeta}^{-1})}{(y^{-1})}.
  \end{array}
\]
So the equality
\[
[{}^{\zeta^{-1}}y][{}^{({\zeta}^{-1})}{y^{-1}}] = 1
\]
implies
\[
[{}^{\zeta^{-1}}y]^{-1}={}^{({\zeta}^{-1})}{y^{-1}}
\]
and therefore
\[
({}^{a}y)^{-1}={}^{a}{(y^{-1})}.
\]
This proves the LHS of (\ref{keyeq1}). Analogous argument verifies
its RHS.

\textbf{Step 2. } Assume (\ref{keyeq1}) hold for each $y \in X$ and
each $a \in G$ with reduced form of minimal length $n$. Suppose $a
\in G$ has minimal length $n+1$. Then $a = \zeta b,$ where $\zeta
\in X\bigcup X^{-1}$, and $b \in G$ has length $n$. Then
\[
\begin{array}{clll}
{}^a{(y^{-1})}&= {}^{(\zeta b)}{(y^{-1})} \\
             &= {}^{\zeta}{({}^b{(y^{-1})})}\quad &:& \text {by \textbf{ML1}}\\
             &= {}^{\zeta}{(({}^by)^{-1})}\quad &:& \text {by the inductive assumption}\\
             &= ({}^{\zeta}{({}^by)})^{-1}\quad &:& \text {by the inductive assumption}\\
             &= ({}^{(\zeta b)}{y})^{-1} \quad &:& \text {by \textbf{ML1}} \\
             & = ({}^ay)^{-1} \quad &:& \text {by}\;\; a = \zeta b.
\end{array}
\]
This proves the LHS of (\ref{keyeq1}), the remaining part is proven
analogously. The lemma has been proved.
\end{proof}

\begin{proofproposition}
We shall prove the implication
\begin{equation}
\label{Ginvareq1} u \in G(Y), a \in G \Longrightarrow {}^au \in
G(Y)\quad {}a^u \in G(Y) \end{equation} This time we use induction
on the length of $u.$ Lemma \ref{Sinvariantlemma} gives the base for
the induction. Assume now (\ref{Ginvareq1}) holds for all $ u \in
G(Y)$ with length $n$. Suppose $u \in G(Y)$ has length $n+1$, so
$u=\zeta v$, where $\zeta \in Y\bigcup Y^{-1}$, and $v \in G(Y)$ has
length $n$. One has
\[
\begin{array}{clll}
{}^au&= {}^{a}{(\zeta v)}&& \\
     &=[{}^{a}{\zeta}] [{}^{a^{\zeta}}v]\quad
\quad &:& \text {by \textbf{ML2}} \\
             &\in G(Y)\quad &:& \text {since
             ${}^{a}{\zeta},\;\;{}^{{a}^{\zeta}}v \in G(Y)$ by the inductive assumption.}
\end{array}
\]
Analogous argument verifies $u^a \in G(Y)$. This proves part (ii) of
the proposition. The proof of (i) is analogous.
\end{proofproposition}

\begin{corollary}
In notation as above let $Y$ be a $G$-invariant subset of $(X,r)$.
Then

i) $S(Y)$ is an $r_S$-invariant subset of the braided monoid
$(S,r_S)$

ii) $G(Y)$ is an $r_G$-invariant subset of the braided group
$(G,r_G)$
\end{corollary}

\begin{proof}
We shall prove (ii) ((i) is analogous).  We know that  $(G,r_G)$ is
a braided group and $r_G$ is defined via the left and right actions
on $G= G(X,r)$. So we have
 \[
 r_G(u,v) = ({}^uv, u^v)  \quad \forall u,v \in G.
 \]
By Proposition \ref{Ginvariantprop}  each pair  $u, v \in G(Y)$
satisfies ${}^uv, u^v \in G(Y).$ This shows that $G(Y)$ is
$r_G$-invariant.
\end{proof}

\begin{theorem}
\label{MAINDECOMPOSITIONTHM} Let $(X,r)$ be a square-free solution,
which decomposes into a disjoint union $X = Y\bigcup Z$  of
$r$-invariant subsets. Let $(Y,r_Y),  (Z,r_Z)$ be the restricted
solutions, $G =G(X,r)$,
 $G_Y=G(Y,r_Y)$, $G_Z=G(Z,r_Z)$. Then
\begin{enumerate}
\item
$G(Y) \simeq G_Y$, $G(Z) \simeq G_Z.$
\item \label{BP1}
$G_Y, G_Z$ is a  matched pair of groups with
actions induced
 from  the braided group $(G, r_G).$
 $G= G(X,r)$ is isomorphic to the double crossed products
 \[
 G \simeq G_Y\bowtie
 G_Z\simeq G_Z\bowtie G_Y.  \]
 In particular, $G$ factorises as:
 \begin{equation}
 \label{factorisationGeq}
G= G(X,r) = G_Y G_Z = G_Z G_Y.
 \end{equation}
 \item
$\Gcal$  decomposes as product of subgroups (which in general is not
a factorisation):
  \begin{equation}
 \label{decompositionGcaleq}
 \Gcal = \Gcal(Y) \Gcal(Z)= \Gcal(Z) \Gcal(Y).
 \end{equation}
\end{enumerate}
\end{theorem}
\begin{proof}
It follows from  \cite{TSh08} Proposition 4.25. that $G_Y, G_Z$ is a
matched pair, so there is a factorisation $G = G_Y G_Z $, and each
$w \in G$ has unique presentation as
\begin{equation}
\label{BP1eq1} w = ua \quad \text{with}\quad u \in G_Y, a \in G_Z.
\end{equation}
On the other hand $(X,r)$ is a solution, thus $(G, r_G)$ is a
braided group and the equality
\[
ua= ({}^ua)(u^a)\quad \text{holds}\quad \forall \; u, a \in G.
\]
$Y$ and $Z$ are $G$-invariant subsets of $X$, so by Proposition
\ref{Ginvariantprop}
\[
\begin{array}{c}
a \in G_Z, u \in G\Longrightarrow {}^ua \in G_Z \\
u \in G_Y, a \in G \Longrightarrow u^a \in G_Y. \end{array}
\]
Therefore each element $w \in G$ presents as
\begin{equation}
\label{BP1eq2} w = ua = a_1u_1 \quad\text{ where}\quad u, u_1 \in
G_Y, a, a_1 \in G_Z, \; a_1 = {}^ua, u_1 = u^a. \end{equation}

The uniqueness of $a_1$ and $u_1$ in (\ref{BP1eq2}) follows from
$G_Y\bigcap G_Z = 1.$ This implies the factorisation $G = G_Z G_Y $,
hence $G_Y, G_Z$ is a strong matched pair.

We apply the group homomorphism $\Lcal$ to (\ref{factorisationGeq}).
Now the equalities $\Lcal(G) = \Gcal$ , $\Lcal(G_Y)=\Lcal(G(Y)) =
\Gcal(Y)$ and $\Lcal(G_Z)=\Lcal(G(Z)) = \Gcal(Z)$ give the
decomposition (\ref{decompositionGcaleq}). Note that each $w \in
\Gcal$ decomposes as a product $w = ua = a_1u_1,$ where $u,u_1 \in
\Gcal(Y), a, a_1 \in \Gcal(Z),$ but this presentation is possibly
not unique.
\end{proof}

\begin{proposition}
With the assumptions and notation of Theorem \ref{MAINDECOMPOSITIONTHM},
there are isomorphisms of monoids $ S(Y) \simeq S(Y, r_{Y}), \quad
S(Z) \simeq S(Z, r_{Z})$. $S(X)\bigcap S(Y)=1$. Furthermore, (S(Y),
S(Z)) is a strong matched pair of monoids,  $S$ is isomorphic to the
double crossed product
\[S = S(X,r) \simeq S(Y)\bowtie
 S(Z)\simeq S(Z)\bowtie
 S(Y).\] There is a factorisation of monoids
\[
S=S(Y)S(Z)=S(Z)S(Y),
\]
where each $w \in S$ decomposes uniquely as \[w = ua = a_1u_1,\quad
u,u_1 \in S(Y), a, a_1 \in S(Z).\]
\end{proposition}

The following lemma is straighforward. It verifies the associativity
of bicross products for $\Gcal$-invariant subset of $(X, r)$.

\begin{lemma}
Notation as above. The double cross product on $\Gcal$-invariant
disjoint subsets of $(X,r)$ is commutative and associative. More
precisely,  suppose $Y_1, Y_2, Y_3$ are pairwise disjoint
$\Gcal$-invariant subsets of $X$. Let $Y = Y_1\bigcup Y_2 \bigcup
Y_3$. Then
\[
S(Y_i\bigcup Y_j) \simeq S(Y_i)\bowtie S(Y_j)\simeq S(Y_j)\bowtie
S(Y_i), 1 \leq i < j\leq s. \] Analogous statement is true for the
groups $G(Y_i \bigcup Y_j),  1 \leq i < j\leq s.$ Furthermore,
\[
\begin{array}{rcl}
S(Y) & \simeq S(Y_1)\bowtie [S(Y_2) \bowtie S(Y_3)] \simeq
[S(Y_1)\bowtie S(Y_2)] \bowtie S(Y_3)
\\
G (Y) & \simeq G(Y_1)\bowtie [G(Y_2) \bowtie G(Y_3)] \simeq
[G(Y_1)\bowtie G(Y_2)] \bowtie G(Y_3).
\end{array}
\]
\end{lemma}

\begin{theorem}
\label{decompositiontheorem} Let $(X,r)$ be a nontrivial square-free solution.
Suppose $X$ decomposes into a disjoint union
\[
X = \bigcup_{1 \leq i \leq p} Y_i,
\]
of  $\Gcal$-invariant subsets $Y_1, Y_2, \cdots Y_s$. Then
\[
\begin{array}{rcl}
S=S(X,r) & \simeq S(Y_1)\bowtie S(Y_2) \bowtie \cdots  S(Y_s)
\\
G= G(X,r) & \simeq G(Y_1)\bowtie G(Y_2) \bowtie \cdots  G(Y_s).
\end{array}
\]
In particular,
\begin{enumerate}
\item
 $S=S(X,r)$ factorises as
 a product of submonoids:
\begin{equation}
\label{decompositionS_Y} S=S(X,r)=S(Y_1)\;S(Y_2)\;\cdots \;S(Y_s),
\end{equation}
where each $u \in S$ has unique presentation $ u = u_1u_2\cdots
u_s$, $u_i \in S(Y_i), 1\leq i\leq s.$
\item
$G=G(X,r)$ factorises as a product of subgroups:
\begin{equation}
\label{decompositionG_Y} G=G(X,r)=G(Y_1)\;G(Y_2)\;\cdots \;G(Y_s),
\end{equation}
where each $u \in G$ has unique presentation $ u = u_1u_2\cdots u_s,
u_i \in G(Y_i), 1\leq i\leq s.$
\item
\begin{equation}
\label{decompositionGcal_Y}
\Gcal=\Gcal(X,r)=\Gcal(Y_1)\;\Gcal(Y_2)\;\cdots \;\Gcal(Y_s),
\end{equation}
in the sense that each $a \in \Gcal$ is presented as a product $a
= a_1a_2\cdots a_s$, where $a_i \in \Gcal(Y_i), 1\leq i\leq s$, but this
presentation is possibly not unique. At least one of the groups
$\Gcal(Y_i)$ is nontrivial. (We have $\Gcal(Y_i) = 1$ if and only if
$Y_i \subset \ker \Lcal$ ).
\end{enumerate}
\end{theorem}
\begin{corollary}
\label{decompositionorbits} Let $(X,r)$ be a nontrivial square-free
solution, which is either finite, or infinite but with a finite set
of $\Gcal$-orbits. Let $X_1\cdots X_t$ be the set of all orbits in
$X$ denoted so that the first $t_0$ orbits are exactly the
nontrivial ones. Then
\begin{equation}
\label{decompositionorbitseq}
\begin{array}{rcl}
S=S(X,r)&=&S(X_1)\;S(X_2)\;\cdots \;S(X_t) \\
G=G(X,r)&=&G(X_1)\;G(X_2)\;\cdots \;G(X_t)\\
\Gcal=\Gcal(X,r) &=& \Gcal(X_1)\;\Gcal(X_2)\; \cdots
\;\Gcal(X_{t_0}).
\end{array}
\end{equation}
\end{corollary}
For multipermutation square-free solutions $(X,r)$ with $\mpl X = m$
there is a natural and important decomposition: $X$ decomposes  as a
disjoint union of its $(m-1)$ th retract classes. The retract
classes $[x^{(k)}], 1 \leq k, x \in X$ are introduced in section
\ref{generalmplsection},
see Notation \ref{notret0}, and Facts \ref{VIPfacts}. They are
disjoint $r$-invariant subsets of $(X,r)$ and behave nicely. Note
that when $\mpl X = m$, and $k < m-1$ at least one of the $k$-th
retract classes is not $\Gcal$-invariant. Moreover, each $(m-1)$-th
retract class $[x^{(m-1)}]$ is $\Gcal$-invariant and contains the
orbit $\Ocal_{\Gcal} (x)$. More precisely, each retract class
$[x^{(m-1)}]$ splits into a disjoint union of the orbits $X_{i_j},$
which intersect it nontrivially.

\begin{theorem}
\label{decompositionretractclasses} Let $(X,r)$ be a nontrivial
square-free solution, which is either finite, or is infinite but
with a finite number of $\Gcal$-orbits. Suppose it has a finite
multipermutation level $\mpl (X,r)= m.$ Then
\begin{enumerate}
\item
\label{decompretractclasses1}
$\Ret^{(m-1)}(X,r)$ is a finite set of order $\leq t,$ where $t$ is
the number of $\Gcal$-orbits of $X$. Let
\[Y_1 = [\xi_1^{(m-1)}], \cdots, Y_s =[\xi_s^{(m-1)}]\] be the set of
all distinct $(m-1)$-retract classes in $X$. One has $s \geq 2$, and
$|Y_i| \geq 2$ for some $1 \leq i\leq
s$.
\item
\label{decompretractclasses2}
 $X$ is a disjoint union $X = \bigcup_{1 \leq i \leq s} Y_i.$
 Each $Y_i, 1 \leq i \leq s,$ is $\Gcal$-invariant, and  $\mpl (Y_i, r_i ) \leq m-1$, where $(Y_i,r_i)$ is the restricted solution.
 \item
 \label{decompretractclasses3}
 The  monoid $S= S(X,r)$ and the group
$G=G(X,r)$ have factorisations as in (\ref{decompositionS_Y}) and
(\ref{decompositionG_Y}), respectively. $\Gcal=\Gcal(X,r)$ also
decomposes as a product of subgroups (\ref{decompositionGcal_Y}),
but some pairs of these subgroups may have nontrivial intersection.
\end{enumerate}
\end{theorem}

\begin{proof}
Clearly $X= \bigcup_{1 \leq i \leq s} Y_i$ is a disjoint union.
It follows from Proposition \ref{vip_retrprop} that $s \geq 2$,
$Y_i$ is $\Gcal$-invariant for $1 \leq i \leq s$ and contains each $\Gcal$-orbit which intersects it nontrivially.
If we assume  $\mid Y_i\mid = 1, 1 \leq i \leq s$, this would imply that all $\Gcal$- orbits in $X$ are trivial,
and therefore $(X,r)$ is a trivial solution, a contradiction.
It follows than that  $\mid Y_i\mid \geq 2$ for some $i, 1 \leq i \leq s.$
The inequality $mpl (Y_i, r_i)\leq m-1$ follows from  Facts \ref{VIPfacts}  \ref{fact2}. We have proved (1) and (2).
(3) follows straightforwardly from Theorem \ref{decompositiontheorem}.
\end{proof}
\begin{remark}
Proposition \ref{theoremC} shows that in the particular case when $2 \leq m \leq 3$
$(X, r)$ is a strong twisted union
\[
X = Y_1 \stu Y_2 \stu \cdots \stu Y_s,
\]
and $\Gcal$ decomposes as a product of abelian subgroups
\[
\Gcal = \Gcal(Y_1)\Gcal(Y_2)\cdots  \Gcal(Y_s),
\]
\end{remark}
\section{Multipermutation solutions of low levels.}
\label{Multipermutation solutions of low levels}
A natural question arises:
\begin{question} Suppose $(X,r)$ is a multipermutation square-free solution.
What is the relation between the multipermutation level $\mpl(X,r)$
and the algebraic properties of $S(X,r)$, $G(X,r)$, $\Gcal(X,r)$,
$k\Gcal(X,r)$, $\Acal(k, X,r)$?
\end{question}

The following is straightforward.

\begin{lemma}
\label{mpl=1} Suppose $(X,r)$ is a  square-free solution of order
$\geq 2$ (but not necessarily finite). Then the following conditions
are equivalent.
\begin{enumerate}
\item $\mpl(X,r)=1$. \item (X,r) is the trivial solution, i.e.
$r(x,y)=(y,x),$ for all $x,y \in X$. \item $S(X,r)$ is the free
abelian monoid generated by $X.$ \item $G(X,r)$ is the free abelean
group generated by $X.$
\item $\Gcal(X,r)= \{\id\},$ the trivial group.
\end{enumerate}
\end{lemma}

\subsection{Square-free  solutions of  multipermutation level  2 }
\label{mpl2section}
Detailed study of square-free  solutions of  multipermutation level
$2$ is performed in  \cite{TSh0806}, and \cite{TSh08}.
We recall first some results from these works needed in the sequel.

\begin{proposition}\label{sigprop} \cite{TSh0806}, Let  $(X,r)$ be a square-free
solution of  finite order, let $X_i, 1 \leq i \leq t$ be the set of
all $G(X,r)$-orbits in $X$ enumerated so that $X_1, \cdots ,
X_{t_0}$ is the set of all nontrivial orbits (if any). Then the
following are equivalent.
\begin{enumerate}
\item \label{mpl2_1} $(X,r)$ is a multipermutation solution of
level 2 \item \label{mpl2_2} $t_0 \geq 1$ and for each $ j, 1 \leq
j \leq t_0$,  $x,y \in X_j$ implies $\Lcal_x = \Lcal_y$. \item
\label{mpl2_3} $t_0\geq 1$ and for each $x\in X$ the permutation
$\Lcal_x$ is an $r$-automorphism, i.e. $\Gcal(X,r) \subseteq
\Aut(X,r)$.
\end{enumerate}
\end{proposition}

\begin{theorem}
\label{significantth} \cite{TSh0806} Let  $(X,r)$ be square-free
solution of multipermutation level $2$ and finite order, and $X_i,
1 \leq i \leq t$ be the set of all $G(X,r)$-orbits as in Popostion
\ref{sigprop}. Let $(X_i, r_i),1 \leq i \leq t $ be the restricted
solution. Then:
\begin{enumerate}
 \item  \label{mpl2_4} $\Gcal(X,r)$ is a nontrivial abelian group.
\item \label{mpl2_5} Each $(X_i,r_i), 1 \leq i \leq t_0$ is a
trivial solution. Clearly in the case $t_0 < t,$ each $(X_j,r_j)$,
with $t_0 \leq j \leq t$ is a one element solution. \item
\label{mpl2_6} For any  ordered pair $i, j, 1\leq i\leq t_0, 1\leq
j\leq t,$ such that $X_j$ acts nontrivially on $X_i$, every $x\in
X_j$ acts via the same  permutation $\sigma^j_i\in \Sym(X_i)$
which is a product of disjoint cycles of equal length $d=d^j_i$
\[
\sigma^j_i= (x_1\cdots x_d) (y_1\cdots y_d)\cdots (z_1\cdots z_d),
\]
where  each element of $X_i$ occurs exactly once. Here  $d^j_i$ is
an invariant of the pair $X_j, X_i$. \item \label{mpl2_7} $X$ is
\emph{a strong twisted union} $X= X_1\stu X_2 \stu \cdots\stu X_t$.
\end{enumerate}
\end{theorem}

\begin{fact}
\label{fact_lri_cc_involutive}\cite{TSh08}, Proposition 2.5. Let
$(X,r)$ be a quadratic set. Then any of the following two conditions
imply the third. (i) $(X,r)$ is involutive; (ii) $(X,r)$ is
nondegenerate and cyclic; (iii) $(X,r)$ satisfies \textbf{lri}.
\end{fact}
\begin{proposition}
\label{mpl2braidmonoidprop}
Suppose $(X,r)$ is a multipermutation solution of level 2. Then
\begin{enumerate}
\item
The associated braided monoid  $(S,r_S)$ is a symmetric set
which satisfies the cyclic conditions and \textbf{lri } (see
Definition \ref{lri&cl}). $(S,r_S)$ is not square-free.
 Furthermore, $S$  acts on itself
 as automorphisms:
\begin{equation}
\label{auteq} {}^a{(uv)} = ({}^au)({}^av)\quad (uv)^a = (u^a)(v^a)
\quad \forall\quad a, u, v \in S.
 \end{equation}
\item
 An analogous statement is true for the associated
braided group $(G, r_G).$
\end{enumerate}
\end{proposition}
\begin{proof}
By Facts \ref{theoremAMP} $(S, r_S)$ is an nondegenerate involutive
set-theoretic solution of YBE, therefore it is a
 symmetric set.
 By Proposition \ref{sigprop} $\Gcal(X, r) \subseteq Aut (X,r),$
 thus, by Corollary \ref{Aut(Xr)Cor2} and by the definition of
 automorphism of solutions one has
\begin{equation}
\label{usefuleq}
\begin{array}{cl}
 \Lcal_{{}^{\alpha}x} =\Lcal_{x} = \Lcal_{x^{\alpha}} &\quad \forall\; x,
 \alpha \in X\\
{}^a{(xy)}=({}^ax)({}^ay)&\quad \forall\; a\in S, \;
 x, y  \in X.
 \end{array}
 \end{equation}
Using (\ref{usefuleq}) and induction on the length $\mid v \mid$ of
 $v \in S$) one shows easily that
\begin{equation} \label{usefuleq2}
\begin{array}{lclc}
 \quad&  {}^{a^b}v= {}^av \quad
 \quad&\quad &v^{{}^ba}= v^a, \quad\text{for all}\;
a,b, v \in
 S;
\\
\quad
  &{}^{{}^ba}v= {}^av,
\quad
\quad & \quad &v^{a^b}= v^a \quad\text{for all}\; a,b, v \in S.
\end{array}
 \end{equation}
replacing $b$ with $v$ in each of the above equalities one yields
the cyclic conditions for the solution $(S, r_S)$. We have shown
that $(S, r_S)$ satisfies (i) and (ii) of Fact
\ref{fact_lri_cc_involutive}, hence it satisfies (iii). This
verifies \textbf{lri} ($({}^au)^a= u = {}^a{(u^a)}$ for all $a,u \in
S$). We will  verify that the left action of $S$ (see Facts
\ref{theoremAMP}) is as automorphisms (the proof for the right
action is analogous). This follows from the equalities:
\[\begin{array}{clll}
{}^a{(uv)} &= ({}^au)({}^{a^u}v)\quad &:& \text{by  \textbf{ML2}}\\
           &= ({}^au)({}^av)\quad &:& \text{by  (\ref{usefuleq2})}.
           \end{array}
\]
We claim that $(S, r_S)$ is not square-free, or equivalently there
exist an $a \in S,$ such that ${}^aa \neq a$. Assume the contrary.
By hypothesis $(X,r)$ is not the trivial solution, so there exist
$x,y \in X$, with ${}^yx\neq x$. Let $a = xy$
\[\begin{array}{clll}
{}^aa = {}^{xy}{(xy)} &= ({}^{xy}x)({}^{xy}y)\quad &:&\text{ by (\ref{auteq})}\\
           &= ({}^{y}{({}^xx)})({}^{x}{({}^yy)})\quad&:& \text{ by $\Gcal$ \;
           abelian}\\
         &= ({}^yx)({}^xy)\quad \quad&:&\text{ by $(X,r)$ \;
           square-free}.\\
           \end{array}
\]
So $ xy = a ={}^aa=({}^yx)({}^xy)$ are equalities in $S$. The only
quadratic relation in $S$ involving $xy$ is
\[xy = {}^xy. x^y.\]
Therefore  one of the following is an equality of words in the free
monoid $\langle X \rangle:$
\begin{equation}
 \label{usefuleq3}
({}^yx)({}^xy) = xy \quad \text{ in} \quad\langle X \rangle
 \end{equation}
 \begin{equation}
 \label{usefuleq4}
({}^yx)({}^xy) ={}^xy. x^y \quad \text{ in} \quad\langle X \rangle.
 \end{equation}
\textbf{ Case a.}  (\ref{usefuleq3}) holds.  Then ${}^yx = x$ which
contradicts the choice of $x$ and $y$. \textbf{ Case b.}
(\ref{usefuleq4}) is in force. Hence ${}^yx = {}^xy$. We apply right
action by $^y$ on both sides of this equality and obtain
\begin{equation}
 \label{usefuleq5}
({}^yx)^y = ({}^xy)^y.
 \end{equation}
Thus
\[
\begin{array}{cll}
x\quad  &= ({}^yx)^y \quad &\text{by \textbf{lri}}\\
        &= ({}^xy)^y \quad &\text{by (\ref{usefuleq5})}\\
        &= {}^{x}{(y^y)}\quad &\text{by $\Gcal$ abelian} \\
        &= {}^xy \quad \quad &\text{by $(X,r)$ \;
           square-free}.\\
           \end{array}
\]
Therefore ${}^xy = x = {}^xx$. It follows then by the nondegeneracy
of $(X,r)$ that $y=x$, and ${}^yx = {}^yy = y.$ A contradiction with
the choice of $x,y.$  We have verified the first part of the
proposition.
  An analogous argument
proves the statement for the braided group $(G, r_G)$.
\end{proof}


It is natural to ask

\begin{questions}
1) When can an abelian  group of permutations $H \leq \Sym X$ be
considered as a permutation YB group of a solution $(X,r)$?

In particular,

2) When is  $H \leq \Sym X$ a  permutation YB group of a
solution $(X,r)$ of multipermutation level 2?
\end{questions}

An answer to the second question is given in the following
Proposition \ref{mpl2permgroups}. Furthermore, as shows  Corollary
\ref{mpl2permgroupscor}, every finite abelian group is isomorphic to
the permutation group of  a solution $(X,r)$ with $\mpl X = 2$.

\begin{proposition}
\label{mpl2permgroups} Let $H$ be an abelian permutation group on a
set $X$. Then the following are equivalent:
\begin{itemize}
\item[(a)] there is a solution $(X,r)$, with $\mpl(X,r)=2$, such that
$\mathcal{G}(X,r)=H$;
\item[(b)] there is a function $f$ from the set of $H$-orbits on $X$ to $H$
with the properties
  \begin{itemize}
  \item $f(X_i)$ fixes every point in $X_i$;
  \item the image of $f$ generates $H$.
  \end{itemize}
\end{itemize}
\end{proposition}

\begin{proof} We use the following facts:
If $H$ is an abelian permutation group on $X$, and if $h\in H$ fixes
$x\in X$, then $H$ fixes every point of the $H$-orbit containing
$X$. For, if $k\in H$, then
\[h(k(x)) = k(h(x)) = k(x).\]

Also, if a solution $(X,r)$ has $\mpl(X,r)=2$, and $\Lcal_a(x)=y$, then
$\Lcal_x=\Lcal_y$. For in the first retract, $\Lcal_{[a]}$ is the identity, so
that $[x]=[y]$, which means precisely that $\Lcal_x=\Lcal_y$. It follows
that if $x$ and $y$ lie in the same orbit of $\mathcal{G}(X,r)$,
then $\Lcal_x=\Lcal_y$.

Now suppose that $H=\mathcal{G}(X,r)$ for some solution $(X,r)$ with
$\mpl(X,r)=2$. For any orbit $X_i$, choose $a\in X_i$, and let
$f(X_i)=\Lcal_a$. This element fixes $a$, and hence fixes every point of
its orbit. Moreover, the previous paragraph shows that $f(X_i)$ is
independent of the choice of $a\in X_i$. Also, the image of $f$
consists of all permutations $\Lcal_a$, for $a\in X$; so it generates
$\mathcal{G}(X,r)=H$.

Conversely, suppose that we are given a function $f$ with the
properties in the proposition. For all $a\in X_i$, we define
$\Lcal_a=f(X_i)$, and then construct a map $r:X\times X\to X\times X$ in
the usual way:
\[r(a,b)=(\Lcal_a(b),\Lcal_b^{-1}(a)).\]
By assumption, $\Lcal_a(a)=a$. Also, for any $a,b\in X$, $\Lcal_a(b)={}^ab$
lies in the same orbit as $b$, and hence $\Lcal_{{}^ab}=\Lcal_b$; similarly
$\Lcal_{a^b}=\Lcal_a$; and the fact that the group is abelian now implies
that
\[\Lcal_{{}^ab}\Lcal_{a^b}=\Lcal_b\Lcal_a=\Lcal_a\Lcal_b.\]
It follows that we do have a solution. Moreover, the group $H$ is
generated by all the maps $\Lcal_a$, for $a\in X$; so
$H=\mathcal{G}(X,r)$.
\end{proof}

From this we can deduce the following.

\begin{corollary}
\label{mpl2permgroupscor} Let $H$ be a finite abelian group. Then
there is a solution $(X,r)$ with $\mpl(X,r)=2$ such that
$\mathcal{G}(X,r)\cong H$.
\end{corollary}

\begin{proof} Let $h_1,\ldots,h_r$ generate $H$. Now let
$X=H\cup\{a_1,\ldots,a_r\}$, where $H$ has its regular action on
itself and fixes the points $a_1,\ldots,a_r$. Define a function $f$
by
\begin{eqnarray*}
f(H) &=& \id,\\
f(\{a_i\}) &=& h_i.
\end{eqnarray*}
The conditions of the proposition are obviously satisfied by $f$.
\end{proof}

\subsection{Square-free solutions of multipermutation level 3}
\label{mpl3section}
It is straightforward that $\mpl X > 2$ if and only if $\Ret(X,r)$ is a
nontrivial solution, or equivalently
\[
\exists \, \alpha, x \in X,  \quad \text{such that}\quad
[{}^{\alpha}x] \neq [x].
\]

\begin{remark}
\label{mpl>2cor} The following are equivalent

i) $\mpl(X,r) > 2$

ii) There exists a first-retract class $[x]=[x^{(1)}]$ which is not
$\Gcal$-invariant.

(iii) There exists a $\Gcal$-orbit $X_0$ in $X$   and a pair
$\alpha, \beta \in X_0$ such that $\Lcal_{\alpha}\neq
\Lcal_{\beta}.$
\end{remark}

\begin{proposition}
\label{mpl3prop1} Let $(X,r)$ be a  square-free solution (of
arbitrary cardinality).
\begin{enumerate}
\item
$\mpl(X, r) \leq 3$ if and only if the following condition holds
\begin{equation}
\label{eqstu1}\Lcal_{({}^{\beta}x)}= \Lcal_{({}^{\alpha}x)}\quad
\forall \alpha,\beta, x \in X \quad\text{with}\quad
\Ocal_{\Gcal}(\alpha)= \Ocal_{\Gcal}(\beta).
\end{equation}
\item
$\mpl X = 3,$ if and only if (\ref{eqstu1}) holds,  and there exists a pair $x,
\alpha \in X$ such that
\[
\Lcal_{({}^{\alpha}x)} \neq \Lcal_{x}
\]
\end{enumerate}
\end{proposition}

\begin{proof}
 The following implications hold:
 \[\begin{array}{rl}
\mpl(X, r) \leq 3 &\Longleftrightarrow \Ret^2(X,r) \quad \text{is
either the trivial solution or one element solution}\\
&  \Longleftrightarrow [\alpha] \sim [{}^u{\alpha}] \quad \text{in
$\Ret(X,r)$}\quad\forall
\alpha\in X, u \in \Gcal\\
 &  \Longleftrightarrow [{}^{\alpha}x] =
[{}^{\beta}x]\quad \forall \alpha,\beta, x\in X,
\quad\text{with}\quad
\Ocal_{\Gcal}(\alpha)= \Ocal_{\Gcal}(\beta)\\
&  \Longleftrightarrow \Lcal_{{}^{\alpha}x} =
\Lcal_{{}^{\beta}x}\quad\forall \alpha,\beta, x\in X,
\quad\text{with} \quad\Ocal_{\Gcal}(\alpha)= \Ocal_{\Gcal}(\beta).
\end{array}
\]
This implies the first part of the proposition. Now (1) together
 Remark \ref{mpl>2cor} imply (2).
\end{proof}

\begin{proposition}
\label{mpl3prop2}
 Let $(X,r)$ be a nontrivial square-free solution
with  condition (\ref{eqstu1}). Suppose $X_i, 1 \leq i \leq t,$ is
the set of all $G(X,r)$-orbits in $X$ enumerated so that $X_1,
\cdots , X_{t_0}, (t_0 \geq 1)$ are exactly the nontrivial ones.
Then the following are equivalent.
 Then the following conditions hold:
 \begin{enumerate}
 \item
 \label{mpl3_stu1}
There are  equalities
\begin{equation}
\label{eqstu2}
\begin{array}{rll}
 {}^{{}^{\beta}{x}}{\alpha}= {}^x{\alpha}&\quad\quad
{}^{{}^y{\alpha}}{x}= {}^{\alpha}{x} &\quad\forall x,y \in X_i,
\forall \alpha,\beta \in X_j, \; 1 \leq i < j \leq t.
\end{array}
\end{equation}
\item
\label{mpl3_G^j_abelian} Each group $\Gcal(X_j), 1 \leq j \leq t_0,$
is abelian. In the case when $t_0 < t,$ $\Gcal(X_j) = \{ 1\}$, for
all $t_0 < j \leq t$.
\item \label{mpl3_stu2}
$X$ is \emph{a strong twisted union} $X= X_1\stu X_2 \stu \cdots\stu
X_t$.
\item In particular,
$\mpl(X,r) = 3$ implies (\ref{eqstu1}) and conditions (1), (2), (3).
\end{enumerate}
\end{proposition}

\begin{proof}
We apply the two sides of  (\ref{eqstu1}) to the element $\alpha,$
and use the cyclic condition to yield: $
\Lcal_{{}^{\beta}x}(\alpha)= \Lcal_{{}^{\alpha}x} (\alpha) =
\Lcal_{x} (\alpha)$.  This interpreted in our typical notation
gives:
\[
{}^{{}^{\beta}{x}}{\alpha}= {}^x{\alpha} \quad \forall x \in X,
\alpha, \beta \in X_j, 1 \leq j \leq t,
\]
We have verified the left hand-side equality of (\ref{eqstu2}).
The right hand-side is analogous. This proves part
(\ref{mpl3_stu1}).

To prove   (\ref{mpl3_stu2}) it will be enough to show that for two
arbitrary orbits $X_i, X_j, 1 \leq i < j \leq t$ the set $X_{ij}=
X_i \bigcup X_j$ is $r$-invariant,  and the restricted solution
$(X_{ij}, r_{\mid X_{ij}})$ is a strong twisted union $X_{ij}= X_i
\stu X_j$.

Consider  two orbits $X_i, X_j$, where $1 \leq i < j \leq t$.
As a union of two $\Gcal$-invariant subsets of $X$,  the set
$X_{ij}= X_i \bigcup X_j$ is
$r$-invariant, so (\ref{eqstu2}) implies that the restricted
solution $(X_{ij}, r_{\mid X_{ij}})$ is a strong twisted union
$X_{ij}= X_i \stu X_j$, which proves (\ref{mpl3_stu2}).

Suppose now $1 \leq j \leq t.$ Without loss of generality we can
assume that the group $\Gcal(X_j)$ is nontrivial (by hypothesis
$(X,r)$ is a nontrivial solution). Now in the equality
(\ref{eqstu1}) we set $x = \alpha$ and since $(X,r)$ is square-free
(${}^{\alpha}{\alpha}= \alpha$) we obtain the left hand-side of the
following
\begin{equation}
\label{mpl3eq1}
\begin{array}{rll}
\Lcal_{{}^{\beta}{\alpha}}= \Lcal_{\alpha},& \quad
\Lcal_{{\beta}^{\alpha}}= \Lcal_{\beta}&\quad \forall  \alpha,
\beta \in X_j,\quad 1 \leq j \leq t.
\end{array}
\end{equation}
The equality in the right hand-side is analogous.  Recall that
since  $(X,r)$ is a solution
one has
\[
\Lcal_{\beta}\circ\Lcal_{\alpha} =
\Lcal_{{}^{\beta}{\alpha}}\circ\Lcal_{{\beta}^{\alpha}}\quad
\forall \alpha , \beta \in X,
\]
which together with (\ref{mpl3eq1}) implies
\begin{equation}
\label{mpl3eq2} \Lcal_{\beta}\circ\Lcal_{\alpha}
=\Lcal_{\alpha}\circ\Lcal_{\beta}, \quad \forall  \alpha,  \beta
\in X_j, \quad 1 \leq j \leq t. \end{equation}
By definition,
$\Gcal(X_j)$ is the subgroup of $\Gcal,$ generated by the set of
all $\Lcal_{\alpha}, \alpha \in X_j.$ It follows then from
(\ref{mpl3eq2}) that each nontrivial $\Gcal(X_j)$,  $1 \leq j \leq
t,$ is abelian.
\end{proof}
\begin{corollary}
\label{mpl3_Cor1} Let $(X,r)$ be a finite square-free solution with
$\mpl X = 3$. Let $X_1 \cdots X_t$
 be the set of $\Gcal$-orbits in $X$. Then
\begin{enumerate}

\item
\label{mpl3_G(X_j)_abelian1} Each group $\Gcal(X_j), 1 \leq j \leq t,$
is abelian ($\Gcal(X_j)=\{e\}$ is also possible). There exists a $j,
1 \leq j \leq t$ for which $\Gcal(X_j)$ is nontrivial. Without loss
of generality we shall assume that the nontrivial groups are
$\Gcal(X_j)$ with $1 \leq j \leq t_0.$

\item
\label{mpl3_G(X_j)orbits} Suppose $1 \leq j \leq t_0$, so $\Gcal(X_j)$
acts on $X$ nontrivially. Consider the set of $\Gcal(X_j)$-orbits in
$X$. Then the elements of each $\Gcal(X_j)$-orbit in $X$ act equally
on $X_j$, that is:
\[
(\Lcal_{\alpha})_{\mid X_j}= (\Lcal_{\beta})_{\mid
X_j}\quad\text{whenever}\quad \beta\in \Ocal_{\Gcal(X_j)}(\alpha).
\]
Moreover, $X = X_1\stu X_2\stu \cdots\stu X_t.$
\item
\label{mpl3_Gproduct}  The group $\Gcal$ is a product of abelian
subgroups
\[\Gcal= \Gcal(X_1)\Gcal(X_2)\cdots \Gcal(X_{t_0}).\]
\end{enumerate}
\end{corollary}

This corollary gives interesting information about the graph
$\Gamma$ of finite solutions $(X,r)$ with $\mpl X = 3$.  Take an
arbitrary $j$ for which the group $\Gcal(X_j)$ is nontrivial. The
action of $\Gcal(X_j)$ on $X$ is represented graphically by taking
the subgraph $\Gamma^{(j)}$ of $\Gamma$ with the same set of
vertices, but only edges labeled by $\alpha,$ where $\alpha \in
X_j$, are considered. Then clearly, there is a $1-1$ correspondence
between the set of connected components in $\Gamma^{(j)}$ and the
set of $\Gcal(X_j)$-orbits in $X$. For each component
$\Gamma^{(j)}_x$ ($x$ is an arbitrary vertex in the component) the
corresponding orbit is $\Ocal_{\Gcal(X_j)}(x)$. As a set it
coincides with the set of all vertices in $\Gamma^{(j)}_x.$ It
follows then from Corollary \ref{mpl3_Cor1} that all vertices in a
connected component of $\Gamma^{(j)}, 1 \leq j \leq t_0$ have the same
action on the set $X_j$.

\begin{question}
In notation as above, suppose that $(X,r)$ is a finite square-free
solution which satisfies condition (1) and (2) of Corollary
\ref{mpl3_Cor1}. What additional ``minimal'' condition should be
imposed on the actions (if any) to guarantee $\mpl X = 3$? We are
searching a condition which formally is weaker that the obviously
sufficient condition (\ref{eqstu1}).
\end{question}

\begin{proposition}
\label{theoremC} Let $(X,r)$ be a nontrivial finite square-free
solution of multipermutation level $m, 2 \leq m \leq 3$. Let
\[Y_1 = [\xi_1^{(m-1)}], \cdots, Y_s =[\xi_s^{(m-1)}]\] be the set of
all distinct $(m-1)$-retract classes in $X$.
Let $S(Y_j), G(Y_j), \Gcal(Y_j), 1 \leq j\leq
s,$ be as in Notation
\ref{S(Y)G(Y)def}. Then
\begin{enumerate}
\item
\label{thC1}
$s \geq 2$, $|Y_i|\geq 2$ for some $1 \leq i\leq s$. (We
enumerate them so that the first $s_0$ are exactly the nontrivial
ones).
\item
\label{thC2}
Each retract class $Y_i, 1 \leq i\leq s$, is $\Gcal$-invariant and
the restricted solution $(Y_i, r_i)$ where $r_i=r_{\mid Y_i}$ has
multipermutation level $\leq 2$. (In the case when $s_0 < i \leq s,
\mpl Y_j = 0$).
\item
\label{thC3}
$(X, r)$ is a strong twisted union
\[
X = Y_1\stu Y_2\stu\cdots \stu Y_s.
\]
\item
\label{thC4}
All groups $\Gcal(Y_j), 1 \leq j \leq s,$ are abelian and
$\Gcal(Y_j) = 1$ is possible for at most one retract class $Y_j$ (with $Y_j \subset \ker \Lcal$).
Furthermore,
\[
\Gcal = \Gcal(Y_1)\Gcal(Y_2) \cdots \Gcal(Y_{s_0}).
\]
\end{enumerate}
\end{proposition}

\begin{proof} We give a sketch of the proof for the case $\mpl X=3$, the case $\mpl X=2$ is analogous.
\ref{thC1} is clear. \ref{thC2} follows from Facts \ref{VIPfacts}\ref{fact2}
We shall prove \ref{thC3}. Clearly, $X$ splits into a disjoint union
$X= \bigcup_{1 \leq j \leq s}Y_j.$
of the $\Gcal$-invariant subsets $Y_j, 1 \leq j \leq s$. By Definition \ref{STUmsetsdef}
it will be enough to show that for each pair $i\neq j, 1 \leq i,j \leq s$
the $\Gcal$-invariant subset $Y_{ij}= Y_i\bigcup Y_j$ is a strong twisted union
$Y_{ij}= Y_i \stu Y_j.$
Suppose $\alpha, \beta \in Y_i,$ then $\alpha^(2) = \beta^(2)$ and by Facts \ref{VIPfacts},
(\ref{facteq3}) one has
\[
{}^{{}^{\alpha}{x}} {\beta} = {}^x{\beta} \quad \forall x \in X,
\]
which implies $Y_{ij}= Y_i \stu Y_j.$ This proves \ref{thC3}.

By Corollary \ref{ret2abelianGcorollary} each group $\Gcal(Y_i), 1\leq $ is abelian.
By Remarks \ref{trivialG(Y)remark} $\Gcal(Y_j)= 1$ \emph{iff} $Y_j \subset \ker \Lcal$, and
\ref{trivialG(Y)remark2} implies that if such a retract class exists it contains the set $X_0 = X \bigcap \ker \Lcal,$
and therefore it is unique. Now  \ref{thC4} follows from the decomposition Theorem \ref{decompositionretractclasses} (3).
\end{proof}

We show in section 6  that for square-free solutions of arbitrary
cardinality and finite multipermutation level $m$ one has
\[sl \Gcal(X,r) \leq m-1, \quad sl G(X,r) \leq m \]
Thus
\[
\mpl X = 3 \Longrightarrow 1 \leq sl \Gcal(X,r) \leq 2.
\]

We conclude the section with an example of a square-free solutions
$(X,r)$ with $\mpl(X,r)=3$ and  abelian  YB permutation group
$\Gcal(X,r).$

\begin{example}
\label{gap-mpl-sl-lemma} Let $(X,r)$ be the square-free solution
defined as
\[
\begin{array}{rl}
X= X_1 \bigcup X_2\bigcup X_3&\quad  \quad\\
 & \\
X_1 =\{x_i\mid 1 \leq i \leq 8\}\quad& X_2= \{a, c\}\quad X_3= \{b, d\}\\
\\
\Lcal_a = (b\;d)(x_1\;x_2)(x_3\;x_4)(x_5\;x_6)(x_7\;x_8)\quad&\Lcal_c = (b\;d)(x_1\;x_5)(x_2\;x_6)(x_3\;x_7)(x_4\;x_8)\\
& \\
\Lcal_b =
(a\;c)(x_1\;x_3)(x_2\;x_4)(x_5\;x_7)(x_6\;x_8)\quad&\Lcal_c =
(a\;c)(x_1\;x_8)(x_2\;x_7)(x_3\;x_6)(x_5\;x_4)\\&
\\
\Lcal_{x_i} = \id_{X}\quad 1 \leq i \leq 8.
\end{array}
\]
Then  $\mpl(X,r)=3$, $\Gcal(X,r)$ is abelian, and the group
$G(x,r)$ is solvable of solvable length $2$.
\end{example}

\begin{proof}
Note that ${}^ab = d,\; a^b=c$, so $\Lcal_a\Lcal_b=
\Lcal_{{}^ab}\Lcal_{a^b}.$ Similarly one veryfies that
$\Lcal_z\Lcal_t= \Lcal_{{}^zt}\Lcal_{z^t},$ for all $z,t \in X$,
so condition {\bf l1} is satisfied and the left action given above
defines a solution. Direct computation shows:
\[\begin{array}{rl}
\Lcal_a \circ\Lcal_b= \Lcal_b \circ\Lcal_a =\Lcal_d\circ\Lcal_c
=\Lcal_c\circ\Lcal_d =&
(a\;c)(b\;d)(x_1\;x_4)(x_2\;x_3)(x_5\;x_8)(x_6\;x_7)
\\
\\
\Lcal_a \circ\Lcal_c= \Lcal_c \circ\Lcal_a = \Lcal_b\circ\Lcal_d
=\Lcal_d\circ\Lcal_b =& (x_1\;x_6)(x_2\;x_5)(x_3\;x_8)(x_4\;x_7)
\\
\\
\Lcal_a \circ\Lcal_d= \Lcal_d \circ\Lcal_a = \Lcal_b\circ\Lcal_c
=\Lcal_c\circ\Lcal_b =&
(a\;c)(b\;d)(x_1\;x_7)(x_2\;x_8)(x_3\;x_5)(x_4\; x_6).
\end{array}
\]
Hence $\Gcal= \Gcal(X,r)= \langle \Lcal_a,\Lcal_b,\Lcal_c,\Lcal_d
\rangle$ is abelian. Next, the equality $\Lcal_{{}^ab}=\Lcal_d
\neq \Lcal_b,$ implies that $\mpl(X,r) \geq 3.$ It is easy to see
that
\[
\Gcal=\langle \Lcal_a \rangle \times \langle \Lcal_b \rangle
\times \langle \Lcal_c \rangle \simeq C_2\times C_2\times C_2.
\]
We will show that $\mpl(X,r)=3.$ For the retracts one has
\[\begin{array}{rl}
\Ret(X,r) =& ([X], r_{[X]}),\;\text{where}\; [X]=\{[a],[b], [c],
[d], [x_1]\}, \\
\\
&\Lcal_{[a]} =\Lcal_{[c]}= ([b]\;[d]), \Lcal_{[b]} =\Lcal_{[d]}=
([a]\;[c])
\\
\\
\Ret^2(X,r) =& ([[X]], r_{[[X]]}),\;\text{where}\; [[X]]=\{[a],[b],
[x_1]\},  \Lcal_{[[a]]} =\Lcal_{[[b]]}= \Lcal_{[[x_1]]}= e
\\
\\
\Ret^3(X,r)\quad& \text{is the one element solution on}
\quad\{[a]\}.
\end{array}
\]
So we have  $\mpl(X,r) = 3$. We have seen that  $\Gcal$ is abelian,
or equivalently,  $\sol\Gcal= 1$. It is easy to see that  $\sol
G(X,r)= 2.$ Indeed, $\mpl (X, r) = 3$ implies $2 \leq\sol G(X,r)$,
by Theorem \ref{sltheorem} $\sol G(X,r) \leq\sol \Gcal +1 =2$.

\end{proof}
\begin{question}
\label{boundmplq} Are there multipermutation square-free solutions
$(X,r)$ of arbitrarily high multipermutation level, and with abelian
permutation group $\Gcal$? If not, what is the largest integer $M$
for which there exist solutions  $(X,r)$ with $\mpl(X,r)=M$ and
$\Gcal$ abelian.
\end{question}

We show, see  Theorem \ref{theoremA}, that assuming $\Gcal(X,r)$
abelian, one has $\mpl(X,r)\leq t$, where $t$ is the number of
$\Gcal$-orbits in $X$. We
still do not have examples of solutions with high multipermutation
level and $\Gcal$ abelian.

\section{Solutions with abelian permutation group}
\label{abelian permutation group}
We can say a surprising amount about solutions $(X,r)$ for which
$\Gcal(X,r)$ is abelian.  In this section we  keep the notation and
conventions from the previous sections.

As usual $(X,r)$ is a square-free solution of arbitrary cardinality,
$\Gcal=\Gcal(X,r)$ denotes its YB permutation group. In the cases
when we assume $X$ finite this will be written explicitly. Without
restriction to the case of necessarily finite solutions we shall assume
that $X$ has a finite number of $\Gcal$- orbits, as in the previous
sections, $\Xcal = \{X_1, \cdots, X_t \}$ will denote the set of
$\Gcal$-orbits in $X$. Clearly, in the case when $X$ is finite this
condition is always in force. As discussed in the previous sections,
each $(X_i, r_i), 1 \leq i \leq t$, is also a square-free solution,
where $r_i$ is the restriction $r_i= r_{\mid X_i}$. The main results
of the section are the following theorems.

\begin{theorem}
\label{theoremA} Let $(X,r)$ be a square-free solution of arbitrary
cardinality. Suppose its permutation group $\Gcal$ is abelian, and
$X$ has a finite number of $\Gcal$- orbits,  $X_1, \cdots, X_t$,
where $t \geq 2$. Then the following conditions hold.
\begin{enumerate}
\item
\label{theoremA1} Each $(X_i, r_i), 1 \leq i \leq t$,  is a trivial
solution.
\item
\label{abelianGth2} $(X,r)$ is a strong twisted union of its
$\Gcal$-orbits (see Definition \ref{STUmsetsdef}):
\[
X = X_1\stu X_2\stu \cdots\stu X_t.
\]
\item
\label{abelianGth3} $(X,r)$ is a multipermutation solution, with
\[
\mpl (X,r) \leq t.
\]
\end{enumerate}
\end{theorem}

\begin{corollary}
Every finite  square-free solution $(X,r)$ with abelian permutation
group $\Gcal$ has multipermutation level $\mpl (X,r) \leq t,$ where
$t$ is the number of its $\Gcal$- orbits.
\end{corollary}

\begin{theorem}
\label{theoremB} Let $(X,r)$ be  a square-free solution with abelian
permutation group $\Gcal$. Suppose  $X$ is a strong twisted union,
$(X, r)  = X_1\stu X_2$ of the solutions $(X_1, r_{X_1})$, and
$(X_2, r_{X_2}).$  Then the three solutions are multipermutation
solutions and
\[
\mpl X \leq \mpl X_1 + \mpl X_2.
\]
\end{theorem}

\begin{remark}
Note that using a
different argument Ced\'{o}, Jespers, and
Okni\'{n}ski, \cite{Eric2009}, have proven that each finite
square-free solution $(X,r)$ with $\Gcal$ abelian  is retractable
(and therefore multipermutation solution),  but do no give
estimation of the multipermutation level.
\end{remark}

\begin{remark}
Theorem \ref{theoremA} confirms Conjectures I and II \cite{T04} in
the case when the finite square-free solutions $(X,r)$ has abelian
permutation group $\Gcal(X,r)$, or equivalently its YB group
$G(X,r)$ has solvable length 2.
\end{remark}

\begin{remark}
\label{m&t_independent_remark} Note that in the general case there
is no relation between $\mpl X$ and the number of orbits $t=t(X)$.
As show Lemma \ref{canonicalextensionlemma}, and Theorem
\ref{beautifulconstruction}, for every integer $m \geq 3$ there
exist finite square-free solutions $(X,r)$  with exactly 2 orbits
and with $\mpl X = m$.
\end{remark}
The proofs of the main results in this section heavily rely on a
necessary and   sufficient condition for $\mpl (X,r) = m$ given by
Theorem \ref{mplmtheorem} in the next section. This condition is given
in terms of \emph{long actions}, or as we call them informally
\emph{towers of actions}. We need to develop first some basic
technique for computation with towers of actions. We end the section
with a simple construction \emph{doubling of solutions} which is
also used to illustrate that a solutions $(X,r)$ may have exactly
two orbits and arbitrarily large finite multipermutation level $m=
\mpl X$.

\begin{definition}
\label{canonicalextensiondef} Let $(X,r_X)$ and $(X^{\prime},
r_{X^{\prime}})$ be a disjoint identical finite square-free solution, where $X
= \{x_1 \cdots x_n\}$, $X^{\prime} = \{x_1^{\prime} \cdots
x_n^{\prime}\}$. Let  $ Y=\{\alpha\}$ be a one element set disjoint
with $X\bigcup X^{\prime}$, and let $(Y, r_0)$ be the one element
trivial solution. (By definition one has $\mpl Y = 0$). Consider the
extension of solutions $(Z, r)$
\[
Z = (X \stu _0 X^{\prime})\stu\{ \alpha\},
\]
where  $r$ is an extension of the YB maps, $r_X, r_{X^{\prime}},
r_Y$ defined (as usual) via the canonical isomorphism of solutions
$\Lcal_{\alpha} = (x_1 \;x_1^{\prime})\cdots (x_n \;x_n^{\prime})$
\[\begin{array}{rl}
r_{\mid X}= r_X,\quad  r_{\mid X^{\prime}}= r_{X^{\prime}}
\\
r(x_i, x_j^{\prime}) = (x_j^{\prime}, x_i)\quad 1 \leq i, j \leq n
\\
r(\alpha, x_j)  = (x_j^{\prime}, \alpha),\quad  r(\alpha,
x_j^{\prime}) = (x_j, \alpha).
\end{array}
\]
We call $(Z,r)$ a \emph{canonical doubling} of $(X,r_X)$,
and denote it  $Z=X^{[2,\alpha]}$.
\end{definition}

In the following lemma, $\wr$ denotes the wreath product of a group with
a permutation group. We discuss wreath products more extensively below.

\begin{lemma}
\label{canonicalextensionlemma} Let $(X,r_X)$ be a square-free
solution with $\mpl X = m$. Let $(Z, r)=X^{[2,y]}$ be a canonical
doubling of $(X,r_X)$.  Denote $G_X=G(X,r_X), G_Z=G(Z,r).$ Then
\begin{enumerate}
\item
$\mpl (Z, r) = m+1$ and clearly, $(Z, r)$ has exactly two $\Gcal(Z,
r)$ orbits, namely $Z_1 = X \bigcup X^{\prime}$, and $Z_2 = \{y\}$.
\item
There is an isomorphism of groups
\[
G_Z \simeq (G_X\times G_X)\rtimes C_\infty \simeq G_X\wr C_\infty,
\]
where the generator of the infinite cyclic group interchanges the
two factors.
\item
There is an isomorphism of groups
\[
\Gcal(Z)\simeq (\Gcal_X\times\Gcal_X)\rtimes C_2 \simeq \Gcal_X\wr C_2.\]
\end{enumerate}
\end{lemma}

The proof is straightforward.

Clearly, $\Lcal_{\alpha}\in \Aut(X \natural_0 X^{\prime})$ but does
not belong to the permutation group $\Gcal (X \stu _0
X^{\prime})$, so  Lemma \ref{canonicalextensionlemma} is a
particular case of Lemma \ref{raisingmpl}.

\begin{remark}
It is straightforward to see that
\[\sol(G_Z)=\sol(G_X)+1,\qquad \sol(\Gcal_Z)=\sol(\Gcal_X)+1,\]
where $\sol(G)$ denotes the solvable length of the group  $G$.
We will see more general results later.
\end{remark}

\subsection{Computations with actions in $(X,r)$ }
\label{Computations with actions }
In cases when we have to write a sequence of successive actions we
shall use one also well known notation
\begin{equation}
\label{la} \alpha\la x = {}^{\alpha}{x}
\end{equation}

\begin{definition} Let
$\zeta_1, \zeta_2,\cdots, \zeta_m \in X.$ The expression
\[
\omega=
(\cdots((\zeta_m \la \zeta_{m-1})\la\zeta_{m-2})\la
\cdots\la\zeta_2)\la\zeta_1
\]
will be called \emph{a tower of
actions } or shortly \emph{a tower}.
\end{definition}
Clearly, the result of this action has the shape $\;
{}^{u}{\zeta_1}$, where
\[
u =(\cdots((\zeta_m \la
\zeta_{m-1})\la\zeta_{m-2})\la \cdots\la\zeta_3)\la \zeta_2,
\] so it
belongs to the $\Gcal$-orbit of $\zeta_1$.

The following two remarks and lemma are straightforward and hold for
the general case of square-free solutions, where $\Gcal(X,r)$ is not
necessarily abelian, and $X$ is of arbitrary cardinality.

\begin{remark}
\label{stringsintwoalphabetsR} Let $\Sigma_1, \Sigma_2$ be two
disjoint alphabets, $m_1,m_2$ be positive integers, $m=m_1+m_2+1$.
Let $\omega= \zeta_m\zeta_{m-1}\cdots\zeta_2\zeta_1$ be a string
(word) of length $m$ in the alphabet $\Sigma_1\bigcup \Sigma_2$.Then
one of the following conditions is satisfied

i) $\omega$ contains a segment $v$ of the shape $v=\beta y_{q}\cdots
y_2y_1\alpha,$ where $q \geq 1, \; y_k \in \Sigma_j, 1 \leq k \leq
q$, $\alpha,\beta \in \Sigma_i,$ and $1 \leq i\neq j\leq 2$

ii)  $\omega=y_{q}\cdots y_2y_1\alpha_{p}\cdots \alpha_2\alpha_1$,
where $y_s \in \Sigma_j, 1\leq s \leq q$, $\alpha_k \in \Sigma_i,
1\leq k \leq p,$ and either $p \geq m_i+1,$ or $q \geq m_j+1,$
($p=0, q = m_1+m_2+1$, or $p = m_1+m_2+1,q=0$ is also possible).
\end{remark}
This remark has a transparent but very useful interpretation for
towers of actions.

\begin{remark}
\label{towerintwoalphabetsR} Let $X_1, X_2$ be disjoint subsets of
the solution $(X,r),$ $m_1,m_2$ be positive integers, $m=m_1+m_2+1$.
Let $\zeta_1, \zeta_2, \cdots\zeta_m \in X_1\bigcup X_2,$ and
$\omega$ be the tower:
\[
\omega= (\cdots(\zeta_m\la\zeta_{m-1})\la
\cdots\la\zeta_2)\la\zeta_1
\]
Then either

i) $\omega$ contains a segment $(((\cdots\la\beta\la
y_{q})\la\cdots\la y_2)\la y_1)\la\alpha,$ where $q \geq 1, y_k \in
X_i, 1 \leq k \leq q$, $\alpha,\beta \in X_j,$ and $1 \leq i\neq
j\leq 2$; or

ii) $\omega$ has the shape
\[
\omega=(\cdots(((\cdots(y_{q}\la y_{q-1})\la\cdots \la
y_2)y_1)\la\alpha_{p})\cdots \alpha_2)\la\alpha_1,
\]
where $y_s \in X_j, 1\leq s \leq q$, $\alpha_k \in X_i, 1\leq k \leq
p$,  with $1\leq i\neq j\leq 2$. Furthermore, either $p \geq m_i+1,$
or $q \geq m_j+1,$ ($p=0, q = m_1+m_2+1$, or $q = m_1+m_2+1,q=0$ is
also possible).
\end{remark}

\begin{lemma}
\label{cuttowersL} Let $(X,r)$  be   square-free solution, $y_1,
\cdots y_k \in X$, with $k \geq 1$, and let $Z$ be an $r$-
invarionat subset of $X$. Suppose there is an equality
\begin{equation}
\label{formula2} ((\cdots((\alpha\la y_k)\la y_{k-1})\la \cdots
y_2)\la y_1) = (\cdots( y_{k}\la y_{k-1})\la \cdots y_2)\la y_1
\quad \forall \alpha \in Z
\end{equation}
Then any longer tower
\begin{equation}
\label{formula3} \omega=(\cdots(((\cdots(((((\cdots( a_s\la
a_{s-1})\la \cdots )\la a_1)\la \alpha)\la y_k)\la y_{k-1})\la
\cdots \la y_2)\la y_1)\la b_p)\la\cdots \la b_2)\la b_1
\end{equation}
with $a_1, \cdots a_s, b_1,\cdots b_p \in X$ and $\alpha \in Z$, can
be simplified by ``cutting'' the leftmost sub-tower of length $s+1$,
that is there is an equality:
\begin{equation}
\label{formula4} \omega = \omega^{\prime}= ((\cdots(((\cdots(y_k\la
y_{k-1})\la \cdots \la y_2)\la y_1)\la b_p)\la\cdots \la b_2)\la b_1.
\end{equation}
In the particular cases $s =0$ (respectively  $p= 0$) the $a$'s,
(respectively the $b$'s) are simply missing in the formulae above.
\end{lemma}
The lemma follows straightforwardly from the clear implication
\[
\alpha\in Z \Longrightarrow ((\cdots (a_s\la a_{s-1})\la \cdots )\la
a_1)\la \alpha  \in Z
\]

We shall refer (informally) to the procedure described in Lemma
\ref{cuttowersL} as \emph{truncation}.

From now on  we assume that $\Gcal(X,r)$ is abelian.

\begin{lemma}
\label{favouriteL} Let $(X,r)$  be a square-free solution with
abelian permutation group $\Gcal$. Suppose $Y,Z$ are
$r$-invariant subsets of $X$ and the following \textbf{stu}-type
condition is satisfied:
\begin{equation}
\label{vipbase1} {}^{{}^{\alpha}y}{z}={}^{y}z\quad \forall \alpha,z
\in Z \quad\text{and}\quad \forall y \in Y.
\end{equation}
Then for every pair $\alpha, z \in Z,$ and every finite sequence
$y_1, \cdots, y_k \in Y$, $k \geq 1$ one has
\begin{equation}
\label{vipformula1} ((\cdots(\alpha\la y_k)\la y_{k-1})\la \cdots \la
y_2)\la y_1)\la z = ((\cdots( y_{k}\la y_{k-1})\la \cdots \la y_2)\la
y_1)\la z.
\end{equation}
\end{lemma}

\begin{proof}
We shall prove (\ref{vipformula1}) using induction on $k$. Clearly,
(\ref{vipbase1}) gives the base for the induction. Assume the
statement of the lemma is true  for $k-1$ where $k>1$. Suppose
$\alpha, z \in Z, y_1, \cdots y_k \in Y.$ For convenience we
introduce the elements $y_{k-1}^{\prime}, y_{k-1}^{\prime\prime}$ as
follows:
\begin{equation}
\label{vipeqnot} y_{k-1}^{\prime} =y_{k}\la y_{k-1}
={}^{y_k}{y_{k-1}} \quad \quad y_{k-1}^{\prime\prime} =
({\alpha}\la{y_k})\la y_{k-1}= {}^{({}^{\alpha}{y_k} )}{y_{k-1}}
\end{equation}
Then the following equalities hold.
\[
\begin{array}{rll}
\alpha\la y_{k-1}^{\prime}= &{}^{\alpha}{({}^{y_k}{y_{k-1}})}
\quad &\text{ (\ref{vipeqnot})}
 \\
&={}^{{}^{\alpha}{y_k}}{({}^{{\alpha}^{y_k}}{y_{k-1}})} \quad &
\text{ \textbf{l1}} \\
&={}^{{\alpha}^{y_k}}{({}^{{}^{\alpha}{y_k} }{y_{k-1}})} \quad &
\text{ $\Gcal$ abelian}\\
&={}^{{\alpha}^{y_k}}{((\alpha \la y_k) \la y_{k-1})} \quad &\\
&=({\alpha}^{y_k})\la y_{k-1}^{\prime\prime}\quad  & \text{ (\ref{vipeqnot})}
\end{array}
\]
Thus
\begin{equation}
\label{vipeq1} \alpha\la y_{k-1}^{\prime}=({\alpha}^{y_k})\la
y_{k-1}^{\prime\prime}.
\end{equation}
Now consider the equalities
\[
\begin{array}{rll}
&((\cdots((\alpha\la y_k)\la y_{k-1})\la \cdots y_2)\la y_1)\la z \quad &\\
=& ((\cdots( y_{k-1}^{\prime\prime}\la y_{k-2})\la \cdots \la
y_2)\la
y_1)\la z \quad &  \text{ (\ref{vipeqnot})}\\
=& ((\cdots((( {\alpha}^{y_k})\la y_{k-1}^{\prime\prime})\la
y_{k-2})\la \cdots \la y_2)\la
y_1)\la z \quad & \text{ ${\alpha}^{y_k}\in Z$, and by IH}\\
=&((\cdots(( \alpha\la y_{k-1}^{\prime})\la y_{k-2})\la \cdots \la
y_2)\la y_1)\la z \quad &
\text{ (\ref{vipeq1})} \\
=&((\cdots(y_{k-1}^{\prime}\la y_{k-2})\la \cdots\la y_2)\la y_1)\la
z \quad  &
\text{ by IH} \\
=& ((\cdots((y_{k}\la y_{k-1}) \la y_{k-2})\la \cdots\la y_2)\la
y_1)\la z \quad  & \text{ (\ref{vipeqnot})}
\end{array}
\]
where IH is the inductive assumption.
This proves the Lemma.
\end{proof}

\begin{remark}
Note that in the hypothesis of Lemma \ref{favouriteL} we do not
assume that the sets $Y,Z$ are disjoint. Furthermore the stu-type
condition is not imposed symmetrically on both sets, i.e. even if
$Y$ and $Z$ are disjoint we do not assume that necessarily $Y \stu Z$.
\end{remark}

\subsection{Proofs of the theorems}

\begin{lemma}
\label{gabelianlemmaA} Let $(X,r)$ be a square-free solution, with
abelian permutation group  $\Gcal$. Then the following two
conditions hold.
\begin{enumerate}
\item
Let $Y$  be a $\Gcal$-orbit of $X$. Then for
 any $x \in Y$ one has
$(\Lcal_{x})_{\mid Y}= \id_{ Y}$
\item
Suppose $Y, Z$ are two distinct $\Gcal$-orbits of $X$, $(Y, r_Y),
(Z, r_Z)$ are the canonically induced solutions on $Y$ and $Z$. Then
the actions satisfy the \textbf{stu} condition:
\[
\begin{array}{lcl}
(\Lcal_{{}^x{\alpha}})_{\mid Y}=(\Lcal_{\alpha})_{\mid Y};\quad &
(\Lcal_{{}^{\alpha}x})_{\mid Z}=(\Lcal_{x})_{\mid Z}\quad\text{for
all}\quad x\in Y, \alpha \in Z
\end{array}
\]
Moreover, the induced solution $(T, r_T)$ on the union $T= Y\bigcup
Z$ is a strong twisted union $T= Y\stu Z.$
\end{enumerate}
\end{lemma}

\begin{proof}
Let $x \in Y.$ To prove (1) it will be enough to show
\begin{equation}
\label{lemmaeq1}
 {}^{x}{({}^tx)}= {}^tx, \;  \forall t \in X.
\end{equation}

Now the equalities
\[
\begin{array}{rll}
{}^{x}{({}^tx)}&={}^{t}{({}^xx)} \quad & \text{ $\Gcal$
abelian}\\
&={}^{t}{x} \quad  & \text{ $(X, r)$ square-free }.
\end{array}
\]
imply (1). Assume now that $x,y\in Y, \alpha, \beta \in Z.$
We have to show
\begin{equation}
\label{lemmaeq2}
 {}^{{}^{y}{\alpha}}{x}= {}^{\alpha}{x}, \quad\quad   {}^{{}^{\beta}{x}}{\alpha}=
 {}^{x}{\alpha}.
\end{equation}
 Consider the equalities
\[
\begin{array}{rll}
{}^{\alpha}{x}=&{}^{y}{({}^{\alpha}{x})} \quad & \text{ $y,
{}^{\alpha}{x} \in Y$ and (1)}\\
=&{}^{({}^{y}{\alpha})}{({}^{y^{\alpha}}{x})} \quad & \text{ by \textbf{l1}} \\
&{}^{({}^{y}{\alpha})}{x} \quad &\text{ $x, y^{\alpha} \in Y$ and
(1)}
\end{array}
\]
We have shown the left hand side equality in (\ref{lemmaeq2}).
Analogous argument gives the remaining equality.
\end{proof}

\begin{corollary}
\label{gabeliancorA} Let $(X,r)$ be a square-free solution of
arbitrary cardinality and with abelian permutation group  $\Gcal$.
Suppose $\Gcal$ acts non transitively on $X$ and splits it into a
finite number of $\Gcal$- orbits $X_1, \cdots X_t,\; t \geq 2.$ Then
each $(X_i, r_i)$ is a trivial solution and $X$ is a strong twisted
union
\[
X= X_1\stu X_2\stu\cdots\stu X_t.
\]
\end{corollary}

\begin{proposition}
\label{gabelianpropA} Under the hypothesis of Theorem
\ref{theoremA}. $(X,r)$ is multipermutation solution with $\mpl X
\leq t,$ where $t$ is the number of $\Gcal$-orbits of $X$.
\end{proposition}

\begin{proof}
By Theorem \ref{mplmtheorem} it will be enough to show that for each
choice of $y_1, \cdots y_{m+1}\in X$ there is an equality
\begin{equation}
\label{mplmeq} \omega := (\cdots((y_{m+1}\la y_{m})\la y_{m-1})\la
\cdots \la y_2)\la y_1 = (\cdots( y_{m}\la y_{m-1})\la \cdots \la y_2
)\la y_1=: \omega^{\prime}.
\end{equation}
Clearly, since the orbits are exactly $m$, there will be some $1
\leq \lambda< \lambda+\mu\leq m+1$, such that $y_{\lambda},
y_{\lambda+\mu}$ are in the same orbit, say $X_i$.

Two cases are possible.

\textbf{Case 1.} $\mu = 1.$ In this case,$\lambda+\mu= \lambda+1$
$(y_{m+1}\la y_{m})\la \cdots \la y_{\lambda+1}) = u \in X_i$ and
since $X_i$ is a trivial solution, one has $(\cdots (y_{m+1}\la
y_{m})\la \cdots \la y_{\lambda+1})\la y_{\lambda} =
{}^u{y_{\lambda}}=y_{\lambda}$, thus
\[
\omega=(\cdots((y_{m+1}\la y_{m})\la y_{m-1})\la \cdots \la y_2)\la
y_1 =(\cdots(y_{\lambda}\la y_{\lambda-1})\la \cdots \la y_2)\la y_1
\]
Similarly,
\[
\omega^{\prime}=(\cdots( y_{m}\la y_{m-1})\la \cdots \la y_2)\la y_1
=(\cdots(y_{\lambda}\la y_{\lambda-1})\la \cdots \la y_2)\la y_1
\]
So $\omega=\omega^{\prime}$, which  proves (\ref{mplmeq}).

\textbf{Case 2.} $\mu > 1$. In this case we set
$y_{\lambda+\mu}=\alpha$, $y_{\lambda} = z$. The tower $\omega$
contains the segment $(\cdots (\alpha\la y_{\lambda+\mu-1})\la
\cdots \la y_{\lambda+1})\la z,$ with $\alpha, z \in X_i$. By Lemma
\ref{favouriteL} we can cut $\alpha$ from the tower to yield
\[
(\cdots (\alpha\la y_{\lambda+\mu-1})\la \cdots \la
y_{\lambda+1})\la z = (\cdots(y_{\lambda+\mu-1}\la
y_{\lambda+\mu-2}) \cdots \la y_{\lambda+1})\la z.
\]
We shall assume $\lambda+\mu < m+1$ (The proof in the case
$\lambda+\mu = m+1$ is analogous). By Lemma \ref{cuttowersL} there
are equalities
\[
\begin{array}{rl}
\omega=&(\cdots((\cdots((\cdots(y_{m+1}\la y_{m})\la\cdots
\alpha)\la y_{\lambda+\mu-1})\la \cdots \la y_{\lambda+1})\la
z)\la\cdots )\la y_1\\
=&(\cdots((\cdots(y_{\lambda+\mu-1}\la y_{\lambda+\mu-2})\cdots \la
y_{\lambda+1})\la z)\la\cdots )\la y_1\\
=&(\cdots((\cdots((\cdots(y_{m}\la y_{m-1})\la\cdots \alpha)\la
y_{\lambda+\mu-1})\la \cdots \la y_{\lambda+1})\la z)\la\cdots )\la
y_1.
\end{array}
\]
The proposition has been proved.
\end{proof}
\textbf{Theorem \ref{theoremA}} follows straightforwardly from
Corollary \ref{gabeliancorA} and Proposition \ref{gabelianpropA}.

We shall now prove \textbf{Theorem \ref{theoremB}}. Suppose the
square free solution $(X,r)$ is a strong twisted union $(X, r)  =
X_1\stu X_2$ and $\Gcal(X,r)$ is abelian. Then by Theorem
\ref{theoremA} $(X, r)$,$(X_1, r_{X_1}), (X_2, r_{X_2})$ are
multipermutation solutions. Let $\mpl X_1 = m_1$, $\mpl X_2 = m_2$. We
claim that $\mpl X \leq m_1 +m_2.$ Denote $m = m_1 + m_2 + 1.$ By
Theorem \ref{mplmtheorem} it will be enough to show that for any
choice of $\zeta_1, \cdots \zeta_m \in X$ there is an equality
\begin{equation}
\label{TBeq1} \omega= (\cdots(\zeta_m\la\zeta_{m-1})\la
\cdots\la\zeta_2)\la\zeta_1 = \omega^{\prime}=
(\cdots(\zeta_{m-1}\la\zeta_{m-2})\la \cdots\la\zeta_2)\la\zeta_1
\end{equation}
By Remark \ref{towerintwoalphabetsR} two cases are possible.

\textbf{ Case 1.} $\omega$ contains a segment
$(((\cdots\la\beta\la y_{q})\la\cdots\la y_2)\la y_1)\la\alpha,$
where $q \geq 1, y_k \in X_i, 1 \leq k \leq q$, $\alpha,\beta \in
X_j,$ and $1 \leq i\neq j\leq 2$. Since $X$ is strong twisted union
of $X_1$ and $X_2$ the hypothesis of  Lemma \ref{favouriteL} is in
force, and therefore
\[
((\cdots(\beta\la y_{q})\la\cdots\la y_2)\la y_1)\la\alpha =
((\cdots( y_{q}\la y_{q-1})\la \cdots\la y_2)\la y_1)\la\alpha
\]
Now apply Lemma \ref{cuttowersL} to deduce (\ref{TBeq1})

\textbf{ Case 2.} $\omega$ has the shape
\[
\omega=(\cdots(((\cdots(y_{q}\la y_{q-1})\la\cdots \la
y_2)\la y_1)\la\alpha_{p})\cdots \alpha_2)\la\alpha_1,
\]
where $y_s \in X_j, 1\leq s \leq q$, $\alpha_k \in X_i, 1\leq k \leq
p$,  with $1\leq i\neq j\leq 2$. Furthermore, either $p \geq m_i+1,$
or $q \geq m_j+1,$ ($p=0, q = m_1+m_2+1$, or $q = m_1+m_2+1,q=0$ is
also possible). Without loss of generality we may assume $q \geq
m_j+1.$ But $\mpl X_j = m_j$, so Theorem \ref{mplmtheorem} implies the
equality
\[
(\cdots(y_{q}\la y_{q-1})\la\cdots \la y_2)\la y_1 = (\cdots(y_{q-1}\la
y_{q-2})\la\cdots \la y_2)\la y_1.
\]
We apply again  Lemma \ref{cuttowersL} to obtain (\ref{TBeq1}). The
case when $p \geq m_i+1$ is analogous and we leave it to the reader.

Theorem \ref{theoremB} has been proved. Q.E.D.

\section{Multipermutation solutions of finite multipermutation level}
\label{generalmplsection}
\subsection{General results}
\label{General results}
We shall use the notation from section 2. We  also recall some
notions and basic facts from \cite{T04}.

\begin{notation}
\label{notret0} Let $(X,r)$ be a  square-free solution. For each
integer $k \geq 0$ as usual, we shall use  following notation.
\begin{enumerate}
\item
 $\Ret^k(X,r)$ denotes the $k$-th retract of $(X,r)$, but
 when $k=1$ it is convenient to
 use both notations
$\Ret(X,r)=\Ret^1(X,r)$ and $([X], r_{[X]})$ for the retract.
For completeness we set $\Ret^0(X,r)=(X,r)$.
\item
$x^{(k)}$ denotes the image of $x$ in $\Ret^k(X,r).$ The set
 \[
 [x^{(k)}]: = \{\xi \in X \mid x^{(k)}=\xi^{(k)}\}
\]
is called \emph{the kth retract class of $x$}. (In \cite{T04} it is
referred to as \emph{the kth retract orbit of $x$}).

\item
In the case  when $\mpl (X, r) = m < \infty,$ and $X$ has a finite number of
$\Gcal$-orbits, we let these orbits be
\[X_1, \ldots, X_t.\]
\item
We fix a notation for the elements of the $(m-1)$th retract:
\[\Ret   ^{m-1}(X,r) = \{  \zeta_1^{(m-1)}, \cdots , \zeta_s^{(m-1)} \}.\]
(By Corollary \ref{intransitiveactionforinfiniteXCor} one has $s \leq t$).
\item
The  $m-1$th retract classes will be denoted by
\[
Y_i : = [\zeta_i^{(m-1)}], \; 1 \leq i \leq s.
\]
For each $i, 1 \leq i \leq s$
we denote the set of all $\Gcal$-orbits of $X$ which intersect $Y_i$ nontrivially by
\[
X_{i1}, X_{i2}, \cdots , X_{it_i}.
\]
\end{enumerate}
\end{notation}

\begin{remark}
In the above notation, suppose that $\mpl (X, r) = m < \infty,$ and that
$X$ has a finite number of $\Gcal$-orbits, say
$X_1, \ldots, X_t$.
Then by Corollary \ref{intransitiveactionforinfiniteXCor}  $\Ret^{m-1}$ is a finite set of order $s, 2 \leq s \leq t$.

Furthermore, it follows from Proposition \ref{vip_retrprop} that
for each pair $i, j 1 \leq i \leq s, 1 \leq j \leq t$
one has
\[
Y_i\bigcap X_j\neq \emptyset \Longleftrightarrow X_j \subseteq Y_i.
\]
so each $(m-1)$th retract class $Y_i, 1 \leq i \leq s$ is a disjoint union of
the set of all
$\Gcal$-orbits which intersect it nontrivially:
\[
Y_i = \bigcup_{1 \leq k \leq t_i} X_{ik}.
\]
\end{remark}

The following corollary is  straightforward from Lemma
\ref{retractlemma}.

\begin{corollary}
For each integer $k\geq 1$ the canonical map $(X,r)\longrightarrow
\Ret^k(X,r)$, $x \mapsto x^{(k)}$, is a homomorphism of solutions.
\end{corollary}

The following results are extracted from \cite{T04}, where they are stated
for finite square-free solutions $(X,r)$. However, the argument does not
rely on the finiteness of $X$.
\begin{facts}
\label{VIPfacts} \cite{T04}
\begin{enumerate}
\item
\label{fact1}  \cite{T04}, Lemma 8.10. For every $\alpha,\beta, x
\in X,$ and $k$ a positive integer,
\begin{equation}
\label{facteq1} \alpha^{(k)}= \beta^{(k)}\Longrightarrow
({}^{\alpha}{x})^{(k-1)}=({}^{\beta}{x})^{(k-1)}.
\end{equation}
In particular,
\begin{equation}
\label{facteq2} \alpha^{(2)}= \beta^{(2)}\Longrightarrow
{}^{\alpha}{x}\sim {}^{\beta}{x} \quad \forall x \in X, \quad
{}^{\alpha}{\beta}\sim \beta.
\end{equation}
\item
\label{fact2}  \cite{T04}, Lemma 8.9. It follows (\ref{facteq1})
that  for any positive integer $k$, and any $x \in X,$ the
restriction $r_{x,k}$ of  $r$ on $[x^{(k)}]$ is a bijective map
\[r_{x,k}:[x^{(k)}]\times [x^{(k)}]\longrightarrow [x^{(k)}]\times [x^{(k)}],\]
so the kth retract class $([x^{(k)}], r_{x,k})$ is itself a
solution. Furthermore, $([x^{(k)}],r_{x,k})$ is a multipermutation
solution of level $\leq k$. In particular, whenever $[x]$ has
cardinality $\geq 2$, $([x], r_{x,1})$ is the trivial solution.
\item
\label{fact3}
 \cite{T04}, Lemma 8.12.
\begin{equation}
\label{facteq3} \alpha^{(2)}= \beta^{(2)}\Longrightarrow
{}^{{}^{\alpha}{x}} {\beta} = {}^x{\beta} \quad \forall x \in X,
\quad (\Lcal_{{}^{\alpha}{x}})_{\mid [\alpha^{(2)}]}=
(\Lcal_x)_{\mid [\alpha^{(2)}]}.
\end{equation}
\end{enumerate}
\end{facts}

\begin{remark} Note that  Lemma 8.9 in \cite{T04} states inaccurately
that $\mpl([x^{(k)}],r_{x,k}) = k$. The correct statement is
$\mpl([x^{(k)}],r_{x,k}) \leq k$.
\end{remark}

\begin{remark}
\label{trivialG(Y)remark2}
Suppose that $X_0= X \bigcap \ker \Lcal \neq \emptyset$ then $\Gcal(X_0)= 1$.
 Let $\zeta_0 \in X_0$.
Then $X_0= [\zeta_0] \subseteq [\zeta_0^{(k)}]$, for all $k \geq 1$.
Let $Y = [x^{(k)}]$ be a $k$-th retract class distinct from $[\zeta_0^{(k)}]$.
Then this class generates a nontrivial permutation group $\Gcal(Y)$.
It follows then that $\Gcal(Y)= 1$ is possible for at most one
k-th retract class, namely $Y = [\zeta_0^{(k)}]$, and this happens in the
particular case $Y=[\zeta_0^{(k)}]= X_0$.
\end{remark}

\begin{corollary}
\label{ret2abelianGcorollary}
Suppose $Y=[\zeta^{(2)}]$ is a second retract class in $X$.
Then the permutation group $\Gcal(Y)$ is an abelian subgroup of $\Gcal=\Gcal(X, r)$.
 $\Gcal(Y)=1$ \emph{iff} $Y \subset \ker \Lcal$.
\end{corollary}
\begin{proof}
Let $\alpha,\beta \in Y$. Then $\alpha^{(2)}= \beta^{(2)}$, and by (\ref{facteq2}) one has
\[
{}^{\alpha}{\beta}\sim {}^{\beta}{\beta}=\beta,
\]
or equivalently
\begin{equation}
\label{e1}
\Lcal_{({}^{\alpha}{\beta})}=\Lcal_{\beta}
\end{equation}
similarly
\begin{equation}
\label{e2}
\Lcal_{{\alpha}^{\beta}}=\Lcal_{\alpha}.
\end{equation}
So we have
\[
\begin{array}{rll}
\Lcal_{\alpha}\circ \Lcal_{\beta}& = \Lcal_{{}^{\alpha}{\beta}}\circ\Lcal_{{\alpha}^{\beta}} &\quad\quad \text{\textbf{lri}}\\
                                 & = \Lcal_{\beta} \circ \Lcal_{\alpha}  &\quad\quad \text{(\ref{e1}), (\ref{e2})},
\end{array}
\]
hence $\Gcal(Y)$ is abelian.
\end{proof}

\begin{proposition}
\label{vip_retrprop} The following conditions are equivalent:
\begin{enumerate}
\item
$\mpl(X,r)=m$.
\item
 For every $x\in X$ one has
\[X \supset[x^{(m-1)}]\supseteq
\Ocal_{\Gcal}(x),\]  where the left hand side inclusion is strict,
and $\Ocal_{\Gcal}(x)$ is the $\Gcal$-orbit of $x$.
\item
For every $x\in X$ the $m-1$ retract class $[x^{(m-1)}]$ is a
$\Gcal$-invariant proper subset of $X$.
\end{enumerate}
\end{proposition}

\begin{proof}
Note that $\mpl(X,r)=m$ \emph{iff } $\Ret^{m-1}$ is a trivial solution with at least
2 elements. Clearly,   $\Ret^{m-1}$ is a trivial solution  \emph{iff }
\[
{}^{(a^{(m-1)})}{(x^{(m-1)})}= x^{(m-1)} \quad \forall a, x \in X.
\]
The following are equalities in $\Ret^{m-1}$
\[
{}^{(a^{(m-1)})}{(x^{(m-1)})}=({}^{a}{x})^{(m-1)} \quad \forall a, x \in X
\]
\[
\begin{array}{rl}
\mpl(X,r)=m&\Longleftrightarrow  \Ret^{m-1} \quad \text{is a trivial solution of order}\; \geq 2\\
            &\Longleftrightarrow ({}^{a}{x})^{(m-1)}= x^{(m-1)}\quad \forall a, x \in X, [x^{(m-1)}] \subset X \\
            &\Longleftrightarrow  \Ocal_{\Gcal}(x)\subseteq [x^{(m-1)}] \subset X \quad \forall a, x \in X\\
            &\Longleftrightarrow  [x^{(m-1)}]\quad \text{is a
                                  $\Gcal$-invariant proper subset of $X$}.
\end{array}
\]
\end{proof}

\begin{corollary}
\label{intransitiveactionforinfiniteXCor} Let $(X,r)$ be a
square-free multipermutation solution of arbitrary cardinality.
Suppose  it is a multipermutation solution with $2 \leq mpl X = m <
\infty$. Then the number of $\Gcal$- orbits in $X$ is at least the
cardinality of $\Ret^{m-1}$. In particular, $\Gcal$ acts
intransitively on $X.$
\end{corollary}

\begin{proof}
The (m-1)st retract $\Ret^{m-1}$ is a trivial solution with at least
2 elements. By Proposition \ref{vip_retrprop} each $(m-1)$-retract
class $[x^{(m-1)}]$ contains the $\Gcal$-orbit of $x$. This proves
the corollary.
\end{proof}

\begin{theorem}
\label{mplmtheorem} Let $(X,r)$ be an arbitrary square-free solution,
not necessarily of finite  cardinality. Then

i)  $\mpl (X,r)\leq m$ if and only if the following equality holds
\begin{equation}
\label{mplmequality}
\begin{array}{rl}((\cdots((y_m\la y_{m-1})\la
y_{m-2})\la \cdots \la y_2)\la y_1)\la x& \\
&=((\cdots( y_{m-1}\la y_{m-2})\la \cdots \la y_2)\la y_1)\la x, \\
\forall x,\; y_1, \cdots y_m \in X.&
\end{array}
\end{equation}

ii) $\mpl (X,r)= m$ if and only if $m$ is the minimal integer for which
(\ref{mplmequality}) holds.
\end{theorem}

\begin{proof}
We use induction on $m$ to show the implication
\[
\mpl X \leq m \Longleftrightarrow (\ref{mplmequality}).
\]
Clearly, \[ \mpl X \leq 2
  \Longleftrightarrow[{}^zy] = [y], \forall y, z \in X \Longleftrightarrow
{}^{({}^zy)}{x}={}^{y}{x}, \forall x, y, z \in X.
  \]
This gives the base for the induction.

Assume the implications are true whenever $\mpl X \leq m$. Consider now the
retract $([X], r_{[X]}$. Clearly  $\mpl X = \mpl [X] +1$. Furthermore,
by the inductive assumption the following equality holds if and only if
$\mpl ([X], r_[X])\leq m.$
\begin{equation}
\label{mplPropeq1} ((\cdots(([y_{m+1}]\la [y_{m}])\la [y_{m-1}])\la
\cdots [y_3])\la [y_2])\la [y_1] = (\cdots(([y_{m}]\la
[y_{m-1}])\la \cdots [y_3])\la [y_2])\la [y_1]
\end{equation}
for all $y_1, \cdots,  y_m, y_{m+1} \in X$.

(Here we enumerate differently: we write $y_1$ instead of $x$,
etc.) By the obvious equalities
\begin{equation}
[a]\la [b] = {}^{[a]}{[b]} = [{}^{a}{b}]=[a\la b],
\end{equation}
(\ref{mplPropeq1}) is equivalent to
\begin{equation}
\label{mplPropeq2} [(\cdots((y_{m+1}\la y_{m})\la y_{m-1})\la \cdots
\la y_2)\la y_1] = [(\cdots(y_{m}\la y_{m-1})\la \cdots
\la y_2)\la y_1]
\end{equation}
for all $y_1, \cdots,  y_m, y_{m+1} \in X$.

But (\ref{mplPropeq2}) is equivalent to
\begin{equation}
\label{mplPropeq3} ((\cdots((y_{m+1}\la y_{m})\la y_{m-1})\la \cdots
\la y_2)\la y_1)\la x = (\cdots((y_{m}\la y_{m-1})\la
\cdots \la y_2)\la y_1)\la x
\end{equation}
for all $x, y_1, \cdots y_m, y_{m+1} \in X$.

This proves the equivalence $\mpl X \leq m+1 \Longleftrightarrow
\ref{mplPropeq3},$ which proves i).

(ii) follows straightforwardly from (i).

The theorem has been proved.
\end{proof}

\subsection{The groups $G(X,r)$ and $\Gcal(X,r)$}
\label{The groups subsection}
We recall more results.

\begin{facts}
\label{facts} Suppose $(X,r)$ is a square-free solution.
Then
\begin{enumerate}
\item \label{Gtorsionfree} $G= G(X,r)$ is
\emph{torsion free}. Let  $p$ be the least common multiple of the
orders of all permutations $\Lcal_x$ for $x \in X.$  Then the
following conditions hold.
\begin{itemize}
\item[(i)] $yx^p = (({}^yx)^p)y,\quad  x^py = y(x^y)^p \quad \forall x,y \in X.$ Thus the
group $G$ acts via conjugation on the set $X^{(p)}= \{x^p \mid x\in X\}$.
\item[(ii)]  $x^py^p = y^px^p,  \;\forall x,y \in X.$
The subgroup $A$ of $G$ generated by the set $X^{(p)}=\{x^p\mid x
\in X\}$ is isomorphic to the free abelian group in $n$
generators, (see \cite{T04, TM}).
\end{itemize}

\item It is known
that each set-theoretic solution $(X,r)$ of YBE (a braded set) can
be extended canonically to a solution $(S, r_S)$ on $S = S(X,r)$,
see \cite{TSh08} (respectively to a solution $(G, r_G),$ on $G=
G(X,r),$, \cite{Lu}) which makes $(S, r_S),$ a \emph{braided
monoid} (respectively $(G, r_G),$ a \emph{braided group}. In other
words the equality
\begin{equation}
\label{braidedgrouprel} uv= {}^uv.u^v,
\end{equation}
holds in for all $u,v$ in $S$ (respectively in $G$).
\end{enumerate}
\end{facts}

\begin{remark} It follows from the results of  \cite{TSh08} that
the extended solution $(G, r_G)$ satisfy
\begin{enumerate}
\item $(G, r_G)$  is involutive (i.e. $(G, r_G)$ is symmetric set)
if and only if $(X,r)$ is involutive.
\item $(G, r_G)$ is nondegenerate if and only if $(X,r)$ is nondegenerate
\item In particular, if $(X,r)$ is a square-free solution then $(G, r_G)$
is a nondegenerate symmetric set (but in general it is not
square-free). The notion of equivalence $u \sim v$
given by
\[u \sim v\Leftrightarrow(\forall g\in G)({}^ug = {}^vg)\]
is well defined,
and, as usual, $[u]$ denotes the equivalence class of $u$ in $G$.
\end{enumerate}
\end{remark}

\begin{lemma}
\label{kernelL} Let $G=G(X,r).$
 The kernel $K_0 =
\ker \Lcal$ of the group homomorphism $\Lcal :G \longrightarrow
\Sym(X)$ is a normal abelian subgroup of $G$ of finite index. In
particular, $K_0$ contains the free abelian subgroup $A
={}_{gr}[x_1^p, \cdots, x_n^p],$ where $p$ is the least common
multiple of all orders of permutations $\Lcal_x$, for $x \in X$.
\end{lemma}

 \begin{proof}
Clearly $u \in K_0$ if and only if $\Lcal_u = \id_X,$ and by {\bf
lri} the right action $\Rcal_u= (\Lcal_u)^{-1}=\id_X.$  This
straightforwardly implies
\begin{equation}
\label{rel1} u \in K_0 \Longleftrightarrow  {}^ua=a \quad \forall a
\in G \Longleftrightarrow a^u=a, \quad \forall a \in G
\end{equation}
Assume now $u,v \in  K_0$ Then  $uv=^{(\ref{braidedgrouprel})}
{}^uv.u^v=^{(\ref{rel1})}  vu,$ so $K_0$ is abelian. Clearly,
$\Lcal_{x^p}= (\Lcal_x)^p = \id_X,$ so $x^p \in K_0,$ for all $x\in
X,$ and therefore the free abelian group $A$  is contained in $K_0.$
\end{proof}
In assumption and conventions as above we introduce more notation.
\begin{notation}
\label{notation1}
\begin{enumerate}
\item
 $G_i= G(\Ret^i(X,r)),$ $G_0 =  G(X,r) =G$ ($\Ret^0(X,r)):= (X,r))$.
\item
$\Gcal_i= \Gcal(\Ret^i(X,r)).$
\item
$\Lcal^{0}= \Lcal:G(X,r) \longrightarrow \Gcal(X,r)$ is the usual
epimorphism extending the assignment $x \mapsto \Lcal_x, x \in X$
\\
$\Lcal^i= \Lcal:G_i \longrightarrow \Gcal_i$ is the canonical
epimorphism extending the assignment $x^{(i)} \mapsto
\Lcal_{x^{(i)}} \in \Sym (\Ret^i(X,r)),\; x \in X.$
\item
 $K_i$,
is the pull-back of $\ker \Lcal^i$ in $G,$ in particular $K_0= \ker
\Lcal.$
\item
$\mu_i: G_i  \longrightarrow G_{i+1}$ are the canonical epimorphisms
extending
$x^{(i)} \mapsto x^{(i+1)}$, where $0 \leq i < \mpl X,$ see Lemma \ref{rethomlemma} and Proposition \ref{longsequenceprop}.
\item
$N_i$ is the pull-back of $\ker \mu_i$ in $G.$
\item
$\varphi_i: \Gcal_i  \longrightarrow \Gcal_{i+1}$ is the canonical
epimorphism extending  the assignments
$\Lcal_{(x^{(i)})} \mapsto \Lcal_{x^{(i+1)}}, x\in X$, see Lemma \ref{rethomlemma} and Proposition \ref{longsequenceprop}.
\item
$H_i$ is the pull-back of $\ker \varphi_i$  in $G.$
\end{enumerate}
\end{notation}
Note that by definition, for $1\leq i\leq m-1$ one has
\[
K_1 =\{u \in S\mid \Lcal_{[u]} = \id _{[X]}\}, \quad K_i = \{u \in
S\mid \Lcal_{(u^{(i)})} = \id _{\Ret^i(X,r)}\}.
\]

\begin{lemma}
\label{rethomlemma} In assumption and notation as above
 the following conditions hold.
\begin{enumerate}
\item
\label{rethom1}
 The canonical epimorphism of solutions
\[
\mu: (X,r) \longrightarrow ([X], [r]);\quad x \mapsto [x],
\]
extends to a group epimorphism $\mu_0: G_0 \longrightarrow G_1.$
Analogously there exists a group epimorphism $ \mu_1: G_1
\longrightarrow G_2.$
\item
\label{rethom2} There is a canonical epimorphism
$\varphi_0: \Gcal_0  \longrightarrow \Gcal_1\quad\Lcal_x \mapsto
\Lcal_{[x]}, \; \forall x\in X.$
\item
\label{rethom4} The kernels $N_0 = \ker \mu_0,\; K_0 = \ker
\Lcal^{0},$ $H_0$- the pull back of $\ker \varphi_0$ into $G$, and
$K_1$, the pull back of $\ker \Lcal^1$ into $G$ satisfy
\begin{equation}
\label{kernel2}
\begin{array}{c}
N_0 \subset K_0 \subset K_1= H_0 \\ \\
 \ker \mu_1 \simeq N_1/N_0; \quad\quad  \ker \Lcal_1\simeq K_1/N_0; \quad\quad \ker \varphi_0 \simeq  K_1/K_0.
\end{array}
\end{equation}
\item
\label{rethom5} In particular,  $N_0$ is an abelian normal subgroup
of $G_0$, and there is a canonical epimorphism of groups
\[f_0: G_1 \longrightarrow \Gcal_0\quad [x] \mapsto \Lcal_{[x]}, \quad x\in X \quad \text{with}\; \ker f_0 \simeq K_0/ N_0.\]
\item
\label{rethom6} There are short exact sequences:
\begin{equation}
\label{vipexactsequences}
\begin{array}{rl}
1 \longrightarrow N_0 \longrightarrow G
\stackrel{\mu_0}{\longrightarrow} G_1 \longrightarrow 1\quad \quad& 1
\longrightarrow K_0 \longrightarrow G
\stackrel{\Lcal^{0}}{\longrightarrow} \Gcal \longrightarrow 1
\\
\\
1 \longrightarrow N_1/N_0 \longrightarrow G_1
\stackrel{\mu_1}{\longrightarrow}   G_2 \longrightarrow 1\quad \quad& 1
\longrightarrow K_0/N_0 \longrightarrow G_1
\stackrel{f_0}{\longrightarrow}  \Gcal \longrightarrow 1
\\ \\
1 \longrightarrow K_1/N_0 \longrightarrow G_1
\stackrel{\Lcal^1}{\longrightarrow} \Gcal_1 \longrightarrow 1 \quad \quad&
1 \longrightarrow K_1/K_0 \longrightarrow \Gcal
\stackrel{\varphi_0}{\longrightarrow} \Gcal_1 \longrightarrow 1,
\end{array}
\end{equation}

Moreover, the following diagram is commutative:

\begin{equation}
\label{vipcomdiagram}
\setlength{\unitlength}{0.7cm}
\begin{picture}(10,8)
\put(0,2){$1$} \put(0,4){$1$} \put(0.8,2){$\rightarrow$}
\put(0.8,4){$\longrightarrow$} \put(2,0){$1$} \put(1.5,2){$K_1/K_0$}
\put(2,4){$N_0$} \put(3,1){$\swarrow$} \put(3.2,2){$\rightarrow$}
\put(3,4){$\longrightarrow$} \put(3,5){$\nearrow$} \put(4,0){$1$}
\put(4,1){$\downarrow$} \put(4,2){$\mathcal{G}_0$}
\put(4,3){$\downarrow$}
\put(3.3,2.9){$\scriptstyle{\mathcal{L}^{0}}$} \put(4,4){$G_0$}
\put(4,5){$\downarrow$} \put(4,6){$K_0$} \put(4,7){$\downarrow$}
\put(4,8){$1$} \put(5,2){$\longrightarrow$}
\put(5.2,2.3){$\scriptstyle{\phi_0}$} \put(5,3){$\swarrow$}
\put(4.8,3.2){$\scriptstyle{f_0}$} \put(5,4){$\longrightarrow$}
\put(5.2,4.3){$\scriptstyle{\mu_0}$} \put(5,7){$K_1$}
\put(4.5,6.5){$\nearrow$} \put(5.5,6.5){$\searrow$} \put(6,0){$1$}
\put(6,1){$\downarrow$} \put(6,2){$\mathcal{G}_1$}
\put(6,3){$\downarrow$}
\put(6.3,2.9){$\scriptstyle{\mathcal{L}^{1}}$} \put(6,4){$G_1$}
\put(6,5){$\downarrow$} \put(5.8,6){$K_1/N_0$}
\put(6,7){$\downarrow$} \put(6,8){$1$} \put(7,2){$\longrightarrow$}
\put(7,4){$\longrightarrow$} \put(7,5){$\swarrow$} \put(8,2){$1$}
\put(8,4){$1$} \put(7.8,6){$K_0/N_0$} \put(9,7){$\swarrow$}
\put(10,8){$1$}
\end{picture}
\end{equation}
\end{enumerate}
\end{lemma}

\begin{proof}
Parts (\ref{rethom1}), (\ref{rethom2})
are clear.
We shall verify  (\ref{rethom4}).
Clearly, the kernel $N_0=\Ker \mu_G$ consists of all $a \in G,$
such that $[a]= 1_{[G]},$
hence it will be enough to show
 the  implication
\[
[a]= 1_{[G]} \Longrightarrow \Lcal_a = \id_X.
\]
Indeed, suppose $[a]= 1_{[G]}$. Then for an arbitrary  $x \in X$ one
has $[ax]=[a][x]=[x],$ so
\begin{equation}
\label{eq1} {}^xy = {}^{ax}y= {}^{a}{({}^xy)} \quad \forall y \in X.
\end{equation}
In particular, (\ref{eq1}) is true for $y=x,$ thus for an arbitrary
$x \in X$ one has
\begin{equation}
\label{eq2} x ={}^xx = {}^{ax}x=
{}^{a}{({}^xx)} = {}^ax,
\end{equation}
where the equality $x ={}^xx$ follows from our assumption $(X,r)$ square-free.
We have shown ${}^ax = x,$ for every $x\in X,$ thus $\Lcal _a=
\id_X.$ This verifies $N_0\subseteq K_0$. By
Lemma \ref{kernelL}  the group $K_0$ is  abelian, so is
$N_0$.

The equality $H_0= K_1 $ follows from the  implications:
\[
\begin{array}{cl}
u \in H_0 &\Longleftrightarrow \Lcal_{[u]} =
\id_{[X]}\\
&\Longleftrightarrow {}^{[u]}{[x]} = [{}^ux]=[x]\quad  \forall x \in
X\\
&\Longleftrightarrow {}^{{}^ux}z = {}^xz \quad \forall x,z \in X\\
&\Longleftrightarrow  u \in K_1.
\end{array}
\]
The inclusions (\ref{kernel2}) for the three kernels  are clear.
This implies the second line in   (\ref{kernel2}). The existence of
the short exact sequences
 (\ref{vipexactsequences}) is straightforward from  (\ref{kernel2}). One easily sees that the diagram (\ref{vipcomdiagram})
is commutative.
\end{proof}
We  discuss some basic differences between the two kernels $N_0$ and
$K_0$ below.
\begin{remark} Suppose $(X, r)$ is a nontrivial square-free solution
of finite order (so  $\mpl X \geq 2$).
 Then  $K_0$ is a normal subgroup of $G$
of finite index $[G:K_0 ]$,  and in contrast, the index $[G:N_0]$ of
$N_0$ is not finite. Furthermore, $A \subset K_0 $, but $A \cap
N_0= e.$   Indeed,  by hypothesis $(X,r)$ is a nontrivial solution
then, by Lemma \ref{mpl=1} the set $[X]$ has order $> 1.$ Furthermore
the retract
$([X], [r])$ is a braided set. Hence $[X]$ generates the group
$G_{1}=G([X], r_{[X]}),$ $[x] \neq 1_{G_{[X]}},\; \forall x \in X$.
The group $G_1$ is torsion free as a YB group of square-free
solution of order $> 1$, see Facts \ref{facts}, in particular,
$[x^p]=[x]^p \neq 1_{G_{[X]}},$  so $\forall x \in X, x^p$ is not in
$N_0.$ On the other hand we have shown in Lemma \ref{kernelL} that
$x^p \in K_0, \;\forall x \in X.$
\end{remark}
The following proposition is an iteration of Lemma
\ref{rethomlemma}

\begin{proposition}
\label{longsequenceprop} Let $(X,r)$ be a nontrivial square-free
solution. Suppose $\mpl(X,r) = m$. Then  the following conditions
hold.
\begin{enumerate}
\item
For all $j, 0 \leq j \leq m-1,$
there are canonical group epimorphisms
\begin{equation}
\begin{array}{rl}
\mu_j: G_j \longrightarrow G_{j+1} \quad& x^{(j)} \mapsto x^{(j+1)},
\\
\Lcal_j: G_j \longrightarrow \Gcal_j \quad& x^{(j)} \mapsto
\Lcal_{x^{(j)}}
\\
f_j: G_{j+1} \longrightarrow \Gcal_j \quad& x^{(j+1)} \mapsto
\Lcal_{x^{(j)}}
\\
\varphi_j: \Gcal_j \longrightarrow \Gcal_{j+1} \quad&
\Lcal_{x^{(j)}} \mapsto \Lcal_{x^{( j+1)}}.
\end{array}
\end{equation}
\item
\label{H0} For $0\leq j \leq m-1$ let $N_j$, (respectively, $K_j,
H_j$) be the pull-back in $G$ of the kernel $\ker\mu_j,$
(respectively, the pull-back of $\ker\Lcal_j$, $\ker \varphi_j$).
Then there are inclusions
\[\begin{array}{ccccccccccccc}
N_0 & \subset & N_1 & \subset & N_2 & \subset & \cdots & \subset & N_j & \subset & N_{j+1} & \subset & \cdots \\[5pt]
\bigcap & & \bigcap & & \bigcap & & & & \bigcap & & \bigcap & & \\[5pt]
K_0 & \subset & K_1 & \subset & K_2 & \subset & \cdots & \subset & K_j & \subset & K_{j+1} & \subset & \cdots \\[5pt]
& & \| & & \| & & & & \| & & \| & & \\[5pt]
& & H_0 & \subset & H_1 & \subset & \cdots & \subset & H_{j-1} & \subset & H_j & \subset & \cdots
\end{array}\]

and
\begin{equation}
\label{kernel3}
\begin{array}{rl}
\ker\mu_j \simeq N_j/ N_{j-1}, \quad \quad&
\ker \Lcal^j\simeq K_j/ N_{j-1} \\ \\
\ker f_j\simeq K_j/ N_j, \quad \quad& \ker \varphi_j \simeq K_{j+1}/ K_{j}.
\end{array}
\end{equation}
\item
\end{enumerate}
\end{proposition}

\begin{remark} Note that $\mpl(X,r) = m$ if and only if $H_{m-1}=G$.
\end{remark}

\begin{remark} By assumption $(X,r)$ is a square-free solution, thus {\bf lri} holds and the
graph $\Gamma(X,r)$ is well defined.  $\mu$ induces a homomorphism
of graphs
\[
\mu_{\Gamma}: \Gamma(X, r_X) \longrightarrow \Gamma([X], r_{[X]})
\]
The graph $\Gamma([X], r_{[X]})$ is a homomorphic image of
$\Gamma(X, r_X)$, though not in general a retraction.
\end{remark}

Recall that each solvable group $G$ has a canonical solvable series,
namely \emph{the derived series}
\[
G\supset G^{\prime}\supset G^{(2)}\supset\cdots \supset
G^{(s)}=1,
\]
where the derived subgroups $G^{(k)}$ are defined recursively.
$G^{\prime}$ is the commutator of $G$ (it is generated by the
comutators $[x,y]=xyx^{-1}y^{-1}, \; x,y \in G$) and for all $k\geq
1, G^{(k+1)}= (G^{(k)})^{\prime}$. Clearly, each $G^{k}$ is a normal
subgroup of $G.$ The length $s$ of the derived series is called
\emph{the solvable length} of $G,$ it is the minimal length of
solvable series for $G.$
We shall denote the solvable length of $G$ by $\sol(G)$.
The following fact is well known, and can be extracted, with a slight
modification of the proof, from \cite[Proposition 6.6]{Milne}.

\begin{fact}
\label{solvablelemma1} Let $N$ be a normal subgroup of $G$, and let
$\overline{G}=G/N.$ Suppose $N$ and  $G/N$ are solvable of solvable
lengths $m$ and $s,$ respectively. Then the solvable length $\sol(G)$ satisfies
$\max (m,s) \leq \sol(G) \leq m+s$.
\end{fact}


\begin{remark}
\label{solvableGremark}
It is known that the YB group $G(X,r)$ of every \emph{finite} nondegenerate
symmetric set  is solvable, see \cite{ESS}, Theorem 2.15. In
\cite{T04}, Theorem 7.10. is given  a different proof for the case
of finite square-free solutions.
\end{remark}
We shall prove that for each multipermutation (square-free)
solution $(X,r)$ \emph{of arbitrary cardinality}  the groups $G(X,r)$ and $\Gcal(X,r)$ are solvable,
and the solvable length of
$G(X,r)$ $\leq \mpl(X,r).$

\begin{proposition}
Let $(X,r)$ be a  square-free solution of
arbitrary cardinality, $G = G(X,r),  \Gcal = \Gcal(X,r).$ Suppose
$([X],r_{[X]}) = \Ret(X,r)$,
$G_1= G([X],r_{[X]})$. Then the following are equivalent
\begin{enumerate}
\item
$G$ is solvable.
\item
$\Gcal$ is solvable.
\item
$G_1$ is solvable.
\end{enumerate}
In this case the following inequalities hold:
 \begin{equation}
\label{ebasic}
 \sol (\Gcal)\leq \sol (G_1) \leq \sol (G) \leq \sol (\Gcal) +1.
 \end{equation}
In particular, if some of the retracts $\Ret^i(X,r)$ ($i \geq 0$ is a finite set,
then $G(X,r)$ is solvable.
\end{proposition}
\begin{proof}
 We know that there is a short exact sequence
\[
1 \longrightarrow K_0 \longrightarrow G \stackrel{\Lcal}{\longrightarrow}  \Gcal \longrightarrow 1,
\]
where  the kernel $K_0 = \ker \Lcal$ is an abelian  normal subgroup
of $G$, see Lemma \ref{kernelL}.  Fact \ref{solvablelemma1} implies
then that
\[\sol (G) \leq \sol (\Gcal) +1\]. By  Lemma \ref{rethomlemma}
there is a short exact sequence
\begin{equation}
\label{e5}
 1 \longrightarrow N_0 \longrightarrow G
\stackrel{\mu_0}{\longrightarrow} G_1 \longrightarrow 1,
\end{equation}
where the kernel $N_0$ of  $\mu_0$ is an abelian normal subgroup of
$G,$ so
\[
\sol (G_1) \leq \sol (G).
\]
By Lemma \ref{rethomlemma} $N_0\subset K_0$ and there is a short
exact sequence
\[
1 \longrightarrow K_0/N_0 \longrightarrow G_1 \longrightarrow  \Gcal \longrightarrow 1
\]
thus
\[
\sol (\Gcal) \leq \sol (G_1).
\]
We have  verified the inequalities (\ref{ebasic}).
Clearly this implies the equivalence of (1), (2), (3).

Assume now that for some $i$ the retract $\Ret^i(X,r)$ is of finite order.
Then by Remark \ref{solvableGremark} $G_i= G(\Ret^i(X,r))$ is solvable,
and therefore $G_{i-1}= G(\Ret^{i-1}(X,r))$ is solvable.
By decreasing induction on $i$ we deduce that $G_0=G(X,r)$ is solvable.
\end{proof}

\begin{theorem}
\label{sltheorem} Let $(X,r)$ be a  square-free solution of
arbitrary cardinality, $G = G(X,r),  \Gcal = \Gcal(X,r).$
Suppose  $(X,r)$  is a multipermutation solution with $\mpl(X,r)= m$.
Then $G$ and $\Gcal$ are solvable with
 \begin{equation}
 \label{slGeq}
 \sol (G) \leq m  \quad \text{and} \quad \sol (\Gcal)\leq m-1.
  \end{equation}
Furthermore,
 \[
\mpl(X,r) = 2 \Longrightarrow \sol (G) = 2  \quad\text{and}\quad
\sol (\Gcal) = 1.
\]
\end{theorem}

\begin{proof}
 We shall use induction on $m$ to show that $\sol(G) \leq m.$ Note
that the retraction $([X],[r])$ is a multipermutation square-free
solution of level $\mpl([X],r_{[X]})= m-1$. Base for the induction,
$m=1.$
Then $(X,r)$ is the trivial solution, $\Gcal ={e}$, and  by Lemma
\ref{mpl=1}, $G$ is abelian, so $\sol G = 1 = \mpl(X,r)$.

Suppose the statement is true for $m \leq m_0.$ Let $\mpl(X,r)=
m_0+1.$ Look at the short exact sequence (\ref{e5}). The retraction
$([X],r_{[X]})$ is a multipermutation square-free solution of level
$\mpl ([X],r_{[X]})= m_0,$ so by the inductive assumption the
solvable length of $G_1$ is at most $m_0.$ Clearly, the solvable
length of $N_0$ is exactly 1, hence by Fact \ref{solvablelemma1} the
solvable length of $G$ is at most $m_0+1.$

Using analogous argument one shows that $\sol\Gcal \leq m-1,$ this time we use the short exact sequence
\[
1 \longrightarrow K_1/K_0 \longrightarrow \Gcal \longrightarrow  \Gcal_1 \longrightarrow 1,
\]
where the kernel $K_1/K_0 $ is an abelian normal subgroup of
$\Gcal,$ see Lemma \ref{rethomlemma} again. (Here as usual $\Gcal_1
= \Gcal([X], r_{[X])})$. This verifies \ref{slGeq}.

Assume now that $\mpl(X,r)=2$. This implies that  $\Gcal(X,r)$ is abelian,
or equivalently $\sol \Gcal(X,r)=1$, see Theorem \ref{significantth}.
We have already shown  that $\sol G(X,r) \leq \mpl(X,r) (= 2).$
An assumption that there is  a strict inequality $\sol G(X,r) < 2$ would imply $G(X,r)$
is abelian, and therefore by Lemma
\ref{mpl=1} $\mpl(X,r)= 1,$ a contradiction.

The theorem has been proved.
\end{proof}

In the case when $(X,r)$ is of finite order we show that the
solvable lengths of $G$ and $\Gcal$ differ exactly with $1$, see
Theorem \ref{slG=slGcal+1thm}.

We need a preliminary lemma.

\begin{lemma}
\label{slG=slGcal+1lemma}
Let $A$ be a non-zero free abelian group of finite rank, and let $H$
be a non-trivial finite group acting faithfully on $A$. Let
\[[H,A]=\langle {}^ha-a:a\in A,h\in H\rangle.\]
Then $[H,A]$ is non-zero, and $H$ acts faithfully on $[H,A]$.
\end{lemma}

\begin{proof} We begin by observing that $H$ does indeed act on
$[H,A]$. If $k\in H$, then
\[{}^k({}^ha-a)={}^{khk^{-1}}b-b,\hbox{ where }b={}^ka\in A,\]
so ${}^k({}^ha-a)\in[H,A]$.

Let $\hat A=A\otimes\mathbb{Q}$. Then $\hat A$ is a vector space
over $\mathbb{Q}$, with dimension equal to the rank of $A$, and $H$
acts faithfully on $\hat A$. It suffices to prove the lemma with
$\hat A$ in place of $A$, since elements of $[H,A]$ are multiples of
elements of $[H,\hat A]$. The advantage is that \emph{Maschke's
Theorem} holds: if $B$ is an $H$-submodule of $\hat A$, then there
is a complement $C$, a $H$-submodule such that $\hat A=B\oplus C$
(in other words, $\hat A/B\cong C$).

Now $[H,\hat A]$ is the smallest $H$-submodule $B$ of $\hat A$ such
that $H$ acts trivially on $A/B$. So the complement guaranteed by
Maschke's Theorem is $C_H(\hat A)=\{a\in\hat A:a^h=a\}$. Since
$H\ne\{1\}$ and the action is faithful, $C_H(\hat A)\ne\hat A$, so
$[H,\hat A]\ne\{0\}$.

Finally, suppose that $h\in H$ acts trivially on $[H,\hat A]$. Since
also $h$ acts triviallly on $C_H(\hat A)$ by definition, it acts
trivially on the whole of $\hat A$; since we assume that $H$ acts
faithfully on $A$, we deduce that $h=1$. The lemma is proved.
\end{proof}

\begin{theorem}
\label{slG=slGcal+1thm} Let $(X,r)$ be a square-free solution of finite
order. Then \[\sol (G) = \sol (\Gcal) + 1.\]
\end{theorem}

\begin{proof} We know that there is a natural number $p$ such that
the subgroup of $G$ generated by the $p$th powers of the generators
is a free abelian group $A$. Clearly $A$ is isomorphic to the
integral permutation module for $\mathcal{G}$ (in its action on
$X$), so the action is faithful. (This uses the equation
$ba^pb^{-1}=({}^ba)^p$, which follows from $ba^p=({}^ba)^pb$, see Facts \ref{facts} 1. i.)

Let $A^{(n)}$ be defined inductively by $A^{(0)}=A$ and
\[A^{(n+1)}=[\mathcal{G}^{(n)},A^{(n)}]\]
for $n\geq 0$. By lemma \ref{slG=slGcal+1lemma} and induction, if
$\mathcal{G}^{(n)}\ne\{1\}$, then $A^{(n+1)}\ne\{0\}$ and
$\mathcal{G}^{(n)}$ acts faithfully on $A^{(n+1)}$. So, if $l=\sol G$,
then $A^{(l)}\neq\{0\}$. But $A^{(l)}\le G^{(l)}$ (the $l$th derived
group of $G$); so $\sol G>l$. By our previous observation, we know
that $\sol G\le l+1$; so in fact $\sol G=l+1$ holds, and the theorem is
proved.
\end{proof}
We know that $\mpl(X,r) = 2$ implies $\sol(G(X,r)) = 2.$ Example
\ref{gap-mpl-sl-lemma}  shows that  a gap between $\mpl(X)$ and
$\sol(G(X,r)$ can occur even for $\mpl(X,r) = 3.$

\begin{question}
Suppose $(X,r)$ is a  multipermutation square-free solution of
finite order $|X| > 1$. The following questions are closely
related.
\begin{enumerate}
\item
Suppose $\mpl(X,r) = m$. Can we express a  lower bound for $\sol
G(X,r)$ in terms of $m$ ?
\item
When  there is an  equality?
\[
\sol G(X,r) = \mpl(X,r) \quad\text{or equivalently}\quad \sol\Gcal(X,r)= \mpl(X,r)-1 \text{?}.
\]
\end{enumerate}
\end{question}

In  Section \ref{infinitesolutions} we construct   an infinite
sequence of explicitly defined solutions $(X_m, r_m),$  $m = 0, 1,
2 \cdots, $ such that $\mpl(X_m)= m$, and $m = \sol G(X, r_m) = \sol
\Gcal(X, r_m) + 1,$ see Definition \ref{beautifulconstructiondef}
and Theorem \ref{beautifulconstruction}.

\section{Wreath products of solutions}
\label{Wreath products}
In this section we define the notion of wreath product of solutions,
by analogy with the wreath product of permutation groups.

The following result is true for arbitrary braided sets (without
any further restrictions like being symmetric, finite, or
square-free).

\begin{theorem}\cite{TSh08}
 Let $(X,r_X),$ $(Y,r_Y)$ be disjoint solutions of the YBE, with
YB- groups $G_X=G(X,r_X),$ and $G_Y=G(Y,r_Y)$. Let $(Z,r)$ be a
regular YB-extension of $(X,r_X),$ $(Y,r_Y),$ with a  YB-group $G_Z=G(Z, r)$.
Then
\begin{itemize}
\item  $G_X, G_Y$ is a matched pair of groups with
actions induced
 from  the braided group $(G_Z, r)$.
 \item
  $G_Z$ is isomorphic to the double crossed product $G_X\bowtie
 G_Y$.
\end{itemize}
\end{theorem}

\begin{fact}
\label{propExt+}
[ESS] Let $(X,r_X)$ and $(Y,r_Y)$ be symmetric sets.
If  $Z\in Ext^{+}(X, Y)$, then $G_Z \simeq G_Y \rtimes G_X,$ where the semidirect product
is formed using the action of $G_Y$ on $X$ via $\alpha \rightarrow \Lcal_{\alpha}.$
\end{fact}
$Ext^{+}(X, Y)$ is the set of all symmetric sets  $(Z, r)$ which are extensions of
$(X,r_1), (Y, r_2)$ with $r(x,\alpha)= (\alpha,
x^{\alpha})$.

\begin{lemma}
\label{raisingmpl} Let $(X,r)$ be square-free multipermutation
solution, let $\tau \in \Aut (X,r)$ be an automorphism of $(X,r).$
Let $(Y, r_0)$ be the trivial solution on the one element set
$Y=\{\alpha\}$, where $\alpha$ is not in $X.$ Let $(Z,r_Z) =X\stu Y$
be the strong twisted union defined via $\Lcal_{\alpha}=\tau,$
$\Lcal_{x \mid Y}= \id_Y,$ for all $x \in X.$ (i.e. ${}^{\alpha}x=
\tau(x), {}^x{\alpha}=\alpha,$ for all $x \in X.$) Then
\begin{enumerate}
\item
$(Z, r_Z)$ is a square-free solution, so $Z = X\natural Y$.
\item
$G(Z, r_Z)$ is the  semidirect product $G(Z, r_Z) \simeq G(X,r)\rtimes
C_{\infty}$
\item Furthermore, if $(X,r)$ is a multipermutation solution of finite
multipermutation level, and  $\tau$ does not belong to  $\Gcal(X,r)$
then
\[\mpl(Z, r_Z)= \mpl(X, r)+1.\]
\end{enumerate}
\end{lemma}

The following proposition can be deduced from \cite{TSh08}.

\begin{proposition}
\label{propExt++} Let $(X,r_x), (Y, r_Y)$ be disjoint nondegenerate symmetric sets
(not necessarily finite or square-free). Assume there is an
injective map $Y \longrightarrow \Sym(X), \alpha \mapsto
\sigma_{\alpha}\in \Sym(X)$. Let $(Z,r)$ be the extension of
$(X,r_X)$, and $(Y, r_Y)$ where $Z = X\bigcup Y,$  $r$ extends $r_X$
and $r_Y$, and
\[
r(\alpha, x) = (\sigma_{\alpha}(x), \alpha ) \quad r(x, \alpha) =
(\alpha, \sigma_{\alpha}^{-1}(x))\quad \forall x \in X, \alpha \in Y.
\]
Then $(Z,r)$ is a symmetric set if and only if the assignment $\alpha
\mapsto \sigma_{\alpha}$ extends to a homomorphism $L_Y: G(Y, r_Y)
\longrightarrow \Aut (X,r_X).$
\end{proposition}

\begin{lemma}
\label{strtwistedunionlemma1} Suppose $(X, r_y), (Y,r_Y)$ are
disjoint symmetric sets (most general setting). Let $\sigma \in
\Sym X, \rho \in \Sym Y.$ Let $(Z,r)$, be an extension of $(X, r_X),
(Y,r_Y)$, such that
\[r(x, y) = (\rho(y), \sigma^{-1}(x)) \quad r(y,x) = (\sigma(x), \rho^{-1}(y))
\]
Then $(Z,r)$ is involutive, nondegenerate quadratic set.
Furthermore, $(Z,r)$ is a solution if and only if $\sigma \in \Aut X$ and
$\rho \in \Aut Y.$
\end{lemma}

We now define the wreath product of solutions:

\begin{definition}
\label{wreathconstructiondef} Suppose $(X_0, r_0)$, and $(Y,r_Y)$
are disjoint square-free solutions.
Let $\{(X_{\alpha}, r_{\alpha})\mid  \alpha \in Y  \}$ be the set of
$|Y|$  disjoint solutions $(X_{\alpha}, r_{\alpha})$ indexed by $Y$, where
each $(X_{\alpha}, r_{\alpha})$ is an isomorphic copy of  $(X_0,
r_0)$ defined on a set $X_{\alpha}= \{t_{\alpha, x}\mid  x \in
X_0\}$, and $\;t_{\alpha, x}$ denotes the copy of $x$ in $X_{\alpha}$
($t_{\alpha, x}\neq t_{\alpha, z}$ if and only if $x \neq z)$,  and the
map
 $r_{\alpha}$ and the associated actions are translations
of $r_0$ and its associated actions to $X_{\alpha}$. Thus
\[
r_{\alpha}(t_{\alpha, x}, t_{\alpha, z})= (t_{(\alpha, {}^xz)},
t_{(\alpha, x^z)})\quad \forall x,z \in X.
\]
Let $(X,r_X)=  \natural_{\alpha \in Y} X_{\alpha},$ be the trivial
extension of all $X_{\alpha}$, for $\alpha \in Y$; that is, for all
$\alpha, \beta \in Y$, we have
\[
\begin{array}{c}   r_{\mid X_{\alpha}\times X_{\alpha}} = r_{\alpha}
\\ \\
r(x,z) = (z,x), \; \forall x \in X_{\alpha}, z \in X_{\beta},\;
\text{with} \; \alpha \neq \beta \end{array}
\]
Define the map $Y \longrightarrow \Sym X, \quad\beta \mapsto
\sigma_{\beta},$ where the permutations $\sigma_{\beta}\in \Sym X$
are defined as follows
\begin{equation}
\label{sigma}
\begin{array}{c}
\sigma_{\beta}:X_{\alpha} \longrightarrow  X_{{}^{\beta}{\alpha}},\quad \alpha \in Y\\
\sigma_{\beta}(t_{\alpha,x}) = t_{({}^{\beta}{\alpha}), x}
\end{array}
\end{equation}
Let $(Z,r)$ be the extension of $(X,r_X)$ and $(Y, r_Y)$ defined as
follows
\[\begin{array}{rll}
Z&=& X \bigcup Y, \quad \\
r(\beta, (t_{\alpha,x})) &=& (\sigma_{\beta}(t_{\alpha,x}), \beta ) ,
\\
r((t_{\alpha,x}), \beta) &= &(\beta,
(\sigma_{\beta})^{-1}(t_{\alpha,x})),\\
\quad &&\forall \alpha, \beta \in Y,\quad
t_{\alpha,x}\in X_{\alpha}.
\end{array}
\]
 We call $(Z,r)$  a \emph{wreath product} of $(X_0,r_{X_0})$ and $(Y,r_Y)$,
 and denote it by $(Z,r)=(X_0,r_{X_0})\wr(Y,r_Y).$
\end{definition}

\begin{theorem}
\label{wreathconstructionth}
\begin{enumerate}
\item
The wreath product $(Z,r) = (X_0,r_{X_0})\wr(Y,r_Y)$ is a
square-free solution.
\item
\[
\Gcal(Z, r) = \Gcal(X_0,r_{X_0})\wr\Gcal(Y,r_Y).
\]
\item
Suppose  $(X_0,r_{X_0})$ and $(Y,r_Y)=$ are multipermutation
solutions of finite multipermutation level. Then
\[\mpl (Z,r) = \mpl (X_0,r_{X_0})+\mpl (Y,r_Y) - 1.\]
\end{enumerate}
\end{theorem}

\begin{proof}
Note that $(Z,r)$ satisfies the hypothesis of Proposition
\ref{propExt++}, hence $(Z,r)$ is a solution if and only if the map $Y
\longrightarrow \Sym X, \quad \beta \mapsto \sigma_{\beta},$ extends
to a homomorphism $L_Y: G(Y, r_Y) \longrightarrow \Aut (X,r_X).$

We show first that
\begin{equation}
\label{sigma0} \sigma_{\beta} \in \Aut (X,r_X)\quad \forall \beta \in
Y.
\end{equation}
By Lemma \ref{automorphismlemma} this is equivalent to
\begin{equation}
\label{sigma1} \sigma_{\beta} \circ\Lcal_{t_{\alpha, x}} =
(\Lcal_{\sigma_{\beta}( t_{\alpha, x})}) \circ\sigma_{\beta}\quad
\forall \alpha.
\end{equation}
Note that by the definition of $r$ the associated left action on $Z$
satisfies:
\begin{equation}
\label{sigma2}
\begin{array}{c}
\Lcal_{t_{\alpha, x}} (t_{\gamma, y}) =\begin{cases} t_{(\alpha, {}^xy)}  & \text{if}\quad \gamma = \alpha\\
                  t_{\gamma, y} & \text{else} \quad\quad \quad \end{cases}
\end{array}
\end{equation}
Let $t_{\gamma, y}\in X.$ We apply both sides of (\ref{sigma1}) on
$t_{\gamma, y}$, and compute  using (\ref{sigma2}) and
(\ref{sigma}). {\bf Case 1}. $\gamma = \alpha$.
\begin{equation}
\label{sigma3}
\begin{array}{rl}
\sigma_{\beta}\circ \Lcal_{t_{\alpha, x}}(t_{\alpha, y})&=\sigma_{\beta}(t_{\alpha, {}^xy})\\
                                                        &=t_{{}^{\beta}{\alpha}, {}^xy}
                                                       \\
\Lcal_{\sigma_{\beta}}( t_{\alpha, x})
\circ\sigma_{\beta}(t_{\alpha, y})
                                                         &=\Lcal_{ t_{({}^{\beta}{\alpha}, x)}}(t_{{}^{\beta}{\alpha}, y})
                                                         \\  \\
                                                        &=t_{({}^{\beta}{\alpha},{}^xy)},
 \end{array}
\end{equation}
as desired. {\bf Case 2}. $\gamma \neq \alpha$. In this case, due to
the nondegeneracy of $r_Y$  one has
\begin{equation}
\label{sigma4} {}^{\beta}{\alpha}\neq {}^{\beta}{\gamma},\quad
\forall \beta \in Y.
\end{equation}
Now our computation (applying (\ref{sigma2}), (\ref{sigma}), and
(\ref{sigma4})) give:
\begin{equation}
\label{sigma5}
\begin{array}{rl}
\sigma_{\beta}\circ \Lcal_{t_{\alpha, x}}(t_{\gamma, y})&=\sigma_{\beta}(t_{\gamma, y})\\
                                                        &=t_{{}^{\beta}{\gamma}, y}
                                                       \\
\Lcal_{\sigma_{\beta}( t_{\alpha, x})} \circ
\sigma_{\beta}(t_{\gamma,y})&=\Lcal_{(t_{{}^{\beta}{\alpha},
x})}(t_{{}^{\beta}{\gamma}, y})
                                                  \\
                                           &=t_{({}^{\beta}{\gamma}, y)}.
 \end{array}
\end{equation}
We have verified (\ref{sigma1}), therefore (\ref{sigma0}) holds.

Next we show that the map $L_0: Y \longrightarrow \Aut (X,r)\quad
\beta \mapsto \sigma_{\beta}$ extends to a homomporphism
\[
G(Y, r_Y) \longrightarrow \Aut (X,r).
\]
$Y$ generates $G(Y, r_Y)$, so it will be enough to show that  $L_0$
respects the relations of $G(Y, r_Y)$, or equivalently
\begin{equation}
\label{sigma5+}
\sigma_{\beta}\circ \sigma_{\gamma} =
\sigma_{{}^{\beta}{\gamma}}\circ \sigma_{{\beta}^{\gamma}},
\quad\forall \beta, \gamma \in Y.
\end{equation}
Recall that since $(Y, r_Y)$ is a solution, condition {\bf l1}
holds, that is
\[
{}^{\beta}{({}^{\gamma}{\alpha})}=
{}^{{}^{\beta}{\gamma}}{({}^{{\beta}^{\gamma}}{\alpha})}\quad
\forall \alpha, \beta,\gamma \in Y.
\]
Let $t \in X.$ Then  $t = t_{\alpha, x}$ for some $\alpha \in Y, x
\in X_0.$ We apply both sides of (\ref{sigma5+}) on $t$ and obtain:
\begin{equation}
\label{sigma6}
\begin{array}{rl}
\sigma_{\beta}\circ \sigma_{\gamma}(t_{\alpha, x})&=\sigma_{\beta}(t_{{}^{\gamma}{\alpha}, x})\\
                                                        &=t_{({}^{\beta}{({}^{\gamma}{\alpha})}, x)}
                                                       \\
                                                        &=t_{({}^{{}^{\beta}{\gamma}}{({}^{{\beta}^{\gamma}}{\alpha})}, x)}   \\
\sigma_{{}^{\beta}{\gamma}}\circ
\sigma_{{\beta}^{\gamma}}(t_{\alpha, x})
&=t_{({}^{{}^{\beta}{\gamma}}{({}^{{\beta}^{\gamma}}{\alpha})}, x)}
 \end{array}
\end{equation}
which verifies (\ref{sigma5}). We have shown that the sufficient
conditions for $(Z,r)$ being a solution are satisfied which proves
part (1) of the theorem.

(2) The wreath product of a group $G$ by a permutation group $H$ (acting on
the set $Y$) is defined to be the semidirect product of $N$
by $H$, where $N$ (the \emph{base group}) is the direct product of $|Y|$
copies of $G$, and $H$ acts on $N$ by permuting the direct factors in
the same way as it permutes the indexing elements of $Y$. If $G$ is itself
a permutation group on a set $X$, then $G\wr H$ is in a natural way a
permutation group on $X\times Y$ (regarded as the disjoint union of $|Y|$
copies of $X$).

It is clear
from our construction that $\Gcal(Z,r)$ is isomorphic as abstract group
to $\Gcal(X_0,r_{X_0})\wr\Gcal(Y,r_Y)$, and in fact acts on the union
of copies of $X_0$ by the natural permutation action of the wreath product;
it acts on  $Y$ according to the action of $\Gcal(Y,r_Y)$ (that is, the
base group is in the kernel of this action).

(3) Suppose now that $(X_0, r_0),$ and $(Y, r_Y)$ are
multipermutation solutions with $\mpl (X_0, r_0) = m$ and $\mpl (Y,
r_Y)= n.$ Every retract $\Ret^k(X,r)$ is a trivial extension of the
retracts $\Ret^k(X_{\alpha}, r_{\alpha}), \alpha \in Y.$ It follows
straightforwardly  that $\mpl(X,r_X)= \mpl(X_0, r_0)= m$

We study the retracts $\Ret^N(Z, r), N = 1, 2, \cdots.$ Note first
that by the definition of the map $r$ on $Z \times Z$ it follows
that
\begin{equation}
\label{sigma7}
\begin{array}{rl}
\Lcal_{x\mid X}= \Lcal_{y\mid X} & \Longleftrightarrow \Lcal_{x\mid Z}= \Lcal_{y\mid Z}\quad  x,y \in X\\ \\
\Lcal_{\alpha\mid Y}= \Lcal_{\beta \mid Y} & \Longleftrightarrow
\Lcal_{\alpha\mid  Z}= \Lcal_{\beta \mid Z}\quad  \alpha, \beta \in
Y.
\end{array}
\end{equation}
Hence the retract $\Ret(Z,r)$ can be viewed as  the extension of
$\Ret(X,r_X)$ and $\Ret(Y,r_Y),$ induced from the original actions
of $Y$ onto $X$, but reduced after collapsing all elements  $x \in
X$, with $\Lcal_x = \id_X$ and $\alpha \in Y$, with $\Lcal_{\alpha}
= \id_Y$ into a single point, say $[z_0].$ This does not have
effect on $\Ret (X,r_X),$ $\Ret (Y,r_Y),$ but now these solutions
might intersect in a single  joint point (with trivial action).
For each $k, 1 \leq k \leq m, \alpha \in Y,$  as usual, (see Notation \ref{notret0})
$\alpha^{(k)}$ denotes the equivalence class of $\alpha$ in $Ret^{k}(Z, r).$ Denote
$Y_k= \{\alpha^{(k)}\mid \alpha \in Y\}$. Note that $(Y_k, r_k)$
(with $r_k$ induced from $r_Y$ ) is isomorphic to $\Ret(Y,r_Y)= ([Y], r_{[Y]}$ for
all $k, 1 \leq k \leq m.$ Note also that each retract
$\Ret^k(Z,r),$ with $1 \leq k \leq m-1$ is  an union of
$\Ret^k(X,r_X)$ and $(Y_k, r_{Y_k})$ with possibly one joint point
$z_0^{(k)}$, the class of all $\xi \in Z$ which act trivially on
$\Ret^{k-1}(Z,r).$

Indeed every such a retraction has the effect of obtaining the
retractions $\Ret^k(X_{\alpha}, r_{\alpha}),$
where $\alpha \in Y,$ but
these are disjoint sets and therefore for each pair $\alpha, \beta
\in Y$ the inequality of equivalence classes in $Y$, $[\alpha] \neq
[\beta]$ implies $(\Lcal_{[\alpha]})_{\mid [X]} \neq
(\Lcal_{[\beta]})_{\mid [X]},$ and therefore
$(\Lcal_{[\alpha]})_{\mid [Z]} \neq (\Lcal_{[\beta]})_{\mid [Z]}.$

This way exactly on the $m$-th retraction  $\Ret^m(Z,r)$ all the
elements of  $X$ collapse in a single element $z_0^{(m)}$ (which could
be also the unique joint element with  $Y_m$). From now on the
$m+j$-th retraction $\Ret^{m+j}(Z,r)$ has effect of $j+1$-th
retraction on $(Y, r_Y)$ in the usual way, since as we  know $(Y_m,
r_m)$ is isomorphic to $\Ret(Y, r_Y),$ all element of $Y_m$ have
trivial action on $z_0^{(m)}$. It follows than that we need
exactly $n-2$ more retractions to obtain that $\Ret^{m+n-2}(Z,r)$ is
a trivial solution of order $\geq 2$, hence
$\Ret^{m+n-1}(Z,r)=\{z_0^{(m+n-1)}\},$ is a one element solution.
This verifies (3).
\end{proof}

\begin{openquestions}  Suppose the square-free solution $(Z, r)
= X\stu Y$ is a strong twisted union of $(X,r_X)$ and $(Y,
r_Y),$. Denote $G_Z= G(Z,r), G_X= G(X,r_X),G_Y=G(Y, r_Y)$

1) How are the groups $G_Z, G_X, G_Y$ related?

Proposition 4.6. \cite{TSh08}, shows that for arbitrary braided
set $(Z,r)$  which is an extension of two disjoint sets
$(X,r_X)$ and $(Y, r_Y),$  $G_X, G_Y$ is a matched pair of
groups and $G_Z$ is isomorphic to \emph{the double crossed product
}$G_X\bowtie G_Y.$

Note that in the case when $Z$ is a  strong twisted union of $X$
$Y$, the group $G_X$  acts on  $G_Y$ via automorphisms and $G_Y$, acts
on $G_X$ vis automorphisms, so we expect the structure of $G_Z
\simeq G_X\bowtie G_Y$ to be  more special.

2) How are the groups $\Gcal(Z,r), \Gcal(X,r_X), \Gcal(Y, r_Y)$
related?

3) Can we determine a upper bound for the solvable length of
$G(Z,r)$ in terms of the solvable lengths of $G_X, G_Y.$ Analogous
question for the solvable lengths of $\Gcal(Z,r), \Gcal_X, \Gcal_Y$

Moreover, suppose both $(X,r_X)$ and $(Y, r_Y),$ are
multipermutation solutions.

4) Is it true that $(Z,r)$ is always  a multipermutation solution?

5) How is $\mpl (Z)$ related to $\mpl(X), \mpl (Y)$?

Clearly $\max (\mpl X, \mpl Y)\leq \mpl (Z)$. Can we express an upper
bound for $\mpl (Z)$ in terms of $\mpl(X), \mpl (Y)$?

6) Suppose $(Z,r)$ is a finite  multipermutation square-free
solution, with $\mpl Z = m$. Is it true that $(Z,r)$ can be always
presented as a strong twisted union of $r$-invariant subsets
 of multipermutation level $< m$?
\end{openquestions}

\begin{lemma}
\label{jumpmpl} There exist two disjoint square-free symmetric sets
$(X,r_X), (Y, r_Y),$ with $\mpl(X)=\mpl(Y)=2,$ and a strong twisted
union $(Z, r)= X\stu Y,$ which is a multipermutation solution
with $\mpl(Z,r)= \mpl(X)+\mpl(Y)=4.$ Moreover, $\sol(G(Z,r)) = 4 = \sol(G_X)
+ \sol(G_Y)$, and $\sol(\Gcal(Z,r)) = 3$.
\end{lemma}

\begin{proof}
We shall define  $(X,r_X), (Y, r_Y)$ and $(Z,r)$ explicitly, via the
left actions. Let $(X, r_{X}), (Y, r_{Y})$ be the solutions defined
as
\[
\begin{array}{lll}
X=\bigcup_{1 \leq i\leq 4} X^{i} & X^{i}=\{x_1^{i}, x_2^{i},
x_3^{i}, x_4^{i}\}, &  1 \leq i\leq 4 \\
\Lcal_{x_1^{i}}=\Lcal_{x_3^i}=  (x_2^{i}x_4^{i})\quad&
\Lcal_{x_2^{i}}=\Lcal_{x_4^{i}} = (x_1^{i}x_3^{i}), &1 \leq i\leq 4.
\end{array}
\]
\[
\begin{array}{ll}
Y=\{a_1, a_2, a_3, a_4,  a_1^{\prime}, a_2^{\prime}, a_3^{\prime},
a_4^{\prime}, b,c \}, \quad& \Lcal_{a_i}= \Lcal_{a_i^{\prime}} =
\id_Y, \quad1 \leq i\leq 4,
\\
\Lcal_{b}=(a_1a_2)(a_3a_4)(a_1^{\prime} a_2^{\prime})(a_3^{\prime},
a_4^{\prime}), \quad& \Lcal_{c}=(a_1a_3)(a_2a_4)(a_1^{\prime}
a_3^{\prime})(a_2^{\prime}a_4^{\prime}).
\end{array}
\]
Note that $X$ is a trivial extension of the four isomorphic
solutions $(X^{i}, r_{i}), 1 \leq i \leq 4,$ with $r_i$ defined, as
usual,  via the left action. It is easy to see that  $\mpl X^{i} = 2$
Clearly, one has $\mpl X=2$. Direct computation shows that $\mpl Y
=2$. We set $Z= X \bigcup Y.$ In (\ref{exnewactions}) we
 define a left action  $Z\times Z\longrightarrow Z$ extending the
given actions $X\times X\longrightarrow X$ and $Y\times
Y\longrightarrow Y.$ (All permutations bellow are considered as
elements of $\Sym Z$).
\begin{equation}
\label{exnewactions}
\begin{array}{l}
\Lcal_{x_1^{i}}=\Lcal_{x_3^i}=(x_2^{i}x_4^{i})(a_ia_i^{\prime})
\quad
\Lcal_{x_2^{i}}=\Lcal_{x_4^{i}}= (x_1^{i}x_3^{i})(a_ia_i^{\prime}),\\
\Lcal_{a_i}=\Lcal_{a_i^{\prime}}=
(x_1^{i}x_2^{i}x_3^{i}x_4^{i}),\quad 1
\leq i \leq 4 ;\\
 \Lcal_{b}=(a_1a_2)(a_3a_4)(a_1^{\prime} a_2^{\prime})(a_3^{\prime}
a_4^{\prime})\prod_{1 \leq i \leq 4}[(x_i^{1}x_i^{2})
(x_i^{3}x_i^{4})],  \\
\Lcal_{c}=(a_1a_3)(a_2a_4)(a_1^{\prime}
a_3^{\prime})(a_2^{\prime}a_4^{\prime})\prod_{1 \leq i \leq
4}[(x_i^{1}x_i^{3}) (x_i^{2}x_i^{4})].
\end{array}
\end{equation}
Consider  the map $r: Z\times Z\longrightarrow Z\times Z,$  defined
as $r(z,t) = (\Lcal_{z}(t), \Lcal_{t}^{-1}(z))$ where  $t,z \in Z,
\Lcal_{z}, \Lcal_{t}$ as in (\ref{exnewactions}). One verifies
straightforwardly condition \textbf{l1}, hence $(Z, r)$ is a
symmetric set. Furthermore, condition  \textbf{stu } holds, so
 $(Z,r)= (X, r_X) \stu (Y, r_Y).$
Direct computation shows that $\mpl Z= 4$. Moreover, a short calculation
with GAP~\cite{GAP} shows that the group $\Gcal(Z,r)$ has order $2^{14}$
and solvable length~$3$. (This can easily be seen directly, since $\Gcal$ is
the wreath product of the dihedral group of order~$8$ with the Klein group
of order~$4$.) By Theorem~\ref{slG=slGcal+1thm}, $\sol G=4$.
\end{proof}

\section{Infinite solutions}
\label{infinitesolutions}

We consider solutions $(X,r)$ with $X$ infinite but having finite
multipermutation level.

\begin{question}
 In the hypothesis and notation of Lemma \ref{raisingmpl},
when can we  express $\Gcal(Z,r))$ as a wreath product
 $\Gcal(X,r)\wr\langle \tau ^{\prime}\rangle$
(where $\tau ^{\prime}$ is some appropriate permutation
deduced from $\tau$)? (This is not true in general.)
\end{question}

Note that, in general,   $ \Aut(X,r)\subseteq \Gcal(X,r),$   is
possible (so the existence of a $\tau \notin \Gcal(X,r) $ is not
automatic), as show the following example.

\begin{example}
Let $(X, r)$ be the three element nontrivial solution, with $X =
\{x_1, x_2, x_3\}$ and $r$ defined via the left actions
$\Lcal_{x_3}=(x_1x_2), \Lcal_{x_1}=\Lcal{x_2}= \id_{X}x$. Then
$\mpl(X,r) = 2$ and $\Aut(X,r)=\Gcal(X,r) = \{id_X, (x_1x_2)\}$,
\end{example}

\begin{example}
\label{easyconstruction} We shall construct an infinite  sequence of
explicitly defined square-free symmetric sets
\[
(X_0, r_0), (X_1, r_1), \cdots, (X_m, r_m), \cdots,
\]
such that

(i) for each $m$,  $m= 0,1,2, \cdots,\;$ $X_m \subset X_{m+1}$ is an
$r_{m+1}$-invariant subset of $X_{m+1}.$ Furthermore,

(ii) $(X_m, r_m)$ is a finite multipermutation  solution of order
$|X| = 2^{m+1}-1,$ and

(iii) $\mpl(X_m, r_m)= m$, the solvable length of each $\Gcal(X_m,
r_m)$ is exactly $m$.

(iv) $(X_{m+1}, r_{m+1}$ is a strong twisted union of a solution $(Y_m, r_m)$, with
$\mpl(Y_m)= m$,  $|Y_m| =2(2^m-1)$, and a trivial one-element solution.

As a starting point we chose an infinite  countable set
 $X=\{x_n\mid 1 \leq n \}$.
We define the solutions $(X_m, r_m), m=0, 1, 2, \cdots $
recursively.
\begin{itemize}
\item
$(X_0, r_0)$ is the one element trivial solution with  $X_0 =
\{x_1\}$.
\item
$(X_1, r_1)$ is the trivial solution on the set $X_1=\{x_1, x_2\}.$
\item
We set $(X_2, r_2)=X_1\stu \{x_3\},$ where
$\Lcal_{x_3}=(x_1x_2).$ Clearly, $\mpl X_2=2.$
\item
Construction of $(X_3, r_3).$ Let $(X_2^{\prime},r_2^{\prime})$ be
an isomorphic copy of $(X_2, r_2),$ where $X_2^{\prime}=\{x_4, x_5,
x_6\},$ and the map $\tau: (X_2, r_2) \longrightarrow
(X_2^{\prime},r_2^{\prime})$ with $\tau(x_i) = x_{i+3}, 1 \leq i
\leq 3,$  is an isomorphism of solutions. Let $(Y_2, r_{Y_2})=
X_2\stu _0 X_2^{\prime}$ be the trivial extension. We set $(X_3,
r_3) = Y_2\stu \{x_7\},$ where the map $r_3$ is defined via the
left action $\Lcal_{x_7}= (x_1x_4) (x_2x_5)(x_3x_6).$ One has
$\Lcal_{x_7}\in \Aut(Y_2, r_{Y_2})  \setminus  \Gcal{(Y_2, r_{Y_2})
},$ so $\mpl X_3$ =3.
\end{itemize}
Assume we have constructed the sequence $(X_0, r_0), (X_1, r_1),
\cdots, (X_m, r_m),$ satisfying conditions (i) and (ii). We shall
construct effectively $(X_{m+1}, r_{m+1})$ so that (i), (ii). and
(iii) are satisfied. For $N=2^{m}-1 = |X_m|$, let
$X_m^{\prime}= \{x_{N+1}, \cdots x_{2N},\}$ and let
$(X_m^{\prime},r_{X_m^{\prime}})$ be the solution isomorphic to
$(X_m, r_m)$ via the isomorphism $\tau: X_m\longrightarrow
X_m^{\prime}$ with $\tau(x_i) = x_{i+N}, 1 \leq i \leq N.$ We denote
by $(Y_m, r_{Y_m})$ the trivial extension $ X_m \stu _0
X_m^{\prime},$ and set $X_{m+1} = (Y_m,r_{Y_m})\stu
\{x_{2N+1}\},$ where $r_{m+1}$ is defined via the action 
$\Lcal_{x_{2N+1}}= (x_1x_{N+1})(x_2x_{N+2})\cdots (x_Nx_{2N}).$ One
can show that $\mpl(X_{m+1}, r_{m+1})= m+1,$ and $\Gcal(X_{m+1},
r_{m+1})= \Gcal(X_{m}, r_{m})\wr C_2.$
\end{example}

The following question was posed by Paul Martin.

\begin{question}
\label{n_m} For each positive integer $m$ denote by  ${\bf n}_m,$
the minimal integer so that there exists a square-free
multipermutation solution $(X_m, r_m)$ of order $|X_m| = {\bf
n}_m,$ and with $\mpl(X_m, r_m) = m$. How does ${\bf n}_m$ depend on $m$?
\end{question}
In the proof of Theorem \ref{beautifulconstruction} we construct
an infinite sequence of recursively defined explicit solutions
$(X_m, r_m), m = 0, 1, 2 \cdots, $ s.t. $\mpl(X_m)= m$, and
$|X_m| = 2^{m-1}+1.$ Therefore ${\bf n}_m \leq 2^{m-1}+1.$ By
definition, the unique symmetric set with multipermutation level
$0$ is the one element solution, so ${\bf n}_0=1.$ Direct
computation show that for $1 \leq m \leq 3$ the square-free
solutions $(X_m, r_m)$ in the Construction
\ref{beautifulconstruction}, are of minimal possible order, so
${\bf n}_m = 2^{m-1}+1, m=1, 2, 3.$

\begin{question}
Is it true that ${\bf n}_m = 2^{m-1}+1,$ for all integers $m \geq
1?$
\end{question}

Let $\X = \{x_i \mid  i= 1,2,3, \cdots \}$ be an infinite countable set.

\begin{definition}
\label{veedef}
Let $\rho =(x_{i_1}\; x_{i_2}\; \cdots x_{i_k})$ and
$\sigma = (y_{i_1}\; y_{i_2}\; \cdots y_{i_k})$ be disjoint cycles
of length $k$ in $\Sym(X)$, and let $N$ be a natural number.
Define
\[\begin{array}{rl}
\rho[N] &= (x_{ i_1+N}\; x_{i_2+N}\; \cdots \; x_{i_k+N}) \\ \\
\rho \vee \sigma &= (x_{i_1}\; y_{i_1}\; x_{i_2}\;y_{i_2}\; \cdots \;x_{i_k} \;y_{i_k}) \\ \\
\end{array}
\]
\end{definition}
So $\rho[N]$ is a cycle of the same length $k$ as $\rho$ and is  obtained by shifting the indices $N$-steps to the right.
$\rho \vee \sigma$ is a cycle of length $2k$ and
\[
(\rho \vee \sigma)^2 = \rho \circ \sigma.
\]

\begin{definition}
\label{Y[N]def}
For a pair $(x_i,x_j) \in \X \times \X$ we define $(x_i,x_j)[N] = (x_{i+N},x_{j+N}).$
Let $(Y,r)$ be a symmetric set, where $Y= \{x_1, x_2\cdots x_k \} \subset \X.$ For each integer $N$, $N> k$ we
define  the quadratic set $(Y,r)[N]) = (Y[N],r[N]),$
where $Y[N]= \{x_{1+N}, x_{2+N}\cdots x_{k+N} \},$ and
\begin{equation}
\label{r[N]}
r[N](x_{i+N}, x_{j+N})= (r(x_i, x_j))[N]
\end{equation}
\end{definition}

\begin{remark}
\label{Y[N]remarj=k}
Clearly the left action induced by $r[N]$ satisfies
\begin{equation}
\label{[N]actionseq} {}^{(x[N])}{(y[N])}=({}^xy)[N]
\end{equation}
So {\bf l1} is satisfied and therefore  $(Y[N], r[N])$ is a
square-free solution. There is an isomorphism of solutions
\[\varphi _N: (Y,r) \longrightarrow (Y[N], r[N])\quad  x_i\mapsto x_{i+N}.
\]
Furthermore,
 \[ \Gcal(Y[N],r[N])=\Gcal(Y, r)[N] =  {}_{gr}\langle \Lcal_x[N]\mid x \in Y\rangle.
\]
\end{remark}

\begin{definition}
\label{beautifulconstructiondef}
For $m= 0, 1, 2 \cdots $ define a sequence of cycles $\sigma_m \in \Sym \X$,  each of length $2^m$ as follows
\[
\begin{array}{cll}
\label{sigma_m}
\sigma_1 &=& (x_1\;x_2) \\
\sigma_2 &=&  \sigma_1\vee (\sigma_1[2]) = (x_1\; x_3\; x_2\; x_4)\\
\sigma_3 &=& \sigma_2\vee (\sigma_2[4])\\
\sigma_{m+1} &=& \sigma_m\vee (\sigma_m[2^{m}])
\end{array}
\]

For $m=0,1, 2, \cdots$ we define the solutions $(X_m, r_m)$ and $(Y_m, r_{Y_m})$ recursively, as follows.
\begin{equation}
\label{recursivedef2}
\begin{array}{clll}
(X_0, r_0)&&&\text{the trivial solution on one element set } X_0= \{x_1\}\\
(X_1,r_1)&&&\text{the trivial solution on the set } X_1= \{x_1, x_2\}\\
(Y_1, r_{Y_1})&=& (X_1,r_1) &\quad\text{(for completeness only)}\\
(X_2, r_2)&=& Y_1 \stu \{x_3\} &\;\text{where}\;\Lcal_{x_3} = \sigma_1\\
Y_2 &=& Y_1 \stu Y_1[2],&\;\text{where}\; \Lcal_{x_3} = \Lcal_{x_4}=
\sigma_1,\quad \Lcal_{x_1} = \Lcal_{x_2}= \sigma_1[2]\\
(X_3, r_3) &=& Y_2 \stu \{x_5\},& \;\text{where}\; \Lcal_{x_5} =   \sigma_2 = \sigma_1\vee (\sigma_1[2]) \\
Y_3 &=& Y_2 \stu Y_2[4],&\;\text{where}\;\Lcal_{y\mid Y_2} = \sigma_2, \;\;\forall  y \in Y_2[4],\; \Lcal_{x\mid Y_2[4]}=  \sigma_2[4], \;\; \forall x \in Y_2\\
(X_4, r_4) &=& Y_3 \stu \{x_9\},&\;\text{where}\; \Lcal_{x_9} =  \sigma_3 =\sigma_2\vee (\sigma_2[4]) \\
Y_4 &=& Y_3\stu Y_3[8],&\;\text{where}\; \Lcal_{y\mid Y_3} = \sigma_3, \; \;\forall y \in Y_3[8], \;\Lcal_{x\mid Y_3[8]}=  \sigma_3[8], \; \;\forall x \in Y_3\\
&&\ldots\ldots\ldots\ldots\ldots\ldots&\\
&&&\\
Y_m &=& Y_{m-1}\stu Y_{m-1}[2^{m-1}],&\;\text{where}\; \Lcal_{y}= \sigma_{m-1},\;\; \forall  y \in Y_{m-1}[2^{m-1}],\\
&&&\\
&&&
\text{ and }
\quad
 \Lcal_{x\mid Y_{m-1}[2^{m-1}]}=  \sigma_{m-1}[2^{m-1}],\; \;\forall x \in Y_{m-1}\\
&&&\\
X_{m+1} &=& Y_m \stu \{\xi\},&\;\text{where}\;  \xi = x_{2^m+1}, \; \Lcal_{\xi}=  \sigma_{m}=\sigma_{m-1}\vee (\sigma_{m-1}[2^{m-1}] ).
 \end{array}
\end{equation}
\end{definition}

Clearly $\X= \bigcup_{0 \leq m } X_m$.

\begin{definition}
Define the map \[
r_{\X}: \X\times \X \longrightarrow \X\times \X \quad r(x,y) = r_m(x,y)\text{ where } x, y \in X_m.
\]
\end{definition}

\begin{theorem}
\label{beautifulconstruction} In assumption and notation as above
Let
\[
(X_0, r_0), (X_1, r_1), \cdots, (X_m, r_m), \cdots
\]
be the infinite sequence of quadratic sets defined in Definition
\ref{beautifulconstructiondef}. Then  the following conditions hold
for each $m= 0,1,2,\ldots$:

\begin{enumerate}
\item \label{bc1} $(X_m, r_m)$ is a square-free solution of order
$|X_m| = 2^{m-1}+1$
 \item
\label{bc2}   $(X_{m+1}, r_{m+1})$ is an extension of $(X_m, r_m).$
\item
\label{bc3} $\Ret(X_{m+1}, r_{m+1}) \approx (X_m, r_m) .$ (i.e.
$(X_{m+1}, r_{m+1})$ is a blow up of $(X_m, r_m).$
\item
\label{bc4}  $\mpl(X_m, r_m)= m, m= 0,1,2 \cdots .$
\item
\label{bc5} Each group $\Gcal_{m+1}= \Gcal(X_{m+1}, r_{m+1})$ is
isomorphic to the wreath product $\Gcal(X_{m}, r_m)\wr C_2$ , so $
\Gcal_{m+1}= \underbrace {(({C_2}\wr {C_2})\wr \cdots )\wr
C_2}_{m\; \text{times}}.$
\item
\label{bc6} There are equalities $\sol (G_m)= m$, $\sol (\Gcal_m) = m-1$.
\item
\label{bc7} $(\X, r_{\X})$ is the inverse limit of the solutions
$(X_m, r_m)$. For the retracts one has $\Ret^m(\X, r_{\X})\neq
\Ret^{m+1}(\X, r_{\X})$, and  $\mpl(\X,r)= \infty.$ Furthermore, the
group $\Gcal(\X, r_{\X})$ acts nontransitively on $\X$.
\end{enumerate}
\end{theorem}

\begin{proof}
Under the hypothesis of the theorem we prove first several
preliminary statements.
\begin{remark}
\label{sigma_mlemma1} Let  $m$ be an integer, $m\geq 2$, $N_m = 2^{m}$
\begin{enumerate}
\item
$|Y_m| = 2^{m}$ and $|X_{m+1}| = 2^{m}+1.$
\item
$\sigma_m$ is well defined  via $\sigma_k,$ $k \leq m-1$ and the shift $y \mapsto y[N_{m-1}]$.
We have the following explicit formulae:
\begin{equation}
\label{sigma_meq0}
(\sigma_{m-1}[N_{m-1}])(y[N_{m-1}]) = (\sigma_{m-1}(y)[N_{m-1}]\quad \forall y\in Y_{m-1}\quad
\end{equation}
\begin{equation}
\label{sigma_meq1}
\begin{array}{cll}
\sigma_m (y) &=& y[N_{m-1}]\in Y_{m-1}[N_{m-1}], \quad\forall y\in Y_{m-1}\\
&&\\
 \sigma_m (y[N_{m-1}]) &=& \sigma_{m-1} y \in Y_{m-1}, \quad \forall y\in Y_{m-1}
\end{array}
\end{equation}
\begin{equation}
\label{sigma_meq2}
{}^{x[N_{m-1}]}{z[N_{m-1}]} = ({}^xz)[N_{m-1}]
\end{equation}
\end{enumerate}
\end{remark}

\begin{lemma}
\label{sigma_mlemma2} For each integer $m \geq 1$, $(Y_m,
r_{Y_m})$ is a square-free solution of order $N_m = 2^m$ and
$\sigma_m \in \Aut (Y_m, r_{Y_m})$,
\end{lemma}

\begin{proof}
We shall prove the lemma using induction on $m$. Base for the induction, $m=1$.
By definition $(Y_1, r_{Y_1})$ is the trivial solution $\{x_1, x_2\}$ and
clearly $\sigma_1 = (x_1 x_2) \in
\Aut (Y_1, r_{Y_1}).$
Assume now that the lemma is true for all $k \leq m$.

Clearly, by Definition  of as a set $Y_{m+1}$
is a disjoint union $Y_{m+1} = Y_m\bigcup Y_m[N_m]$, so
$|Y_{m+1}| = 2|Y_m| = 2^{m+1}$.
Furthermore (by definition), the quadratic set $(Y_{m+1}, r_{m+1})$
with  $r=r_{Y_{m+1}}$ satisfies
\[r(x, y) = (\sigma_m[N_m](y), \sigma_m^{-1}x) \quad r(y,x) = (\sigma_(x), \sigma_m[N_m]^{-1}(y) \quad \forall
x \in Y_{m}, y \in Y_{m}[N_m].
\]
By assumption $\sigma_m\in \Aut (Y_m, r_{Y_m}),$ hence $\sigma_m[N_m] \in \Aut (Y_m[N_m], r_{Y_m}[N_m]),$
Hence by Lemma \ref{strtwistedunionlemma1} $(Y_{m+1}, r_{m+1})$
is a solution.

We shall now prove that $\sigma_{m+1} \in \Aut Y_{m+1}.$
By Lemma \ref{automorphismlemma} 3) it will be enough to show  that for each $ x\in Y_m$
\begin{equation}
\label{autre1}
 \sigma_{m+1}\circ \Lcal_{x}  =\Lcal_{ \sigma_{m+1} (x)}\circ \sigma_{m+1}
\end{equation}
is an equality of maps in $Y_{m+1}$, or equivalently
\begin{equation}
\label{autre2}
 \sigma_{m+1} ({}^xz)  ={}^{\sigma_{m+1}  (x)}{\sigma_{m+1} (z)} \forall z \in Y_{m+1}.
\end{equation}

By the inductive assumption we have
\begin{equation}
\label{autre3}
 \sigma_{m} ({}^xz)  ={}^{\sigma_{m} (x)}{\sigma_{m}(z)} \forall x,z \in Y_{m}.
\end{equation}

Let $x, z \in Y_{m+1}.$ By definition $Y_{m+1} = Y_m \bigcup Y_{m}[N_m]$ (this is a disjoint union).

{\bf Case 1.} $x \in Y_{m}$.  {\bf 1.a.} $z\in Y_{m}$
\[
\begin{array}{ll}
{}^xz \in Y_{m}\quad \sigma_m ({}^xz) = {}^xz[N_{m}] & \text{by} \quad (\ref{sigma_meq1})\\&\\
\sigma_m (x)= x[N_{m}]\quad \sigma_m(z)= z[N_{m}] & \text{by} \quad (\ref{sigma_meq1})\\&\\
{}^{\sigma_m (x)}{\sigma_m(z)} = {}^{x[N_{m}]}{z[N_{m}]} = {}^xz[N_{m}] & \text{by} \quad (\ref{sigma_meq2})
\end{array}
\]
So in this case  (\ref{autre2}) holds.
{\bf 1.b.} $z=y[N_{m}],$ where $y \in Y_{m}.$
\[
\begin{array}{ll}
{}^xz = (\sigma_{m}[N_{m}])(z)=\sigma_{m}(y)[N_{m}]\in Y_{m}[N_{m}] & \text{by} \quad(\ref{recursivedef2}
\\&\\
\sigma_m ({}^xz) = \sigma_m (\sigma_{m}(y)[N_{m}])= (\sigma_{m})^2 (y) & \text{by} \quad(\ref{sigma_meq1})\\&\\
\sigma_m (x)= x[N_{m}]\quad \sigma_m(z)=\sigma_m(y[N_{m}]) = \sigma_{m}(y) & \text{by} \quad (\ref{sigma_meq1})\\&\\
{}^{\sigma_m (x)}{\sigma_m(z)} = {}^{x[N_{m}]}{\sigma_{m}(y)} = \sigma_{m}(\sigma_{m}(y)) & \text{by} \quad( (\ref{recursivedef2}).
\end{array}
\]
Hence (\ref{autre2}) holds.

{\bf Case 2.} $x= \xi[N_{m}] \in Y_{m}[N_{m}].$ {\bf 2.a.} $z\in Y_{m}$
In this case (\ref{autre2}) follows from the equalities
\[
\begin{array}{ll}
{}^xz = \sigma_{m}(z)=\in Y_{m} & \text{by} \quad(\ref{recursivedef2})\\&\\
\sigma_m ({}^xz) = \sigma_m (\sigma_{m}(z))= \sigma_{m}(z)[N_{m}] & \text{by} \quad(\ref{sigma_meq1})\\&\\
\sigma_m (x)= \sigma_m (\xi[N_{m}]) = \sigma_{m}(\xi)\in Y_{m}\quad \sigma_m(z)= z[N_{m}] \in Y_{m}[N_{m}]\quad & \text{by} \quad (\ref{sigma_meq1})\\&\\
{}^{\sigma_m (x)}{\sigma_m(z)} = {}^{\sigma_{m}(\xi)}{z[N_{m}]} = \sigma_{m}[N_{m}](z[N_{m}]) & \text{by} \quad (\ref{recursivedef2}) \\&\\
\quad \quad =(\sigma_{m}(z))[N_{m}] & \text{by} \quad( \ref{sigma_meq0}).
\end{array}
\]
{\bf 2.b.} $z=y[N_{m}], y \in Y_{m}.$
Note that
\[
\begin{array}{ll}
{}^xz  = {}^{\xi[N_{m}]}{y[N_{m}]}=({}^{\xi}y)[N_{m}] \in Y_{m}[N_{m}] \quad& \text{by}\quad (\ref{[N]actionseq})  \\&\\
\sigma_m ({}^xz) = \sigma_m (({}^{\xi}y)[N_{m}])   = \sigma_{m}({}^{\xi}y) \quad&
\text{by} \quad (\ref{sigma_meq1})\\&\\
\sigma_m (x)= \sigma_{m}(\xi)\in Y_{m}\quad \sigma_m(z)= \sigma_m(y[N_{m}])=  \sigma_{m}(y)\in Y_{m} \quad& \text{by} \quad (\ref{sigma_meq1})\\&\\
{}^{\sigma_m (x)}{\sigma_m(z)} = {}^{\sigma_{m}(\xi)}{\sigma_{m}(y)} =\sigma_{m}( {}^{\xi}{y} \quad& \text{by IH and } (\ref{autre2}).
\end{array}
\]
where IH denotes the inductive hypothesis.

We have shown that $\sigma_{m+1} \in \Aut Y_{m+1},$ which verifies the lemma.
\end{proof}
The following corollary is straightforward from the recursive definition of the quadratic sets $(X_m, r_m)$
and Lemma \ref{strtwistedunionlemma1}.

\begin{corollary}
\label{X_m_is_asolutioncor} For each $ m = 0, 1, 2, \cdots,\;$
$(X_m, r_m)$ is a square-free solution, of order $|X_m| = 2^{m-1}+1$.
Furthermore, $(X_{m+1}, r_{m+1})$ is an extension of $(X_{m}, r_{m}).$
\end{corollary}
The following lemma gives explicit recursive presentation of the left actions in $X_m,$ respectively $Y_m$

\begin{lemma}
\label{ret_Xmlemma}
Let $m \geq 2.$
\begin{enumerate}
\item
Let $x \in X_{m+1}$ then following equalities hold.
\begin{equation}
\label{ret_Xmeq1}
\begin{array}{c}
\forall x \in Y_{m-1}\quad \Lcal_{x\mid Y_m}=
\Lcal_{x\mid Y_{m-1}}\circ\sigma_{m-1}[N_{m-1}]\\ \\
 \Lcal_{x\mid X_{m+1}} = \Lcal_{x\mid Y_m} \\ \\
\forall y \in Y_{m-1}[N_{m-1}], \; y=x[N_{m-1}], x \in Y_{m-1}\quad
 \Lcal_{y\mid Y_m}= (\Lcal_{x\mid Y_{m-1}})[N_{m-1}]\circ\sigma_{m-1}\\
 \\ \Lcal_{y\mid X_{m+1}} = \Lcal_{y\mid Y_m} \\ \\
z=x_{2^m+1} \quad \Lcal{z\mid X_{m+1}} = \sigma_m.
\end{array}
\end{equation}
\item For all $i, i = 2k-1,  1 \leq k \leq 2^{m-1}$ one has
\begin{equation}
\label{ret_Xmeq2}
\begin{array}{c}
\Lcal_{x_{i}\mid X_{m+1}} = \Lcal_{x_{i+1}\mid X_{m+1}}, \\\\
\Lcal_{x_{i}} \neq \Lcal_{x_{j}}\quad\text{whenever}\quad j \neq i, i+1,\quad  1 \leq j \leq 2^m +1
\end{array}
\end{equation}
\item
There is an isomorphism of solutions
\[
\Ret(X_{m+1}, r_{m+1}) \simeq (X_m, r_m).
\]
\end{enumerate}
\end{lemma}

\begin{proof}
The equalities (\ref{ret_Xmeq1}) follow from the recursive definition of $Y_m$ and $X_m, m= 0,1,2, \cdots$
see Definition \ref{beautifulconstructiondef},  (\ref{recursivedef2}).
The recursive definition and (\ref{ret_Xmeq1}) imply (\ref{ret_Xmeq2}).
Then \[[X_{m+1}] =\{[x_1], [x_3],\cdots [x_{2^{m-2}+1}], \cdots [x_{2^{m-1}-1}]\}.\] It is easy to see
that  $([X_{m+1}], r_{[X_{m+1}]}) \simeq (X_m, r_m)$.
\end{proof}
Now the statement of the theorem follows easily. Indeed, Corollary \ref{X_m_is_asolutioncor} verifies parts (\ref{bc1}),  and
(\ref{bc2}). Part (\ref{bc3}) follows from Lemma \ref{ret_Xmlemma}.
Clearly, induction on $m$ and (\ref{bc3}) straightforwardly imply
(\ref{bc4}).
(\ref{bc5})  is clear from the construction.
(\ref{bc5}) implies $\sol(\Gcal_m) = m-1$. Hence, by Theorem
\ref{slG=slGcal+1thm} we have $\sol(G_m) = \sol(\Gcal_m) +1$, which proves
(\ref{bc6}).
Finally, (\ref{bc7}) is clear.
\end{proof}

\begin{construction} Let $R$ be a finite ring and $A$ a finite
faithful $R$-module. Let $\omega$ be a fixed unit in $R$. For $a\in
A$, let $L_a$ be the permutation $x\mapsto \omega x+(1-\omega)a$.
\end{construction}

Note: We do not need to assume that $R$ is commutative, since we
work only in the subring generated by $1$ and $\omega$. Also, note
that $L_a$ has the effect of multiplying $x-a$ by $\omega$; so it is
clearly a permutation fixing $a$. (Its inverse is obtained by
replacing $\omega$ by its inverse.) If $\omega=-1$, it is the
inversion in $a$.

\begin{proposition}
The following are equivalent:
\begin{itemize}
\item[(a)] the maps $L_a$, for $a\in A$, give a solution;
\item[(b)] the maps $L_a$, for $a\in A$, all commute;
\item[(c)] $(1-\omega)^2=0$.
\end{itemize}
\end{proposition}

\begin{proof} Clearly $L_a(a)=a$ holds for all $a\in A$. Thus, the
maps form a solution if and only if
\[L_{{}^ab}L_{a^b}=L_aL_b,\]
where ${}^ab=L_a(b)$ and $a^b=L_b^{-1}(a)$; and a short calculation
(see below) shows that this condition and the condition
$LaL_b=L_bL_a$ hold for all $a$ and $b$ if and only if
$(1-\omega)^2=0$.

\bigskip

If $\omega=1$, then $L_a$ is the identity map for all $a$, and the
multipermutation level is $1$. Otherwise, $L_a=L_b$ if and only if
$(1-\omega)(a-b)=0$; so the elements of the reduct are the cosets of
the submodule $\{a\in A:(1-\omega)a=0$ (the annihilator of
$1-\omega$. Since $(1-\omega)^2=0$, the element $(1-\omega)$ acts as
zero on the quotient module $A/I$, so the reduct has mpl~$1$, and
the original solution has mpl~$2$.

\bigskip

\paragraph{Calculations}
$L_a$ maps $x$ to $\omega x+(1-\omega)a$, and so
\begin{eqnarray*}
L_aL_b &\mapsto& \omega x+(1-\omega)b \\
       &\mapsto& \omega(\omega x+(1-\omega)b)+(1-\omega)a \\
       &=& \omega^2x+(1-\omega)(a+\omega b).
\end{eqnarray*}
Similarly $L_bL_a$ maps $x$ to $\omega^2x+(1-\omega)(b+\omega a)$.
These are equal if and only if $(1-\omega)^2(a-b)=0$. For this to
hold for all $a$ and $b$, it is necessary and sufficient that
$(1-\omega)^2=0$.

\medskip

Now $a^b=\omega^{-1}a+(1-\omega^{-1})b$ and ${}^ab=\omega
b+(1-\omega)a$, so we calculate that $L_{{}^ab}L_{a^b}$ maps $x$ to
\[\omega^2x+(1-\omega)(2-\omega)a+(1-\omega)(2\omega-1)b.\]
This is equal to $L_aL_b(x)$ if and only if $(1-\omega)^2(a-b)=0$,
and the conclusion follows as before. \end{proof}

\section{More about YB permutation groups}

Let $(Z, r)$ be an YB extension of the disjoint solutions
$(X,r_X)$ and $(Y, r_Y)$. Then for every $z \in Z$ the action
$\Lcal_z$ splits
\[
\Lcal_z = (\Lcal_{z})_{\mid X}\circ (\Lcal_{z})_{\mid XY}.
\]

Recall that if $(X,r_X)$ and $(Y, r_Y)$ are two disjoint solutions
Then the trivial extension $(Z, r)$ is defined as $Z = X \bigcup
Y$, with $r(x, \alpha)= (\alpha, x),$ for all $x \in X, \alpha \in Y$.

\begin{proposition}
\label{PYBextprop1} Let $(X,r_X)$ and $(Y, r_Y)$ be disjoint
square-free solutions with $G_1 = G(X,r_X), G_2=(Y, r_Y),$
$\Gcal_1 = \Gcal(X,r_X),\Gcal_2 = \Gcal(Y,r_Y).$ Then the
following conditions hold.
\begin{enumerate}
\item Suppose $(Z,r)$ is the trivial extension of $(X,r_X)$ and
$(Y, r_Y)$. Then it is a square-free solution, and
\[
G(Z,r)= G_1 \times G_2 \quad \Gcal(Z,r) = \Gcal_1 \times \Gcal_2
\]
Conversely, if $G_1$ and $G_2$ ate permutation YB groups, then $G=
G_1 \times G_2$ is a permutation YB group.
 \item
 \label{semidirectpr1} Suppose $(Z,r)$
is an involutive extension which satisfies
\begin{equation}
\label{ext+} r(\alpha, x) = ({}^{\alpha}x, \alpha)\quad \forall x
\in X, \alpha \in Y. \end{equation} Then $(Z,r)$ is a solution
if and only if the assignment $\alpha \mapsto (\Lcal_{\alpha})_{\mid
X}$ extends to a homomorphism
\[
G(Y,r_Y) \longrightarrow \Aut (X, r_X)
\]
In other words, $G(Y,r_Y)$ acts as automorphisms on $(X,r).$ In this
case $G= G(Z,r)$ and the YB permutation  group $\Gcal= \Gcal(Z,r)$
are  semidirect products: \[ G= G_1 \rtimes G_2, \quad \Gcal=
\Gcal_1 \rtimes  \Gcal_2.
\]
\item
 Conversely, suppose there is an action of $G(Y,r_Y)$ on
$G(X, r_X),$ such that $X$ is invariant under this action. Then the
formula (\ref{ext+} )  induces canonically a solution $(Z,r)$ on
$Z$. In particular, a semidirect product $\Gcal_1 \rtimes\Gcal_2$ of
two permutation YB groups defined via an action of $G(Y,r_Y)$ on
$G(X, r_X),$ which keeps $X$ invariant is itself a permutation YB
group.
 \item Suppose $(Z,r)$ is the wreath product of
solutions $(Z,r)= (X,r_X)\wr(Y,r_Y),$ see Definition
\ref{wreathconstructiondef}. Then there are equalities
\[
G(Z,r) ={G_1}\wr G_2\quad \Gcal(Z,r) ={\Gcal_1}\wr\Gcal_2
\]
\end{enumerate}
\end{proposition}

\begin{question}
\label{wreathproductPYBquestion} Under what conditions is the wreath
product ${G_1}\wr G_2$ of two permutation YB groups $G_1$ and
$G_2$ ismorphic to a permutation YB group?
\end{question}

\subsection*{Acknowledgements}  The work was completed while the  authors
were attending the Isaac
Newton Institute Programme on Combinatorics and Statistical
Mechanics (CSM) 2008. They thank the Isaac Newton Institute for
local support and for the inspiring working atmosphere. The second
author also  thanks the ICTP in Trieste, and the University of
Granada, for support during some stages of the project.

\end{document}